\documentclass[a4paper]{amsart}

\usepackage{amsmath}
\usepackage{amssymb}
\usepackage{amsfonts}
\usepackage{hyperref}

\usepackage[utf8]{inputenc} 

\usepackage{graphicx}

\usepackage{enumerate}  
\usepackage{mathtools}  
\usepackage{bm}         
\usepackage{url}        
\numberwithin{equation}{section}

\usepackage{tikz}

\makeatletter
\def\@listi{\leftmargin\leftmargini
              \parsep 0\p@ \@plus2\p@
              \topsep 2\p@ \@plus1\p@
              \itemsep 2\p@ \@plus1\p@}
\let\@listI\@listi
\@listi
\makeatother

\newtheorem{theorem}{Theorem}[section]
\newtheorem{lemma}[theorem]{Lemma}
\newtheorem{proposition}[theorem]{Proposition}

\newtheorem{corollary}[theorem]{Corollary}

\theoremstyle{definition}
\newtheorem{definition}[theorem]{Definition}

\newtheorem{example}[theorem]{Example}

\newtheorem{remark}[theorem]{Remark}

\newcommand{\la}{\langle}
\newcommand{\ra}{\rangle}

\newcommand{\Al}[1][A]{\ensuremath{\mathbf{#1}} }

\newcommand{\X}{X} 

\newcommand{\A}{\Al}

 \newcommand{\DL}{\mathsf{D}}

\newcommand{\vocab}[1]{\emph{#1}}                  
\newcommand{\class}[1]{\mathsf{#1}}                

 \newcommand{\imp}{\mathbin\to}

 \let\Omega=\varOmega

 \let\Gamma=\varGamma

 \let\Lambda=\varLambda

 \let\Sigma=\varSigma

 \let\Psi=\varPsi

 \let\Delta=\varDelta

 \let\Pi=\varPi

 \let\Theta=\varTheta

\def\zseq#1{\lseq #1 \rseq}                        
\def\zseq#1{\lseq #1 \rseq}                        
\let \lseq \langle                                 
\let \rseq \rangle                                 
\def\zset#1{\{ #1 \}}                              

\newcommand{\pr}{\mathop{\preceq}}

\newcommand{\nnot}{\mathop{\sim}}                  

\renewcommand{\nu}{\Diamond}

\let\imp\Rightarrow                                
\def\zseq#1{\lseq #1 \rseq}                        
\newcommand{\alg}[1]{\mathbf{#1}}                  
\def\zseq#1{\lseq #1 \rseq}                        
\let \lseq \langle                                 
\let \rseq \rangle                                 

 \let\equ==

 \let\implies=\Rightarrow

\newcommand{\eval}[2][\right]{\relax
   \ifx#1\right\relax \left.\fi#2#1\rvert}

\let\abs=\envert

\renewcommand{\preceq}{\preccurlyeq}

\let\imp\Rightarrow                                
\let\phi\varphi
\newcommand{\Ll}[1][A]{\ensuremath{\mathbf{#1}} }
\newcommand{\PN}{\mathbf{N4}}
\newcommand{\N}{\mathbf{N3}}
\newcommand{\logic}[1]{\mathbf{#1}}                
\newcommand{\ie}{\textit{i.e.\/}}                  
\newcommand{\cm}{\mathrel{\Xi}}                    
 \newcommand{\PS}{{\mathsf{PrSp}}} 
 \newcommand{\DM}{{\mathsf{DM}}}
 \newcommand{\Tw}{{\mathsf{Tw}}}
 \newcommand{\NF}{{\mathsf{N4}}}
  \newcommand{\DMS}{{\mathsf{DMSp}}}
    \newcommand{\NFS}{{\mathsf{N4Sp}}}
        \newcommand{\NE}{{\mathsf{NESp}}}
\def\ba{\backslash}
\let\lmb\bigl                                      
\let\rmb\bigr                                      
\newcommand{\UP}{\blacktriangle}                              
\newcommand{\DOWN}{\blacktriangledown}                        
\newcommand{\Up}{\vartriangle}
\newcommand{\Down}{\triangledown}

\usepackage{enumitem}  

\begin{document}

\title{Nelson algebras, residuated lattices and rough sets: A~survey}

\author{Jouni J\"{a}rvinen}
\address[J.~J{\"a}rvinen]{Software Engineering, LUT School of Engineering Science, Mukkulankatu~19, 15210 Lahti, Finland}
\email{jouni.jarvinen@lut.fi}

\author{S\'{a}ndor Radeleczki}
\address[S.~Radeleczki]{Institute of Mathematics, University of Miskolc, 3515 Miskolc-Egyetemv\'{a}ros, Hungary}
\email{matradi@uni-miskolc.hu}

\author{Umberto Rivieccio}
\address[U.~Rivieccio]{Departamento de L\'ogica e Historia y Filosof\'ia de la Ciencia,  
UNED, Madrid, Spain}
\email{umberto@fsof.uned.es}
%

\date{\today}

\vspace*{-10mm}

\maketitle

\begin{abstract}
Over the past 50 years, Nelson algebras have been extensively studied by distinguished scholars as the algebraic counterpart of Nelson's constructive logic with strong negation. Despite these studies, a comprehensive survey of the topic is currently lacking, and the theory of Nelson algebras remains largely unknown to most logicians. This paper aims to fill this gap by focussing on the essential developments in the field over the past two decades. Additionally, we  explore generalisations of Nelson algebras, such as N4-lattices which correspond to the paraconsistent version of Nelson's logic, as well as their applications to other areas of interest to logicians, such as duality and rough set theory.  A general representation theorem states that each Nelson algebra is isomorphic to a subalgebra of a rough set-based Nelson algebra 
induced by a
quasiorder. Furthermore, a formula  is a theorem of  Nelson logic if and only if
it is valid in every finite 
Nelson algebra induced by a quasiorder.
\end{abstract}

\section{Introduction}

In 1949, David Nelson introduced \emph{constructive logic with strong negation}, often called simply  \emph{Nelson logic}. The introduction of Nelson logic was part of a larger effort to provide a constructive account of mathematical reasoning, as evidenced by Nelson's later papers such as \cite{Almu84,Nels59}. While the intuitionists focused on the analysis of the notion of truth, Nelson's proposal sought to provide a formal framework for a constructive analysis of the notion of falsity.

Subsequent research since the 1950s has shown that the theory of Nelson logic 
is interesting beyond its original motivations. Helena Rasiowa was among the first  scholars to take an interest in the topic, initiating the investigation 
of algebraic models of Nelson logic, which are now known as $\N$-lattices or \emph{Nelson algebras}. More recently, following in Rasiowa's footsteps, Sergei Odintsov characterised the algebraic models of the paraconsistent weakening of Nelson logic, which were introduced in \cite{Almu84} and later named as $\PN$-lattices \cite{Odin03,Odin04}. From the study of these classes of algebras, a rich structure theory has emerged, and interesting connections with other algebraic models of non-classical logics have been discovered.

While partial overviews of Nelson logic and its algebraic models can be found in the paper~\cite{Vaka06} and the second part of the book~\cite{Odin08}, no comprehensive survey currently exists. We recommend these works to the reader for further historical background. The present work focuses on the recent developments of the last two decades.
  
Since the turn of the twenty-first century, several noteworthy results have emerged in this area that we believe should be consolidated and reviewed in a single publication. Specifically, we are referring to the following:

\begin{itemize}
\item the extension of the theory of Nelson algebras to the paraconsistent setting of {$\PN$-lattices};

\item recent developments on the representation of $\N$- and $\PN$-lattices, as well as related algebras such as bilattices, using twist structures;

\item novel views of Nelson logic as a substructural logic and of paraconsistent Nelson logic as a relevance logic, both based on term equivalence results between the corresponding classes of algebras;
 
\item the connection between $\N$-lattices and the theory of \emph{rough sets}.
\end{itemize}

Due to space and time limitations, it will not be possible to cover all the recent interesting developments in this survey. However, we would like to mention some noteworthy ones, including:
\begin{itemize}

\item  abstract, universal algebraic characterisations of Nelson algebras~\cite{Nasc1z};

\item investigations of closely related logico--algebraic systems, such as the substructural logic introduced
by Nelson under the name of $\mathcal{S}$~\cite{Nasc1y,Nasc1x};  

\item the extension of the theory of Nelson algebras beyond  the involutive setting~\cite{thiago2021negation,riviecciotwonegs,rivieccio2021fragments,rivieccio2021quasi,riviecciosp2021quasi}.
\end{itemize}

In creating this contribution, we aimed to balance two competing demands. On the one hand, we sought to produce a survey that would provide specialists with quick updates on the latest developments in the field. On the other hand, we aimed to make the content accessible to a wider community, including experts in algebra and logic who may not be familiar with Nelson logic and algebras. We hope to have fulfilled both of these goals, but especially the latter. We believe that bridging the gap between researchers in algebra and logic and experts in Nelson-like logics and algebras could be highly beneficial for the community, given the exciting contemporary developments referenced 
above.

The rest of the  paper is organised as follows. 
Section~\ref{sec:n4} provides an introduction to Nelson logic~\cite{Nels49} and its paraconsistent counterpart~\cite{Almu84}. While we briefly touch upon the former, we mainly focus on the latter and its algebraic models, known as $\PN$-lattices. The decision to present the material in this order, contrary to the historical development of the field, is motivated by pedagogical reasons. In fact, results concerning $\PN$-lattices readily specialise to $\N$-lattices, which we discuss in Section~\ref{sec:n3}.

After covering the fundamental theory, we introduce the twist-structure representation of $\PN$-lattices and discuss its key applications in Subsection~\ref{ss:tw}. As we show in Subsection~\ref{ss:appabs}, the use of twist structures extends far beyond the context of Nelson logics. In the final subsection (\ref{ss:pri}), we provide a brief review of the Priestley-style duality for $\PN$-lattices, which we revisit in Section~\ref{sec:rough} from the perspective of rough set theory.

In Section~\ref{sec:n3} we move towards results that are specific to $\N$-lattices, but we begin 
by showing how both $\N$- and $\PN$-lattices may be presented as residuated algebras (Subsection~\ref{ss:nasres}). We then restrict our attention to $\N$-lattices, 
discussing a number of alternative characterisations 
 (Subsection~\ref{ss:fuca}), which can be grouped into (1) the syntactic, i.e.~obtained by means of 
identities other than~\eqref{Eqn:Nel}, (2) the order-theoretic (e.g.~Proposition~\ref{Prop:eqord}) and 
(3) the congruence-theoretic, leading to the notion of congruence-orderability. 

In Section~\ref{sec:rough} we introduce  rough sets determined by quasiorders, which  give rise to special Nelson algebras. 
Rough set theory was initially based on the assumption that knowledge about the objects of a universe of discourse is expressed in terms of an equivalence  interpreted so that two objects are related if we cannot distinguish them by using the information we have. Rough sets induced by  equivalences form regular double Stone algebras. 
Moreover, each Nelson algebra defined on an algebraic lattice is isomorphic to the rough set algebra based on a quasiorder. We  provide a general representation theorem, 
stating that each Nelson algebra is isomorphic to a subalgebra of a Nelson algebra induced by a quasiorder. Additionally, in this section, we give a completeness result with the finite model property for Nelson logic. The section concludes by exploring residuated lattices determined by rough sets.

\section{Paraconsistent Nelson logic and  $\PN$-lattices} 
\label{sec:n4}
 
\subsection{Nelson logic and its algebraic semantics} 
\label{subsec:nel}

Nelson's system $\N$ is a well-known 
non-classical logic
that combines the constructive approach of 
intuitionistic logic with a De~Morgan negation.
 
A  Hilbert-style calculus for $\N$ may be obtained by 
expanding the positive fragment of intuitionistic logic (axioms (A1) to (A8) in Subsection~\ref{ss:pcl})
with a new \emph{strong negation} connective
$\nnot$ whose behaviour is captured by the following  schemata:
\begin{enumerate}[label=(A\arabic*)]
\setcounter{enumi}{8}
\item $\nnot \nnot \phi \leftrightarrow \phi$
\item $\nnot (\phi \lor\psi) \leftrightarrow (\nnot \phi\land\nnot \psi)$
\item $\nnot (\phi \land\psi) \leftrightarrow (\nnot \phi\lor\nnot \psi)$
\item $\nnot (\phi \to \psi) \leftrightarrow (\phi\land\nnot \psi)$

\item $\phi \to (\nnot \phi \to \psi)$.
\end{enumerate}

The algebraic models of $\N$ form a variety known as \emph{Nelson algebras}  (or \emph{Nelson residuated lattices} or \emph{$\N$-lattices}): this variety has been studied 
since at least the late 1950s (firstly by H.~Rasiowa; see~\cite{Rasi74} and  references therein) and is 
by now 
fairly well understood. 
One of the  main algebraic insights on Nelson algebras was obtained, towards the end of the 1970s, with the discovery
(independently due to M.~M.~Fidel and D.~Vakarelov)
that every $\N$-lattice can be represented as a special binary product 
of (here called a \emph{twist-structure over})  a Heyting algebra.
This correspondence was formulated   by A.~Sendlewski, in the early 1990's, as a categorical equivalence between 
$\N$-lattices
and a category of enriched Heyting algebras, which  made it possible to translate a number of fundamental results from the more developed  theory of intuitionistic  logic into the Nelson realm. 

After  Sendlewski's work, 
the most important advance in the theory of $\N$-lattices 
has probably been  the discovery
   that Nelson logic can be viewed as  one among the so-called substructural logics.
  This result
-- first proved  in 2008 by M.~Spinks and R.~Veroff~\cite{Spin08a,Spin08b} -- 
   entails that, modulo  algebraic signature formalities, 
   $\N$-lattices
   can be presented as a subvariety of (commutative, integral, bounded) residuated lattices~\cite{GaJiKoOn07}; whence the alternative name of   \emph{Nelson residuated lattices}.
Given the flourish of studies on substructural logics and residuated structures (leading to and following the book~\cite{GaJiKoOn07}),
this alternative perspective on 
$\N$-lattices
turned out to be very fruitful. Indeed, it  made in the first place possible 
to recover or recast a number of results on 
$\N$-lattices
by specialising  more general ones about
residuated structures. Furthermore, and perhaps more interestingly, it allowed scholars to formulate new questions
that can  be best appreciated within the broader framework of residuated lattices.


Formally, a  \vocab{commutative residuated lattice} (CRL)
is an algebra $\A = \zseq{A; \land, \lor, *, \imp, 1}$ of type $\zseq{2, 2, 2, 2, 0}$
such that: 
\begin{enumerate}[label=(\roman*)]
\item $\zseq{A; \land, \lor }$ is a lattice (with an order $\leq$);
\item $\zseq{A; *, 1 }$ is a commutative monoid;
\item for all $a, b, c \in A$, we have
$a * b \leq c$ iff $a \leq b \imp c$
. 
\end{enumerate}
The last property, known as \emph{residuation},  may be expressed within the setting of residuated lattices
purely by means of identities.
Thus the class of all 
CRLs
is equational.
A commutative residuated lattice $\A = \zseq{A; \land, \lor, *, \imp, 1}$ is \emph{integral} when the constant $1$ is interpreted 
as the greatest element of the  lattice $\zseq{A; \land, \lor }$, and it is \emph{bounded}
when $\zseq{A; \land, \lor }$ also has a least element -- denoted by $0$, and usually added to the algebraic
language as a constant. We  abbreviate 
``commutative integral (bounded) residuated lattice'' as CI(B)RL.

On each CIBRL $\A = \zseq{A; \land, \lor, *, \imp, 1, 0 }$, 
a \emph{negation} operator is usually defined by 
$\nnot x : = x \imp 0$. Using this abbreviation, we may 
compactly express two key properties of $\N$-lattices, namely  
\emph{involutivity} 
and the \emph{Nelson identity}. 
Formally, a CIBRL 
is \emph{involutive} when it 
satisfies the double negation identity: 
\begin{gather}
\nnot \nnot x \approx x. \tag{Inv} \label{Eqn:Inv}
\end{gather}
In fact, ``one half'' of the above identity, namely $x \leq \nnot \nnot x$ 
(which, as is customary, we consider an abbreviation of the identity $x \approx x \land \nnot \nnot x$) is satisfied by every CIBRL.

$\N$-lattices are involutive CIBRLs, but not every involutive CIBRL is an $\N$-lattice. The last missing ingredient is the 
\emph{Nelson identity}:
\begin{gather}
\lmb(x \imp (x \imp y)\rmb) \land \lmb(\nnot y \imp (\nnot y \imp \nnot x)\rmb) \approx x \imp y. \tag{Nelson} \label{Eqn:Nel}
\end{gather}
Again, we note that 
the following
half of the Nelson identity is easily seen to be satisfied by every CIBRL: 
$$
x \imp y \leq \lmb(x \imp (x \imp y)\rmb) \land \lmb(\nnot y \imp (\nnot y \imp \nnot x)\rmb).
$$
$\N$-lattices may be defined, as a subvariety of CIBRLs, precisely by 
the identities~\eqref{Eqn:Inv} and~\eqref{Eqn:Nel}.

The term $x \imp (x \imp y)$ appearing in the left side of identity~\eqref{Eqn:Nel} plays an important role within the theory of  $\N$-lattices. Indeed, letting
$$x \to y := x \imp (x \imp y)$$ 
we obtain the implication connective $\to$
(known as \emph{weak implication}) which was originally used in Nelson's presentation of the logic.
Then~\eqref{Eqn:Nel} can  be rewritten more compactly as:
$$
x \imp y \approx \lmb(x \to y \rmb) \land \lmb(\nnot y \to  \nnot x).
$$
This rephrasing suggests that, conversely, the residuated (usually known as the \emph{strong}) implication is
term definable from the weak one (provided the conjunction and the negation are available).
The interplay between the two implications is  one of the distinctive features of Nelson logics: 
while the weak one plays a prominent role from a logical point of view -- it is the implication that enjoys the Deduction Theorem,
and which allows us to view Nelson logic as a conservative expansion of positive intuitionistic logic --
the strong one witnesses the algebraisability~\cite{Blok89} of Nelson logic, and allows us to view  $\N$ as a substructural logic.

In the next subsection we shall see an alternative presentation of $\N$-lattices
which takes the weak (rather than the strong) implication as primitive, and is therefore closer to the original setting 
 proposed by Nelson. 

\subsection{Paraconsistent constructive logic with strong negation} \label{ss:pcl}

\emph{Paraconsistent constructive logic with strong negation} ($\PN$) was introduced by Nelson
in~\cite{Almu84} -- though equivalent systems were  also independently considered in the earlier papers~\cite{Lope72,Rout74} -- 
as a 
generalisation
of 
$\N$
which might 
``be applied to inconsistent subject matter without necessarily generating a trivial theory''~\cite[p.~231]{Almu84}.
 Formally, $\PN$ is also (as is 
 $\N$) a conservative expansion of the negation-free fragment of the intuitionistic propositional calculus 
 \cite[Ch.~X]{Rasi74} by a unary logical connective $\nnot$ 
 of \emph{strong negation}~\cite[p.~323]{Spin18}.

The  logico-algebraic community devoted comparatively little attention\footnote{But see, e.g.,~\cite{Wans95},
where $\PN$ is proposed as a logic for non-monotonic reasoning.}
to
$\PN$ 
until 
S.~Odintsov's contribution
in the early 2000's, in which
the nowadays standard
algebraic semantics of 
(the propositional part of)
$\PN$ was introduced and investigated; 
this research, originating from Odintsov's DSc thesis~\cite{Odin07},
was published in the series of papers~\cite{Odin03,Odin04,Odin05} and the monograph~\cite{Odin08}. 
Thanks also to the effort of a few other authors~\cite{Busa09,Kami12,Spin18}, 
$\PN$ is nowadays known to play a central role in the study of three- and four-valued logics. 


Odintsov has shown that the algebraic models of $\PN$
form a class of De Morgan lattices structurally enriched with a (
``weak'')
implication, which he dubbed \emph{$\PN$-lattices}.
Generalising the works of Fidel, Vakarelov, and Sendlewski on 
$\N$-lattices,
he also provided a most useful representation theorem for $\PN$-lattices 
by means of \emph{twist-structures}, which we shall review in Subsection~\ref{ss:tw}.

The basic propositional  language of $\PN$ 
is 
the same as that of the constant-free version of 
$\N$ as originally introduced by Nelson, i.e.~$\{\land, \lor, \to, \nnot \}$. 
Further connectives $*, \imp, \leftrightarrow$, and $\Leftrightarrow$
may be defined as follows:
the \emph{strong implication}
$\imp$ is given by 
$$
\phi \imp \psi := (\phi \to \psi) \land (\nnot \psi \to \nnot \phi),
$$
the \emph{strong} (or \emph{multiplicative})
\emph{conjunction} $*$ by
$$
\phi * \psi := \nnot (\phi \imp \nnot \psi),
$$
and we have two  equivalence connectives,
a \emph{weak} one $\leftrightarrow$
 given by
$$
\phi \leftrightarrow \psi := (\phi \to \psi) \land ( \psi \to  \phi)
$$ 
and
a \emph{strong} one 
$\Leftrightarrow$ given by 
$$
\phi \Leftrightarrow \psi := (\phi \imp \psi) \land ( \psi \imp  \phi).
$$

$\PN$ may be introduced in the standard way through a Hilbert-style calculus~\cite[p.~133]{Odin08} consisting of the following
axiom schemes, with \emph{modus ponens} as the only inference rule:



\begin{enumerate}[label=(A\arabic*)]
\item $\phi \to (\psi \to \phi)$
\item $(\phi \to (\psi \to \gamma)) \to ((\phi \to \psi) \to (\phi \to \gamma))$
\item $(\phi \land \psi) \to \phi$
\item $(\phi \land \psi) \to \psi$
\item $(\phi \to \psi) \to ((\phi \to \gamma) \to (\phi \to (\psi \land\gamma)))$
\item $\phi \to (\phi \lor \psi)$
\item $\psi \to (\phi \lor \psi)$
\item $(\phi \to \gamma) \to ((\psi \to \gamma) \to ((\phi \lor \psi) \to \gamma))$
\item $\nnot \nnot \phi \leftrightarrow \phi$
\item $\nnot (\phi \lor\psi) \leftrightarrow (\nnot \phi\land\nnot \psi)$
\item $\nnot (\phi \land\psi) \leftrightarrow (\nnot \phi\lor\nnot \psi)$
\item $\nnot (\phi \to \psi) \leftrightarrow (\phi\land\nnot \psi)$
\end{enumerate}


A Hilbert-style presentation for $\N$ 
can be obtained
from the preceding one
by adding  the single \emph{explosion}
axiom: 

\begin{enumerate}[label=(A\arabic*)]
\setcounter{enumi}{12}
\item $\phi \to (\nnot \phi \to \psi)$.
\end{enumerate}

Odintsov~\cite{Odin03} gives three
alternative  but equivalent algebra-based semantics for $\PN$,
namely
(1) Fidel-structures (or $\mathbf{F}$-structures),
(2) twist-structures 
and (3) $\PN$-lattices.
%
%
%
$\mathbf{F}$\textit{-structures} were introduced in~\cite{Fide80} as a semantics (for~$\mathbf{N3}$) alternative to twist-structures. An $\mathbf{F}$-structure $\zseq{\mathbf{A}; \zset{N_a}_{a \in A}}$ is a Heyting algebra~$\A$ structurally enriched with a special family 
of unary predicates 
distinguishing classes of counterexamples for elements of~$\A$.
Thus each 
$N_a$
 is a set of all possible negations of $a \in A$. The study of $\mathbf{F}$-structures originated with Fidel's investigations~\cite{Fide77}
  into da Costa's family $\mathbf{C}_n$ ($n \geq 1$) of 
  so-called \emph{logics of formal inconsistency}~\cite{DaCa74}, 
 and they have been mostly exploited in that context; a modern treatment of $\mathbf{F}$-structures
 for 
 $\mathbf{N3}$ and~$\mathbf{N4}$ 
 can be found in
 ~\cite{Odin08}. 
 As $\mathbf{F}$-structures and twist-structures are mutually interdefinable \cite[Chapter 8,~{\S}3]{Odin08}, we shall not deal   further  
with $\mathbf{F}$-structures here, preferring instead to work with twist-structures, which we treat in detail in the next subsection.

Before we present  Odintsov's definition of $\PN$-lattice, we need to introduce a couple of 
definitions. A \emph{De Morgan lattice} is an algebra  $\zseq{A; \land, \lor, \nnot}$ such that
$\zseq{A; \land, \lor}$ is a distributive lattice and the negation $\nnot$ satisfies the \emph{double negation law}
\[ \nnot \nnot x \approx x\] 
and the two \emph{De Morgan laws} 
\[ \nnot (x \lor y) \approx \nnot x \land \nnot y  
\qquad
    \nnot (x \land y) \approx \nnot x \lor \nnot y. \]
A De Morgan lattice whose lattice reduct is bounded is known as a \emph{De Morgan algebra}.

The following definition of $\PN$-lattices can be found in Odintsov's paper~\cite[Def.~5.1]{Odin03}.
It is clearly inspired by (and generalises) Rasiowa's
definition of $\N$-lattices~\cite[Ch.~V]{Rasi74}; see also the presentation of~\cite[p.~357--8]{Blok84}, which appears even closer to Odintsov's. 
Following standard usage, we  write $\abs{a}$ as a shorthand for $a \to a$, for every
$a \in A$. Similarly, given a term $\phi$, we let $\abs{\phi} : = \phi \to \phi$.

\begin{definition} \label{Def:N4}
An algebra $\alg{A} = \zseq{A; \land, \lor, \to, \nnot}$
of type $\zseq{2, 2, 2, 1}$ is an \emph{$\PN$-lattice} 
if:
\begin{enumerate}[label=(N4.\roman*), leftmargin=*, itemsep=2pt]
\item \label{Itm:N1} The reduct $\zseq{A; \land, \lor, \nnot}$ is a De Morgan
      lattice with lattice ordering~$\leq$.
\item \label{Itm:N2} The relation $\preceq$ on~$A$ defined for all $a, b \in A$ by
      $a \preceq b$ iff $a \to b = \abs { a \to b}$
      is a quasiorder (\ie\ is reflexive and transitive). 
\item \label{Itm:N3} The relation $\cm {:=} \preceq \cap \mathrel{(\preceq)^{-1}}$
      is a congruence on the reduct $\zseq{A; \land, \lor, \to}$
      and the quotient algebra $\zseq{A; \land, \lor, \to}/{\Xi}$ is
      a Brouwerian lattice (i.e.~the $0$-free subreduct of a Heyting algebra; see below).
\item \label{Itm:N4} For all $a, b \in A$, it holds that $\nnot (a \to b) \equiv a \land \nnot b \pmod{{\cm}}$.
\item \label{Itm:N5} For all $a, b \in A$, it holds that $a \leq b$ iff $a \preceq b$ and $\nnot b \preceq \nnot a$. 
\end{enumerate}
\end{definition}

A typical example of a $\PN$-lattice (which is not an $\N$-lattice) is the well-known  four-element lattice
$\alg{FOUR}$,
whose Hasse diagram is
shown in Figure~\ref{fig:2}, which provides a semantical basis for the Belnap-Dunn ``useful four-valued logic''~\cite{Beln77}.
The behaviour of the (non-lattice) algebraic operations on $\alg{FOUR}$ (as well as the names of its elements)
is determined by the twist construction that we shall introduce in a moment
(see Example~\ref{Exa:Four}
). Indeed, 
the twist construction
is an easy means of producing further examples of ``proper''  $\PN$-lattices 
(\emph{e.g}.~by taking the factor lattice to be a Brouwerian lattice that is not a Heyting algebra).

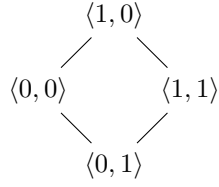
\begin{figure}[h]
\caption{The Belnap lattice $\alg{FOUR}$.}
\label{fig:2}

\bigskip

\begin{center}
\begin{tikzpicture}[scale=1,
  dot/.style={circle,fill,inner sep=1.6pt,outer sep=2pt}]
  \node (K6b) at (0,0) {$\la 0, 1 \ra $};
  \node (K6d) at (1,1) {$\la 1, 1 \ra $};
    \node (K6g) at (0,2) {$\la 1, 0 \ra $};
            \node (K6i) at (-1,1) {$\la 0, 0\ra $};
    \draw[-] (K6d) -- (K6g);
    \draw[-] 
    (K6b) -- (K6d);
    \draw[-] (K6b) -- (K6i) -- (K6g);
\end{tikzpicture}
\end{center}
\end{figure}



In this setting,  $\N$-lattices may be defined as precisely those $\PN$-lattices that further satisfy 
any of the equivalent items listed in Proposition~\ref{Prop:NelWeakening}  below
(
we write $x \pr y$ as an abbreviation of the identity $x \to y \approx \abs{ x \to y }$).

\begin{proposition} \label{Prop:NelWeakening}
For every  $\PN$-lattice $\alg{A}$, 
the following are equivalent:
\begin{enumerate}[label={\rm (\alph*)}]
\item \label{Itm:N4-1} $\alg{A}$ is an $\N$-lattice. 
\item \label{Itm:N4-2} The unary term $
\abs{x}$ is constant over $\alg{A}$.
\item \label{Itm:N4-3} The unary term $x \imp x$ is constant over $\alg{A}$.
\item \label{Itm:N4-4} $\alg{A} \models x \land \nnot x \pr y$. 
\item \label{Itm:N4-5} $\alg{A} \models x \leq y \imp x$.
\item \label{Itm:N4-6} $\alg{A} \models x * y \leq x$, where
$x *y := \nnot (x \imp \nnot y)$.
\end{enumerate}
\end{proposition}


In the standard 
 terminology on substructural logics,
the last item of Proposition~\ref{Prop:NelWeakening} says
that~$\mathbf{N3}$
is the axiomatic extension of~$\mathbf{N4}$ by \emph{weakening}.
We note that every $\PN$-lattice satisfies the identity $\abs{x} \approx x \imp x$, which explains the equivalence
between the second and third item.
In an $\N$-lattice, the term
$\abs{x}$ (or, equivalently, $x \imp x$)
is interpreted
as the greatest element of the lattice reduct of $\A$, and
$\nnot \abs{x}$
as the least element. No constant term
is definable, in general, over $\PN$-lattices. 
On the other hand, as witnessed by $\alg{FOUR}$, only requiring the lattice reduct of an $\PN$-lattice $\A$
to be bounded  is not sufficient to  ensure that $\A$ is
an $\N$-lattice (more on this below).

As in the case of Rasiowa's definition of $\N$-lattices, it is not immediately
evident that the above-defined class of algebras can  be axiomatised
by means of identities only; this is a further result
established by Odintsov~\cite[Thm.~6.3]{Odin03}
exploiting Pynko's insights
on expansions of the Belnap-Dunn logic~\cite{Pynk99}.

 As a variety of algebras in the language $\{ \land, \lor, \to, \nnot\}$, the class of $\PN$-lattices can be presented through the following identities~\cite[Def.~6.1]{Odin03}.
 

\begin{enumerate}[label=(N4.\alph*), leftmargin=*, itemsep=2pt]
\item \label{Itm:PN1eq} $\abs{x}  \to y \approx y $
 \item \label{Itm:PN2eq} $(x \land y) \to z \approx x \to (y \to z)$
\item \label{Itm:PN3eq} $x \to (y \land z) \approx (x \to y) \land (x \to z)$
\item \label{Itm:PN4eq}  $(x \lor y) \to z \approx (x \to z) \land (y \to z) $
\item \label{Itm:PN5eq}  $x \leq y \to x$
\item \label{Itm:PN6eq} $(x \to y) \land (y \to z) \pr x \to z$
\item \label{Itm:PN7eq} $ \abs{x} \leq \nnot (x \to y)  \to x$
\item \label{Itm:PN8eq} $ \abs{x} \leq y \to  (\nnot x \to \nnot (y \to x ) ) $
\item \label{Itm:PN9eq} $ \abs{x} \leq \nnot (y \to  x) \to \nnot x  $
\item \label{Itm:PN10eq} $ x \land (x \to y) \leq y \lor  \nnot (\nnot y \to  \nnot x) $
\item \label{Itm:PN11eq} $ x \leq \abs{y} \lor  \nnot (\nnot y \to  \nnot x) $.
\end{enumerate}

While some of the  above identities 
are somewhat opaque (even to a specialist),
no essentially more transparent axiomatisation seems to be currently available -- by contrast, for $\N$-lattices we have
the elegant equational axiomatisation due to Brignole and Monteiro~\cite{Brig69}. 


Odintsov  proved that $\PN$ is strongly complete with respect to the class of $\PN$-lattices~\cite[Thm.~5.5]{Odin03}.
His results actually entail that $\PN$ 
is \emph{regularly algebraisable} in the sense of~\cite{Blok89}, and has the variety of $\PN$-lattices as its equivalent algebraic
semantics~\cite[Thm.~2.6]{Rivi11}; one may take 
 $\{ \phi \Leftrightarrow \psi \}$
as a system of equivalence formulas 
and 
$\{ x \approx \abs{x} \}$
as a system of defining equations witnessing algebraisability\footnote{Other connectives might be used here, suggesting a number of 
fragments of $\PN$ that are also algebraisable (e.g.~the $\{ \nnot, \to \}$-fragment, the $\{ \imp \}$-fragment and so on).
A study of these fragments has not been attempted yet,
whereas some work has been recently done on fragments of constructive logic with strong negation:
see~\cite{thiago2021negation,riviecciotwonegs,rivieccio2021fragments,rivieccio2021quasi}.
 }. 
 
 
Algebraisability of $\PN$ w.r.t.~the variety of $\PN$-lattices
 immediately entails that the lattice of axiomatic extensions of $\PN$
 is dually isomorphic to the lattice of subvarieties of $\PN$-lattices~\cite[Thm.~6.10]{Odin03};
by the same token, a more general isomorphism holds between
 the lattice of all \emph{finitary} extensions of $\PN$ and the lattice of
 subquasivarieties of $\PN$-lattices \cite[Lecture~6;~Thm.~1.4]{Blok99b},~\cite[Thms.~3.13--3.15]{Pynk99b}. 
This connection is further exploited
 and investigated in~\cite{Odin05,Odin06}. 

As mentioned earlier, an $\PN$-lattice may be bounded without necessarily being an $\N$-lattice.
This suggests that by adding a \emph{falsum} constant $\bot$ to $\PN$ one obtains
a logic $\PN$$_\bot$ --
having  the variety of bounded $\PN$-lattices as equivalent algebraic semantics --
which is still distinct from $\N$. 
In $\PN$$_\bot$ the presence of the propositional constant
$\bot$ allows one to introduce a second (so-called \emph{intuitionistic}) negation $\neg$
given by
$\neg \phi : = \phi \to \bot$ (but notice that $\nnot$ is not \emph{strong} with respect to 
$\neg$ in the same sense as it is in $\N$: see~\cite[p.~4]{Odin08}).
One can then show that 
$\PN$$_\bot$ is a conservative expansion of full intuitionistic logic~\cite[Cor.~8.6.6]{Odin08}.
The variety of bounded $\PN$-lattices is better-behaved, from an algebraic point of view, than the class of all
$\PN$-lattices: it is indeed in the bounded setting that one can  exploit at its best the connection between 
$\PN$-lattices and Heyting algebras given by the twist-structure representation. Before going
into further detail on these results, let us
 introduce formally the definition of twist-structure
and  
the corresponding representation theorem.


\subsection{Twist-structures for $\PN$-lattices}
\label{ss:tw}
 
The twist-structure is a straightforward  yet powerful algebraic
construction used to represent $\N$- and $\PN$-lattices as special 
direct squares of (generalised) Heyting algebras.  For $\N$-lattices, twist-structures
were first introduced by Vakarelov~\cite{Vaka77} 
and independently by Fidel~\cite{Fide78}, but a full
representation -- in fact, a categorical equivalence --
 was only obtained by
Sendlewski~\cite{Send90};
the algebraic content of the latter paper was extended to $\PN$-lattices by Odintsov~\cite{Odin04}.
In more recent years, the  twist construction has
 been employed to obtain similar representations
for a wide family of related algebras: see~\cite{OnRi14,
Rivi11,riviecciotwonegs,Ri20,Rivi0xa,WNM,rivieccio2017four,
rivieccio2019quasi,riviecciosp2021quasi} 
as well as  the works cited in the next subsection.

Recall that a \emph{Brouwerian lattice} (cf.~item~\ref{Itm:N3} of Definition~\ref{Def:N4})
 is an algebra
$  \Ll[B] = \la B ; 
\wedge ,\vee , \to  \ra$ such that
$  \la B ; 
\wedge ,\vee   \ra$ is a lattice with an order $\leq$
and $\to$ is the residuum of $\land$, that is,
$a \land b \leq c $ iff $a \leq b \to c$, for all $a,b, c \in B$.
The latter property entails that the lattice  $  \la B ; 
\wedge ,\vee   \ra$ is distributive and has a top element $1$ (while the bottom element may not exist), so $1$ can
be safely added to the algebraic language as a primitive nullary operation.

On every Brouwerian lattice $ \A[B] = \la B ; 
\wedge ,\vee , \to  \ra$ one can define  the 
set of \emph{dense elements}
$D(\A[B]) := \{ a \lor (a \to b) : a,b \in B\}$, which is a lattice filter
of $\A[B]$. Accordingly, a lattice  $F \subseteq B$ 
such that $D(\A[B]) \subseteq F $
is said to be
\emph{dense}  (or \emph{Boolean}). 
As is well known,
every filter $F \subseteq B$ is the congruence kernel of a congruence on $\A[B]$,
and it is easy to see that a filter $F$ is dense if and only if the associated 
congruence $\theta_F$ 
yields a generalised Boolean algebra $\A[B]/ \theta_F$ as a quotient
(generalised Boolean algebras are precisely the zero-free
subreducts of Boolean algebras: see e.g.~\cite{gratzer1958generalized}).

\begin{definition}
\label{Def:twi}
Let $ \A[B] = \la B ; 
\wedge ,\vee , \to  \ra$
be a Brouwerian lattice, let
$F \subseteq B$ be a dense filter
and $I \subseteq B$ an arbitrary lattice ideal. 
The \emph{$\PN$-twist-structure}
$$
Tw( \A[B],  F, I ) =  \la A ; 
\land, \vee, 
\to,
 \nnot 
 \ra$$
 is the algebra with universe
 $$
 A :=  \{ \la a, b \ra \in B \times B : a \lor b \in F, \, a \land b \in I \}
 $$
and operations defined as follows:
for all $\la a,b \ra , \la c , d \ra \in B \times B$,
\begin{align*}
\nnot \la a , b  \ra & := \la b ,a  \ra, \\
\la a  , b  \ra \land \la c , d   \ra  &  := \la a  \land c  , b  \lor d  \ra, \\
\la a  , b  \ra \lor \la c , d   \ra  & := \la a  \lor c  , b  \land  d \ra, \\
\la a  , b  \ra \to \la c , d   \ra  & := \la a  \to c  , a  \land d  \ra. 
\end{align*}
\end{definition}

As the reader may have noticed,  following standard usage in the literature, in Definition~\ref{Def:twi} we are overloading the symbols
$\land, \lor$ and $\to$ to denote operations on the Brouwerian lattice $ \A[B] $ as well as on the twist-structure $Tw( \A[B],  F, I )$.
Every $\PN$-twist-structure is easily seen to be an $\PN$-lattice; moreover, 
Odintsov's representation result  states that
every $\PN$-lattice can be constructed in this way.

In this representation, $\N$-lattices correspond precisely to the twist-structures
$Tw( \A[B],  F, I )$ such that $\A[B]$ is a  Heyting algebra (with least element $0$)
and $I = \{ 0 \}$. 

\begin{example} \label{Exa:Four}
Let $ \A[B]_2 = \la B_2= \{0,1\} ; \wedge ,\vee , \to  \ra$ be the two-element Boolean algebra
viewed as a Brouwerian lattice.
Taking $F= I = B_2$, we have that 
$Tw( \A[B]_2,  B_2, B_2 )$ is the four-element Belnap
lattice of Figure~\ref{fig:2}.
Note that $
Tw( \A[B]_2,  B_2, B_2 )$ 
is not an $\N$-lattice because it does not satisfy
$\abs{x} \approx \abs{y}$ (cf.~Proposition~\ref{Prop:NelWeakening}). Indeed,
we have for instance
$\abs{\la 1, 1 \ra} = \la 1, 1 \ra \neq \la 1, 0 \ra = \abs{\la 1, 0 \ra} $.
\end{example}

The converse and more interesting direction of Odintsov's representation
works as follows. Given an $\PN$-lattice $\A$, one 
considers the quotient $B(\A) = \zseq{A; \land, \lor, \to}/{\Xi}$ defined in the preceding subsection, which is guaranteed
to be a Brouwerian lattice by item~\ref{Itm:N3} of Definition~\ref{Def:N4}, together with the sets
$A^+ : = \{ a \lor \nnot a : a \in A \}$
and
$A^-  : = \{ a \land \nnot a : a \in A \}$.
Denoting  the canonical epimorphism by $\pi \colon A  \to A /{\Xi}$ and defining
$
F(\A) : = \pi [A^+]
$
and
$
I(\A) : = \pi [A^-],
$
one has that $F(\A) \subseteq B(\A)$ is a dense filter 
and $I(\A) \subseteq B(\A) $ is a lattice ideal. 
This yields the following. 
  
\begin{theorem} \label{Prop:OdRep}
Every $\PN$-lattice $\alg{A}$
is isomorphic to the $\PN$-twist-structure 
\[Tw( B(\A),  F(\A), I(\A) )\] 
via the map $a \mapsto ( a /{\Xi}, \, {\nnot a} /{\Xi})$ 
for all $a \in A$.
\end{theorem}

The object-level correspondence  
established by Theorem~\ref{Prop:OdRep}
can  be easily extended to an equivalence between
the algebraic category of $\PN$-lattices (with algebraic homomorphisms as morphisms) 
and a category having as objects
triples of type $( \A[B],  F, I )$
and, as morphisms, the
Brouwerian lattice homomorphisms that respect
$F$ and $I$; this equivalence was first established (and exploited) in~\cite{JaRi14}, 
to which we refer for further details. 
This suggests that, from the point of view of (e.g.~Priestley-style) duality,
one may study $\PN$-lattices \emph{qua} triples $( \A[B],  F, I )$,
 relying on  the well-known  results of Esakia duality 
for (generalised) Heyting algebras. 

Thanks to Theorem~\ref{Prop:OdRep},
Odintsov~\cite[Sec.~4]{Odin04} was able to characterise 
 $\PN$-lattice homomorphisms  in terms of the homomorphisms 
between the corresponding Brouwerian lattices; the following 
description of congruences -- which can also be obtained as a corollary of the categorical equivalence --  is  particularly useful.

\begin{theorem} \label{Prop:OdCon}
The congruence lattice of 
every $\PN$-lattice $\alg{A}$
is isomorphic to the congruence lattice of the underlying Brouwerian lattice
$B(\A)$.
\end{theorem}

Among other things, Theorem~\ref{Prop:OdCon} entails that an $\PN$-lattice $\alg{A}$ is subdirectly irreducible
if and only if $B(\A)$ is a subdirectly irreducible Brouwerian lattice; this result 
 is obviously quite helpful in the study of subvarieties of $\PN$-lattices.
Furthermore, in view of the observation that
$\PN$-lattices are congruence-permutable~\cite[Thm.~4.24]{Spin18}, it is also an easy observation that
an $\PN$-lattice $\alg{A}$ is directly indecomposable
if and only if the Brouwerian lattice $B(\A)$ is directly indecomposable.

Theorem~\ref{Prop:OdRep} is  directly useful in the study of subvarieties of $\PN$-lattices as well,
because  equational conditions (say, on an $\PN$-lattice $\A$) may correspond to 
conditions on the Brouwerian lattice $B(\A)$ or on 
the subsets
$F(\A)$ and $ I(\A)$.
For instance, as mentioned earlier,
an $\PN$-lattice $\A$ is an $\N$-lattice
(i.e.~satisfies any of the items of Proposition~\ref{Prop:NelWeakening})
if and only if 
$ I(\A)$ is the singleton ideal (that is, iff $B(\A)$ has a least element $0$ and $ I(\A) = \{ 0 \}$).


We conclude the subsection with the anticipated overview on Odintsov's
results on the lattice of extensions of  $\PN_\bot$, and in particular on the connections with extensions 
of intuitionistic logic. 

As mentioned earlier, paraconsistent constructive logic with strong negation $\PN$
may be viewed as a (conservative) expansion of positive intuitionistic logic; more precisely,
the $\{ \land, \lor, \to \}$-fragments of both logics coincide -- the connectives  here being interpreted as suggested by the notation, in particular,
$\to$ is the weak implication on $\PN$. A similar result holds for the expansion of  $\PN$ with truth constants:
the $\{ \land, \lor, \to, \neg \}$-fragment of $\PN_\bot$ (the connective $\neg$ being defined by $\neg \phi := \phi \to \bot$)
coincides with full intuitionistic logic~\cite[Cor.~8.6.6]{Odin08}.
These observations suggest that one can define a mapping $\sigma$  associating, to a given a logic $\logic{L}$ extending
$\PN_\bot$, the $\{ \land, \lor, \to, \neg \}$-fragment of $\logic{L}$, denoted  $\sigma(\logic{L})$, which is a super-intuitionistic logic (i.e.~a strengthening of intuitionistic logic).

Given a bounded $\PN$-lattice $\A \in \PN_\bot$, we
denote by $H(\A):= \zseq{A; \land, \lor, \to, \bot}/{\Xi}$ the Heyting algebra quotient
obtained as per Definition~\ref{Def:N4}.
Similarly,
given a class of $\PN_\bot$-lattices $\mathsf{K}$, let 
$H(\mathsf{K}) : = \{ H(\A) : \A \in \mathsf{K} \}$.
%
Recall that, in the setting of finitary extensions of a  finitary base logic 
(in our case, $\PN_\bot$) which is algebraisable with respect to a given (quasi)variety 
(in our case, the variety of $\PN_\bot$-lattices),
 one can associate to every subclass  $\mathsf{K}$ ($ \subseteq \PN_\bot$)
 the (algebraisable) logic $ \logic{L} (\mathsf{K})$ determined by the quasivariety
 generated by $\mathsf{K}$. 

\begin{proposition}\cite[Lemma~10.1.2,~Prop.~10.1.1]{Odin08}
\label{Prop:OdExt}
\
\begin{enumerate}[label={\rm (\alph*)}]
\item	For every  $\A \in \PN_\bot$ and every formula $\phi $ in the language of intuitionistic logic, we have
$\A \models \phi \approx \abs{\phi}$ if and only if $H(\A) \models \phi \approx 1 
$.
\item	If 
$ \logic{L} (\mathsf{K}) $
is
a logic extending $\PN_\bot$ 
determined by  a class $\mathsf{K}$ of $\PN_\bot$-lattices,
then $\sigma(\logic{L} (\mathsf{K})) = \logic{L}(H(\mathsf{K}))$.
\end{enumerate}
\end{proposition} 

Conversely, given a super-intuitionistic logic $\logic{J}$, one may consider the family
$\sigma^{-1}(\logic{J})$ of  super-$\PN_\bot$-logics that are conservative expansions  of $\logic{J}$. 
This class of logics forms an interval 
$[ \eta(\logic{J}), \eta^{\circ}(\logic{J})]$
in the lattice of $\PN_\bot$-extensions whose endpoints 
are the logics 
defined as follows: $\eta(\logic{J})$
is the expansion of $\logic{J}$ obtained by enlarging the language of $\logic{J}$ with  the strong negation connective, axiomatised by axioms 
(A9)--(A12) from our  Hilbert-style presentation of paraconsistent Nelson logic,
and $\eta^{\circ}(\logic{J})$ is the axiomatic extension of $\eta(\logic{J})$ obtained by  adding the explosion axiom
($\phi \to (\nnot \phi \to \psi)
$) and the  \emph{normality} axiom
$\neg \neg (\phi \lor \nnot \phi)$~\cite[Prop.~10.1.9]{Odin08}.
These two propositional axioms have a straightforward algebraic interpretation: the former, as we have seen, 
characterises $\N$-lattices among $\PN$-lattices (or, equivalently, the
$\PN$-lattices of type $Tw( \A[B],  F, I )$ with $ \A[B]$ bounded by $0$ and $I = \{ 0 \}$), 
while the latter
characterises  the $\PN$-lattices of type $Tw( \A[B],  F, I )$ such that $F$ is the least dense filter,
i.e.~$F =D(\A[B]) $.

The maps $\sigma$,  $\eta$ and $\eta^{\circ}$
provide the following link between family of super-$\PN_\bot$-logics and the family of super-intuitionistic logics, both viewed as complete lattices~\cite[Prop.~10.1.12]{Odin08}:
\begin{enumerate}[label=(\roman*)]
\item $\sigma$ is a  complete and surjective  lattice homomorphism;
\item $\eta$ is a complete  lattice embedding;
\item $\eta^{\circ}$ is an isomorphism between the lattice of  super-intuitionistic logics
and the lattice of normal (i.e.~satisfying~$\neg \neg (\phi \lor \nnot \phi)$) extensions of Nelson logic. 
\end{enumerate}
This connection 
may then be used to obtain information on the order-theoretic structure of the lattice of super-$\PN_\bot$-logics. For instance, excluding
the trivial logic, we can observe that this lattice has, besides a greatest element $\eta^{\circ}(\logic{CL})$, where $\logic{CL}$ is classical logic,
 exactly three co-atoms, one of them coinciding with three-valued {\L}ukasiewicz logic.  
Due to space limitations we shall not mention the numerous further results in this direction, but refer the reader to Sections 10.2--10.5 of~\cite{Odin08}.

%


\subsection{Applications and abstractions of twist-structures} \label{ss:appabs}

Over time, the twist-structure construction  has proven   to be
a flexible and powerful tool, which has been used mainly for
(1) `reducing' certain questions regarding unfamiliar classes of algebraic structures
to more familiar settings, and (2) introducing new classes of algebras -- usually
as intended semantics for some logical system --   possessing certain desired properties.
In more recent years,  twist-structures
have been generalised and extended in several directions,
with original research being naturally  focused on a third line of research, namely (3) the nature, properties and applicability of the construction itself.
In this subsection we shall  discuss a few outstanding examples from each of the above-mentioned
 trends. 

The first (1), to which most applications of twist-structures within the area of 
Nelson logics belong, is certainly the oldest among the three.
Since  the above-mentioned works of Fidel, Vakarelov and Sendlewski, 
its main goal has been   to reduce, as it were,  the study of 
$\N$-lattices, 
and later on also of $\PN$-lattices (Odintsov), to that of (respectively) Heyting algebras and their
$0$-free subreducts (i.e.~Brouwerian lattices),
thus being able to exploit well-developed techniques and  results available for the latter classes of algebras.
A similar technique was applied -- independently of research in the  Nelson setting, apparently -- 
to the study of \emph{bilattices}, algebraic structures 
introduced 
by Ginsberg~\cite{Gi88} as a uniform framework for (non-monotonic) inference in AI, which
generalise the Belnap-Dunn logic
carrying two simultaneous  lattice reducts.
In this case twist-structures provided a link between the theory of bilattices and that 
of plain lattices (or lattices with an involution); the connection between bilattices and the algebras of Nelson logics
 appears to have been first noticed
in~\cite{Rivi0xa}; 
we shall say more on bilattices later. 
To mention but one last example within this trend,
twist-structures have been used to provide a representation for \emph{Sugihara monoids} (algebraic models
of certain relevance logics) as twist-products over \emph{relative Stone algebras} (i.e.~prelinear Brouwerian lattices)
enriched with a modal operator called a \emph{nucleus};
a similar approach has been recently employed to extend the twist-structure construction
beyond the setting of involutive lattices (more on this below).

Having observed that a number of algebras associated to non-classical logics are indeed representable as twist-structures -- e.g.~$\N$- and $\PN$-lattices, various classes of bilattices, Sugihara monoids,  Kleene algebras~\cite{Cign86} 
-- the obvious next step for a mathematical mind was to try and explore abstractly
the limits of the twist construction itself. This resulted in the research directions introduced above as (2) and (3). 

The earliest papers belonging to (2) were perhaps~\cite{JaRi12,Rivi11}, 
followed over the next decade by a number of others~\cite{agliano2021varieties,Busa14,ghorbani2016hoop,OnRi14,
Ri14,rivieccio2017four}. 
The main idea in this setting is to start with a familiar class of algebras (call it $\mathsf{K}$)
and 
look at the class   $Tw(\mathsf{K})$ of all algebras that can be obtained
as twist-structures over  members of $\mathsf{K}$. In most of the examples  appearing in the above-mentioned papers,
$\mathsf{K}$ is a class of residuated lattices which gives rise to a class  $Tw(\mathsf{K})$ whose members
are also residuated lattices with some additional operations (in particular, an involutive negation
given as in Definition~\ref{Def:twi}). Typical questions in this setting are: 
\begin{itemize}
\item
whether $Tw(\mathsf{K})$ forms an equational class, and, if so, one that  admits a transparent presentation relating to 
more familiar classes of algebras; 

\item
which universal algebraic properties does $Tw(\mathsf{K})$ inherit from
 $\mathsf{K}$, or viceversa;
in particular,
whether a result analogous to Theorem~\ref{Prop:OdCon} -- establishing a correspondence between the generators of both classes, if these are equational -- can be obtained; 

\item what is the relation between the logical counterpart of $\mathsf{K}$  and the logic  naturally associated to  $Tw(\mathsf{K})$.
\end{itemize}

To mention but one example, the paper~\cite{Busa14} addresses the following question. Consider the twist construction that, given a Brouwerian lattice $\A[L]$, produces a family of $\PN$-lattices, i.e.~all the possible twist-structures over $\A[L]$
according to Definition~\ref{Def:twi}. Suppose we relax the requirements on the factors, 
that is,
we only require $\A[L]$ to be a (commutative, integral) residuated lattice. Then a straightforward generalisation of the construction
given in Definition~\ref{Def:twi}  allows one to obtain a family of twist-structures over  $\A[L]$, each of them being
itself a (not necessarily integral) commutative residuated lattice~\cite[Thm.~3.1]{Busa14}.
In the terminology introduced above, letting  $\mathsf{CIRL}$  be the variety of all commutative integral residuated lattices, denote by 
$Tw(\mathsf{CIRL})$  the class of algebras obtained as twist-structures over some member of $\mathsf{CIRL}$;
the algebras in  $Tw(\mathsf{CIRL})$ have been dubbed  \emph{Kalman lattices} in~\cite{Busa14}.
The authors provide an equational axiomatisation for Kalman lattices, and extend to the latter class
a number of well-known results on $\PN$-lattices, establishing in particular
an analogue of Theorem~\ref{Prop:OdCon}.
It is further proved in~\cite{Busa14} that
every Kalman lattice $\A$ embeds into a twist-structure $Tw(\A[L])$, with
$\A[L] \in \mathsf{CIRL}$; an isomorphism result such as that of Theorem~\ref{Prop:OdRep}, however, 
could so far only be established
for  the subvariety  of $\mathsf{CIRL}$ that satisfies the \emph{Glivenko identity}:
$$
\neg \neg (x \to y) \approx x \to \neg \neg y.
$$

The above results have a straightforward categorical counterpart. Namely, 
one can establish an adjunction  between the algebraic  category of
$\mathsf{CIRL}$ and the category corresponding to Kalman lattices.
If we restrict our attention to special subcategories, then the adjunction can be upgraded to a categorical equivalence:
this holds for the case of the Kalman lattices corresponding to \emph{full} twist-structures 
(where $F = I = L$, in the terminology of Definition~\ref{Def:twi}) and also for 
the subcategory of $\mathsf{CIRL}$ whose members
 satisfy the Glivenko identity (the latter is not among the results
explicitly stated in~\cite{Busa14},
but is an easy consequence thereof). By specialising this result,  one can indeed recover
the categorical equivalence between $\PN$-lattices and twist-structures of type
$Tw( \A[B],  F, I )$ established in~\cite{JaRi14}. 


Over the last decade, the issues  dealt with in the papers belonging to the research line (2) have been naturally 
formulated in  increasingly more general and abstract contexts, therefore leading to an investigation on the very nature and applicability 
of the twist-structure construction, i.e., the line 
introduced earlier as (3). In our opinion, 
even today the latter studies cannot be said to have attained a fully satisfactory level of generality, for reasons that will be discussed presently. For the time being, let us mention that the most interesting series of papers recently published in this direction are certainly 
those on biliattices due to H.~Priestley and collaborators~\cite{CaCrPr15,CaPr15c,CaPr16}. 

The twist (or `product') representation of bilattices may be viewed as a simplification of
the construction for $\PN$-lattices. Indeed, bilattices having an implication (see~\cite{Bou13,Rivi0xa}) 
correspond, structurally, precisely to those $\PN$-lattices that are representable as twist-structures
$Tw( \A[B],  F, I )$ where $B = F = I$, i.e.~the \emph{full} twist-structures. 
Thus, as far as bilattices as twist-structures are concerned, the parameters $F $ and $I$ 
may be omitted, and we may simply denote by  $Tw( \A[B])$ a twist-structure (over some algebra 
$\A[B]$) viewed as a bilattice.

The algebraic language  in which bilattices are traditionally  presented 
need not include the Nelson implication, but it always 
includes the operations $\land$ and  $\lor$ 
(given, on each  bilattice 
 $Tw( \A[B])$, 
as per Definition~\ref{Def:twi}) as well as an independent pair 
of lattice operations ($\otimes, \oplus$) defined as in a direct product, that is given,
for all $\la a,b \ra , \la c , d \ra \in B \times B$, by:
\begin{align*}
\la a  , b  \ra \otimes \la c , d   \ra  &  := \la a  \land c  , b  \land d  \ra, \\
\la a  , b  \ra \oplus \la c , d   \ra  & := \la a  \lor c  , b  \lor  d \ra.
\end{align*}
These are obviously lattice operations, which determine a second lattice order on $Tw( \A[B])$:
hence the term \emph{bi}-lattice.

Abstractly, a bilattice may be  defined as an algebra 
$\A = \la A; \land, \lor, \otimes, \oplus \ra $ such that 
the reducts 
$\la A; \land, \lor \ra $ and
$\la A; \otimes, \oplus \ra $ are both lattices. 
In practice, the most interesting classes of bilattices arise by imposing some interaction between 
the two  lattice structures, for instance by requiring the lattice operations to be monotone 
with respect to the other lattice order (so-called \emph{interlaced bilattices}) or by 
adding a negation operator $\nnot$
 (also given, on twist-structures, as in Definition~\ref{Def:twi})
that reverses the order corresponding to $\land$ and $\lor$ (as in the case of   $\PN$-lattices)
while preserving the order that corresponds to $\otimes$ and $\oplus$. 

Having the four lattice operations available, and
independently of the existence of an implication in the language, we can
establish twist representation theorems
for various equationally defined classes of bilattices
(see~\cite{CaPr15c,Rivi0xa} for a comprehensive list).
It may be interesting to observe that, in the case of negation-free bilattices,
the equivalence relation $\Xi$ of Definition~\ref{Def:N4} must be replaced by \emph{two}
relations (say, $\Xi_+$ and $\Xi_-$) which can be defined (on each bilattice $\A$) as follows
(this was first noticed in~\cite{BoRi11}): 
$$
\Xi_+ := \{ \la a, b \ra \in A \times A : a \land b = a \otimes b \} 
\qquad
\Xi_- := \{ \la a, b \ra \in A \times A : a \land b = a \oplus b \}.
$$
On a twist-structure, $\Xi_+$ is easily seen to be the relation that identifies all the pairs that share the same second coordinate,
while $\Xi_-$ identifies all the pairs having the same first coordinate. 
Given a bilattice $\A = \zseq{A; \land, \lor, \otimes, \oplus}$,
the twist-structure is then defined on the direct product
$\zseq{A; \land, \lor }/{\Xi_+} \times \zseq{A; \land, \lor }/{\Xi_-}$.
Note that  negation-free bilattices 
are,  in this sense, more general than $\PN$-lattices,
for on the latter the two relations are inter-definable: indeed, we have
$\la a,b \ra  \in \Xi_+ $  iff  $\la \nnot a, \nnot b \ra  \in \Xi_- $ 
and 
$\la a,b \ra  \in \Xi_- $  iff  $\la \nnot a, \nnot b \ra  \in \Xi_+ $.
 
The studies of Priestley and her collaborators~\cite{CaCrPr15,CaPr15c,CaPr16} in the setting of bilattices (though cast in more abstract categorical terms)
resemble the ones we have described with respect to Kalman lattices: 
one starts with a class $\mathsf{K}$ of algebras (typically lattices, perhaps enriched with further operations)
and looks at the resulting class of (enriched) bilattices $Tw(\mathsf{K})$ consisting of (in this case, only the  \emph{full}) twist-structures over members of
$\mathsf{K}$. The main novelty in~\cite{CaPr15c} is the observation that 
the twist representation theorem
(and its 
consequences, 
notably the equivalence of categories) essentially depends only on the existence of certain algebraic terms in the language of $Tw(\mathsf{K})$; the representation results 
can then be proved    
in a uniform categorical way for all the  classes of algebras of interest. 
 
Consider a class of algebras   $\mathsf{K}$ in some algebraic language $\Sigma$. 
To obtain a class $Tw(\mathsf{K})$, it suffices to fix a set $\Gamma$ of pairs of $\Sigma$ terms,
declaring $\Gamma$ to be the algebraic language of $Tw(\mathsf{K})$. For instance, if 
$\mathsf{K}$ is a class of  lattices (so $\Sigma = \{\land, \lor \}$), then the bilattice operation
$\otimes$ may be represented  by the pair of $4$-ary $\Sigma$-terms $\la s_{\otimes}, t_{\otimes} \ra $ defined by
$s_{\otimes}(x_1, x_2, x_3, x_4) : = x_1 \land x_3 $ 
and $t_{\otimes}(x_1, x_2, x_3, x_4) : = x_2 \land x_4 $, 
while the bilattice negation $\nnot$ is given by the pair of binary $\Sigma$-terms
$\la s_{\nnot}, t_{\nnot} \ra$ defined by
$s_{\nnot}(x_1, x_2) : = x_2 $ and
$t_{\nnot}(x_1, x_2) : = x_1 $.

Now, considering the varieties $\mathbb{V}(\mathsf{K})$ and
 $\mathbb{V}(Tw(\mathsf{K}))$ generated  by (respectively) $\mathsf{K}$ and $Tw(\mathsf{K})$,
Cabrer and Priestley  pose the following question: under which conditions can one establish 
 a twist representation theorem
 between  $\mathbb{V}(\mathsf{K})$ and $\mathbb{V}(Tw(\mathsf{K}))$,
  thereby also obtaining an  equivalence between the corresponding algebraic categories?
  The surprisingly simple answer, which constitutes the main new insight of~\cite{CaPr15c},
 is that a sufficient condition for this to hold is the existence of certain terms in $\Gamma$.

Thus,  for instance, one requires the existence of a  $\Gamma$-term $v$ such that,
 for every $\A \in \mathsf{K}$ and for all $a,b,c,d \in A$, one has
 $v ( \la a,b \ra, \la c, d \ra ) = \la a, d \ra $. The role of this term,
 which in the case of bilattices can be given by
 $v (x,y ) := (x \otimes (x \lor y )) \oplus (y \otimes (x \land y) ) $,
 is to merge pairs of elements. One also postulates a
  permuting $\Gamma$-term $s$ satisfying  $s (\la a, b \ra ) = \la b, a \ra $, which
 is obviously the bilattice (or Nelson) negation; 
 see~\cite[Definition~3.1]{CaPr15c} for further details.
 
 If the above-mentioned terms 
 exist,  then a twist representation
 theorem and a categorical equivalence can be established between
 $\mathbb{V}(\mathsf{K})$ and $\mathbb{V}(Tw(\mathsf{K}))$.
As shown in~\cite{CaPr15c},  virtually all  the classes of bilattices 
known  
  to be representable as twist-structures  (as of 2015) possess
 algebraic terms that satisfy the above-mentioned requirements.
Hence, all the known corresponding categorical equivalences -- which had been established by different
\emph{ad hoc} constructions in the previous literature -- 
 are retrieved 
 as special cases of the general results of~\cite{CaPr16,CaPr15c};
 even certain biliattice-like algebras (e.g.~\emph{trilattices}) 
and classes of bilattices not previously known to be representable 
(e.g.~bilattices with a negation and a so-called \emph{conflation} that do not commute with each other) 
are covered by the same approach. 
 
To conclude our overview of abstractions, let us mention an even more general twist construction
introduced 
in a series of recent papers~\cite{thiago2021negation,riviecciotwonegs,
rivieccio2021fragments,rivieccio2021quasi,
rivieccio2019quasi,riviecciosp2021quasi}.
The latter explore the possibility of
representing Nelson-like and bilattice-like algebras endowed with a negation that is not necessarily involutive. 
The motivation for such a study comes from two independent lines of research: on the one hand, the study
of algebraic structures arising as duals of \emph{bitopological spaces}~\cite{jakl2016bitopology};
 on the other hand, an algebraic investigation of the meaning and implications of the 
Nelson identity 
in the setting of (non-involutive) residuated lattices. 

Drawing inspiration from the representation of negation-free bilattices, one may start from a direct product
$L_+ \times L_- $ of two algebras having a lattice reduct (plus perhaps additional operations), requiring 
$L_+$ and $L_- $
to be (not necessarily isomorphic, as in the Nelson and bilattice cases, but) related by two maps
$n \colon L_+ \to L_- $ and 
$p \colon L_- \to L_+ $ that satisfy suitable conditions (e.g.~forming a Galois connection). 
Then, by letting
$\nnot \la a,b\ra : = \la p(b), n(a) \ra $ for all $ \la a,b\ra \in L_+ \times L_- $,
one obtains a negation-like operator satisfying  properties that will depend on what we ask of
the maps $n$ and $p$: as a limit case, by requiring both to be mutually inverse isomorphisms,
we recover the standard twist construction for involutive algebras. 

The four bilattice operations on  $ L_+ \times L_- $ may be defined as in the case of negation-free bilattices, namely,
letting
$$
\la a  , b  \ra \land \la c , d   \ra  := \la a  \land_+ c  , b  \lor_- d  \ra
$$
where $\land_+$ denotes the meet on $L_+$ and $\land_-$ denotes the meet on $L_-$, and similarly with the other
meet ($\otimes$) and the two joins ($\lor$ and $\oplus$). 
Further operations that shuffle the two components may also be easily obtained in this setting, for instance  a Nelson-like implication
can be defined by: 
$$
\la a  , b  \ra \to \la c , d   \ra  := \la a  \to_+ c  , n(a)  \land_- d  \ra
$$
where $\to_+$ denotes the implication on $L_+$.
 
As shown in~\cite{rivieccio2021duality}, 
this straightforward generalisation  proved to be sufficiently flexible to represent a number of classes of algebras related to 
bilattice and Nelson logics (and even,  as a limit case, Heyting algebras themselves). This generalisation, whose limits are currently being investigated, has already enabled us to resolve various questions related to these algebras.




%
%
%
%
%
%

\subsection{Priestley duality for bounded $\PN$-lattices} \label{ss:pri}

As mentioned earlier, an $\PN$-lattice
$\alg{A} = \zseq{A; \land, \lor, \to, \nnot}$ may be viewed as a De Morgan 
lattice $\zseq{A; \land, \lor, \nnot}$ structurally enriched with an implication operation.
If $\alg{A}$ is bounded (by $\bot$ and $\top$), then $\zseq{A; \land, \lor, \nnot, \bot, \top}$ 
is a bounded De Morgan lattice (i.e.~a {De Morgan algebra});
and a De Morgan algebra, in turn,  is  just a bounded distributive lattice
enriched with an involutive dual automorphism. 
This observation inspired Cornish and Fowler~\cite{CoFo77,CoFo79} 
to introduce a Priestley-style topological duality for  De Morgan algebras which is
a straightforward extension of standard Priestley duality
for bounded distributive lattices~\cite{priestley1970representation}. 
In the Cornish-Fowler duality, the  topological spaces that form the category dual to De Morgan lattices (\emph{De Morgan spaces})
are just Priestley spaces equipped with an extra unary function that is an order-reversing homeomorphism.
It is also easy to specialise the duality in order to obtain spaces corresponding to \emph{Kleene algebras}, i.e.~De Morgan 
algebras that additionally satisfy the following \emph{Kleene identity}:
$$
x \land \nnot x \approx x \land \nnot x \land (y \lor \nnot y).  
$$
The latter observation is relevant in the Nelson setting, for it is well known that 
the class of implication-free reducts of $\N$-lattices is (precisely) the class of  Kleene algebras. 

The above considerations suggest that the Cornish-Fowler approach to duality might be successfully extended
to $\PN$- and $\N$-lattices, obtaining a duality for these two classes of algebras that relies on the Cornish-Fowler
duality just as the latter relied on Priestley duality for distributive lattices. 
This strategy was indeed  successfully  pursued by Odintsov, who
 introduced in~\cite{Od10} 
 a Priestley-style duality for $\PN$-lattices based on the Cornish-Fowler duality for De Morgan lattices, 
 thereby also obtaining, by specialisation, a
 duality for $\N$-lattices (the latter class of algebras had however been studied from a duality
 point of view already by Sendlewski~\cite{sendlewski1984topological} and Cignoli~\cite{Cign86}).

Before we get into the details of Odintsov's duality for $\PN$-lattices, let us mention an alternative 
approach which may be employed in a duality-theoretic study of these algebras. 
As mentioned in the preceding subsection, the twist representation of $\PN$-lattices can be easily 
formulated as a (covariant) categorical equivalence  between $\PN$-lattices and a category having for objects 
tuples of type
$( \A[B],  F, I )$, where $\A[B]$ is a Brouwerian lattice, $F$ is a dense lattice filter 
and $I$  is a lattice ideal 
of $\A[B]$. This suggests that, by introducing a  category of topological spaces 
dual to the category
of tuples of type $( \A[B],  F, I )$, we would also obtain 
a duality for $\PN$-lattices simply by composing the relevant functors.
This is the approach applied to $\PN$-lattices in~\cite{Jans13,JaRi14} and extended to more general algebras in~\cite{rivieccio2021duality} 
(see also the earlier papers~\cite{JuRi12,MoPiSlVo00} 
on dualities for various classes of bilattices). 

As Odintsov's duality relied on the Cornish-Fowler duality for De Morgan lattices,
so the approach developed in~\cite{Jans13,JaRi14}
relies on Esakia duality for Heyting algebras, which is essentially a restriction (rather than an extension) 
of Priestley duality for distributive lattices.
Indeed, the fact that Esakia duality concerns  more familiar algebraic structures and spaces may be regarded as an advantage of the  twist-structure-based  approach to duality
over  Odintsov's proposal. 

Let us now recall the main definitions and results involved in the duality for $\PN$-lattices
introduced in~\cite{Od10}. 
The original Priestley duality concerns the category~$\DL$ of bounded distributive
lattices and bounded lattice homomorphisms. Usually, Priestley-style
dualities deal with  algebras that have a \emph{bounded} lattice reduct. The presence of the bounds is not a necessary requirement, but the corresponding spaces 
turn out to be more natural than those one would obtain by allowing for unbounded algebras.
In the same spirit, Odintsov~\cite{Od10} also develops his duality only for bounded $\PN$-lattices; for details on how to deal with (Esakia duality for) unbounded algebras, see~\cite{Jans13,JaRi14}.

To every bounded distributive lattice~$\Al[L]$, one associates the set
$X(\Al[L])$ of its prime filters. 
On $X(\Al[L])$ one has the \emph{Priestley topology}~$\tau$,
generated by the sets $\phi(a):=\{x\in X(\Al[L]):a\in x\}$ and
$\phi'(a):=\{x\in X(\Al[L]):a\not\in x\}$, and the inclusion relation
between prime filters as an order. The resulting ordered topological
spaces are called \emph{Priestley spaces}.
Abstractly, a \emph{Priestley space} is defined as a compact ordered topological space
$\la X, \tau, \leq \ra $ such that, for all $x,y \in X$, if $x \not \leq y$, 
then there is a clopen up-set $U \subseteq X$ with $x \in U$ and $y \notin U$.
This condition is known as the \emph{Priestley separation axiom}.
It follows that $\la X, \tau \ra$ is a Stone space.

A homomorphism~$h$ between bounded distributive lattices $\Al[L]$
and~$\Al[L]'$ gives rise to a function~$X(h):X(\Al[L]')\to X(\Al[L])$, defined
by $X(h)(x')=h^{-1}[x']$, that is continuous and order-preserving.
Taking functions with these properties (called \emph{Priestley functions})
as morphisms
between
Priestley spaces one obtains the category~$\PS$, and $X$ is now
readily recognised as a contravariant functor from $\DL$ to~$\PS$.

For a functor in the opposite direction, one associates to every
Priestley space $\X=\la X,\tau,\leq \ra$ the set~$L(\X)$ of clopen
up-sets. This is a bounded distributive lattice with respect to the
set-theoretic operations $\cap, \cup, \emptyset$, and~$X$. To a
Priestley function $f:\X\to\X'$ one associates the function~$L(f)$, given
by $L(f)(U')=f^{-1}[U']$, which is 
a bounded lattice
homomorphism from $L(\X')$ to~$L(\X)$. 
So $L$~constitutes
a contravariant functor from $\PS$ to~$\DL$.

The two functors are adjoint to each other with the units given by:
$$\begin{array}{llll}
\Phi_{\Al[L]}\colon\Al[L] \to L(X(\Al[L]))&&\Phi_{\Al[L]}(a) = \{ x \in X(\Al[L]) : a \in x \}\\
\Psi_{\X}\colon\X \to X(L(\X)) &&\Psi_{\X}(x) = \{ U \in L(\X) : x \in U \}
\end{array}$$
One shows that these are the components of a natural transformation from
the identity functor on~$\DL$ to $L\circ X$, and from the identity
functor on~$\PS$ to $X \circ L$, respectively, satisfying the required
diagrams for an adjunction. In particular, they are morphisms in their
respective categories. Furthermore, they are isomorphisms and
thus the central result of Priestley duality is obtained.

\begin{theorem}
 The
  category $\DL$ of bounded distributive lattices with homomorphisms
  is dually equivalent to the category~$\PS$  
 of Priestley spaces with Priestley functions  via the above-defined functors $X$ and $L$.
\end{theorem}

As mentioned earlier, in order to obtain a duality for bounded De Morgan lattices,
 it suffices to find a way to represent the De Morgan negation on Priestley spaces. 
This was accomplished by Cornish and Fowler~\cite{CoFo77,CoFo79} as follows. 
Given a bounded De Morgan lattice
 $\Al[L] = \la L; \land, \lor, \nnot, 0, 1 \ra$
and a prime filter $ x \in X(\Al[L])$, define:
$$\nnot x := \{ a \in L :  \nnot a \in x \}.$$

For any $x \in X(\Al[L])$, we have that $ \nnot x $ is a prime ideal. Hence, defining:
$$
g ( x ) : = L \setminus  \nnot x
$$
we have that $g(x)$ is a prime filter (this construction dates back, at least,  to~\cite{Bial57}
). It is then easy to check that the map
$g \colon X(\Al[L]) \to  X(\Al[L]) $ is an order-reversing involution on the poset $\la X(\Al[L]), \subseteq \ra $, i.e., that 
$g^2 = id_{X(\Al[L]) }$ and,
for all $x, y \in X(\Al[L]) $, we have:
$$
x \subseteq y \quad  \textrm{ iff }  \quad g(y) \subseteq g(x).
$$

If we endow $X (\Al[L])$ with the Priestley topology,
we have that the structure
$$
\la X (\Al[L]), \tau, \subseteq, g \ra
$$
is a \emph{De Morgan space}, which is abstractly defined as follows:
a \emph{De Morgan space} is a structure  $\X = \la X, \tau, \leq, g \ra$ 
where
 $\la X, \tau, \leq \ra$ is a Priestley space 
 and
$g \colon \X \to  \X $ is an order-reversing homeomorphism such that $g^2 = id_{\X}$.

Conversely, given a De Morgan space $\X = \la X, \tau, \leq, g \ra$,
one defines on the Priestley dual $\la L(\X); \cap, \cup, \emptyset, X \ra$
an operation~$\nnot$ as follows. For any $U \in L(\X)$,
let 
$$g[U] := \{ g(x) : x \in U\}$$ 
and
$$
\nnot U : = {X} - {g[U]}.
$$%
Then $\la L(\X); \cap, \cup, \nnot, \emptyset, X \ra$ is a  bounded 
De Morgan lattice. One also shows that the unit maps~$\Phi_{\Al[L]}$
preserve the negation, so they are bounded De Morgan lattice
isomorphisms.

On the spatial side one defines a \emph{De Morgan function} $f
\colon \X \to \X'$ to be a Priestley function for which $f
\circ g = g' \circ f$. One then shows that the unit maps~$\Psi_{\X}$
are in fact De Morgan functions and hence De Morgan isomorphisms
(since the extra structure~$g$ can be viewed as a unary algebraic
operation).

\begin{theorem}
The category $\DM $ of bounded De Morgan
  lattices with homomorphisms is dually equivalent to the category $\DMS$ of
  De Morgan spaces with De Morgan functions
via the above-defined functors $X$ and $L$.
\end{theorem}

De Morgan  duality specialises to one between the full subcategories of
bounded Kleene lattices 
and \emph{Kleene spaces}, defined as follows. Given a De Morgan
space $\la X, \tau, \leq, g \ra$, consider the sets
$$
X^+  : = \{ x \in X  : x \leq g (x) \}, \qquad \quad
X^-  : = \{ x \in X  : g (x) \leq   x \}.
$$
A \emph{Kleene space} is then defined as a De Morgan space $\la X,
\tau, \leq, g \ra$ such that $X = X^+ \cup X^-$.
specialising Kleene algebras further to Boolean algebras one obtains
the classical Stone duality by insisting that $g$ be the identity map.

In order to extend the Cornish-Fowler duality to $\PN$-lattices, it only remains to take care of the Nelson implication operator.
 
Let $\A = \la A; \land, \lor, \to, \nnot \ra$ be an $\PN$-lattice,
and let $\X (\A)$ be the De Morgan space corresponding to the 
$\to$-free reduct of $\A$.
We are going to use the quasiorder $\pr$ of Definition~\ref{Def:N4}
to distinguish  the prime filters of $\A$ between two further kinds.

A filter $x \in \X (\A)$ is a \emph{special filter of the first kind} 
(\emph{sffk} for short) if, for all $a,b \in A$,
$a \in x$ and $a \preceq b$ imply $b \in x$.
Dually,
$x \in \X (\A)$ is a \emph{special filter of the second kind} 
(\emph{sfsk} for short) if, for all $a,b \in A$,
$a \in F$ and $\neg b \preceq \neg a $ imply $b \in F$.

Every prime filter in $\X (\A)$ is either a \emph{sffk} or a \emph{sfsk} (it may be both).
That is, we have
$$
X(\Al) = X^1 (\Al) \cup X^2 (\Al),
$$
where $ X^1 (\Al)$ denotes the set of prime filters of the first kind and  $ X^2 (\Al)$ the set of prime filters of the second kind.
On the other hand, as mentioned earlier, in general we do not have 
$$
X(\Al) = X^+ (\Al) \cup X^- (\Al)
$$
unless the De Morgan reduct of $\Al$ is a Kleene algebra. The following properties, however, hold for
arbitrary $\PN$-lattices. 

\begin{proposition}
\label{Prop:OdDua}
Let $\Al$ be an $\PN$-lattice. 
\begin{enumerate}
\item  $\la X (\Al), \tau, \subseteq, g \ra$ is a De Morgan space.
\item $g[X^1 (\Al)] = X^2 (\Al)$.
\item $X(\Al) = X^1 (\Al) \cup X^2 (\Al)$ and $X^1 (\Al) \cap X^2 (\Al) = X^+ (\Al) \cap X^- (\Al)$.
\item $ X^1 (\Al)$ is closed in $\tau$ and $ X^1 (\Al)$ with the induced topology is an 
Esakia space\footnote{An \emph{Esakia space}
(also known as \emph{Heyting space})
is a
Priestley space such that, for any open set $O$, the downset $O \!\!
\downarrow$ is also open~\cite{Es74, Pr84}.}.
\item For all $x \in X^1 (\Al) $ and $y \in X^2 (\Al) $, if  $x \subseteq y$, then $x \in X^+ (\Al)  $, $y \in X^- (\Al)  $ and there exists 
$z \in X (\Al)$  such that $x, g[y] \subseteq z \subseteq g[x], y$.
\item  For all $x \in X^2 (\Al) $ and $y \in X^1 (\Al) $, if $x \subseteq y$, then $x \in X^+ (\Al)  $, $y \in X^- (\Al)  $  and
$x \subseteq g[y]$.
\end{enumerate}

\end{proposition}

Odintsov~\cite{Od10}  defines an \emph{N4-space} to be a tuple 
$\X
= \la X, X^1, \tau, \leq, g \ra $ such that
properties (1) to (6) of Proposition~\ref{Prop:OdDua} 
(where $X^2 : = g[X^1]$)
are satisfied.

Given an N4-space $\la X, X^1, \tau, \leq, g \ra $, the algebra $\la
L(\X); \cap, \cup, \nnot, \emptyset, X \ra$ defined as before is a bounded De
Morgan lattice. On this lattice, one defines an implication operation $\to 
$ as follows: for any $U, V \in L(\X)$,
$$
U \to V \ : = \  \big ( \, X^1 \, \ba \, ((U \ba V ) \cap X^1 ) \!\! \downarrow  \big  ) \cup \big ( \,  X^2 \, \ba \, (g[U] \ba V ) \, \big ).
$$
Then
$\la L(\X) ; \cap, \cup, \to, \nnot, \emptyset, X \ra$
is a bounded $\PN$-lattice.

To complete the picture, one defines an \emph{N4-function} to be a
function~$f$ between N4-spaces $\la X, X^1, \tau, \leq, g \ra$ and
$\la Y, Y^1, \tau', \leq', g' \ra$ that satisfies the following:
\begin{enumerate}[label=(\roman*)]
\item $f$ is a De Morgan function from $\la X, \tau, \leq, g \ra$ to $\la.
  Y, \tau', \leq', g' \ra$.
\item $f[X^1] \subseteq Y^1$.
\item $f \colon X^1 \to Y$ is an Esakia 
 function, i.e., for any open $O \in \tau'$,
$$
f^{-1}[(O \cap Y^1)\!\downarrow] \cap X^1 = (f^{-1}[O \cap Y^1])\!\downarrow \cap X^1.
$$
\end{enumerate}


\begin{theorem}\cite[Thm.~5.4]{Od10}
\label{th:caeqnfnfs}
The category $\NF$
of bounded $\PN$-lattices with homomorphisms is dually
  equivalent to the category $\NFS$ of N4-spaces with N4-functions  
via the above-defined functors $X$ and $L$.
\end{theorem}

By specialising Theorem~\ref{th:caeqnfnfs} one obtains a duality for N3-lattices: as expected, the corresponding spaces
are precisely those N4-spaces of type $
\la X, X^1, \tau, \leq, g \ra $ that satisfy $ X^+ \cup X^- = X $.


We conclude the section by giving a few more details on the alternative approach to duality  (called ``two-sorted duality'' in~\cite{rivieccio2021duality}) 
that was mentioned earlier. 
We restrict our attention to \emph{bounded} 
$\PN$-lattices, corresponding to triples $(\A[H],  F, I)$
such that $\A[H]$ is a Heyting algebra (rather than a Brouwerian lattice), but 
the result is established for the unbounded case as well 
in~\cite{JaRi14}. 
This duality 
also builds on Priestley and Esakia but (unlike Odintsov's) it does not rely on the results of
Cornish and Fowler. 

The category $\NF$ of $\PN$-lattices is defined as before; moreover, 
we have a category $\Tw$ whose objects are triples $( \A[H],  F, I )$
such that $\A[H]$ is a Heyting algebra, $F \subseteq H$ is a dense filter of 
$\A[H]$ and $I \subseteq H$  a lattice ideal. 
A morphism  between objects $( \A[H_1],  F_1, I_1 ), ( \A[H_2],  F_2, I_2 ) \in \Tw$ is
a  Heyting algebras homomorphism $h \colon H_1 \to H_2$ that preserves the filter and the ideal,
that is, 
we require
$h[F_1] \subseteq F_2$ and $h[I_1] \subseteq I_2$.

One can then define functors $N \colon \Tw \to \NF$ and $T \colon \NF \to \Tw$
that establish a covariant equivalence between both categories. 
So far as objects go, given $( \A[H],  F, I ) \in \Tw$, we follow Definition~\ref{Def:twi}, letting
 $N(\A[H],  F, I) : = Tw( \A[B],  F, I )$.
Conversely, given a bounded  $\PN$-lattice $\A \in \NF$, we let 
$T(\A) := ( B(\A),  F(\A), I(\A) )$ be given as in 
Theorem~\ref{Prop:OdRep}.
As for  morphisms, given 
objects $( \A[H_1],  F_1, I_1 ), ( \A[H_2],  F_2, I_2 ) \in \Tw$ and a $\Tw$-morphism $h$ between them,
we let $N(h) \colon Tw ( \A[H_1],  F_1, I_1 ) \to Tw ( \A[H_2],  F_2, I_2 )$ be given by
$N(h) (\la a,b \ra) := \la h(a), h(b) \ra$.
Conversely, to a bounded $\PN$-lattice homomorphism
$g \colon \A_1 \to \A_2$ we associate the  $\Tw$-morphism
$T(g) \colon ( B(\A_1),  F(\A_1), I(\A_1) ) \to ( B(\A_2),  F(\A_2), I(\A_2) )$
given by $T(g) (a /{\Xi_1}) := g(a) /{\Xi_2} $,
where $\Xi_1$ and $\Xi_2$ are the equivalence relations
defined on $\A_1$ and $\A_2$ according to
Definition~\ref{Def:N4}.

\begin{theorem}[\cite{JaRi14},~Thm.~2.6]
The category $\NF$
of bounded $\PN$-lattices with homomorphisms is 
  equivalent to the category $\Tw$ 
via the above-defined functors $N$ and $T$.
\end{theorem}

Now it suffices to invoke Esakia's results 
on Heyting algebras to obtain a duality between the category $\Tw$ (and, hence, also $\NF$)
and a category of enriched topological spaces defined as follows. 

Given a Priestley space  $\la X, \tau, \leq \ra$,
denote by $\max(X)$  the set of points of $X$ that are maximal
with respect to $\leq$. 
Following~\cite[Def.~3.3]{JaRi14}, we define an \emph{NE-space} as a tuple
$\X =  \la X, \tau, \leq, C, O \ra$ such that $\la X, \tau, \leq \ra$ is an Esakia space,
$C \subseteq X$ is a closed set such that $\max(X) \subseteq C$ and 
$O \subseteq X$ is an open up-set. 
To an NE-space $\X = \la X, \tau, \leq, C, O \ra$ we associate the triple
$L(\X) := (L(X), F_C, I_O)$, where $L(X)$ is the Heyting algebra of clopen up-sets
corresponding to the Esakia space $\la X, \tau, \leq \ra$,
$F_C := \{ U \in L(\X) :  C \subseteq U \}$
and
$I_O := \{ U \in L(\X) :  U \subseteq O \}$.
Since $F_C$ is a dense filter and $I_O$  an ideal of $L(\X)$,
we have that $(L(\X), F_C, I_O) \in \Tw$.
Conversely, to each  $( \A[H],  F, I ) \in \Tw$ we associate the
NE-space 
$X( \A[H],  F, I ) : = \la X (\Al[L]), \tau, \subseteq, C_F, O_I \ra
$
such that
$ \la X (\Al[L]), \tau, \subseteq \ra$ is the Esakia space dual
 to the Heyting algebra $\A[H]$,
$C_F : = \bigcap \{ \phi (a) : a \in F \}$
and 
$O_I : = \bigcup \{ \phi (a) : a \in I \}$.

A morphism  between NE-spaces
$ \la X_1, \tau_1, \leq_1, C_1, O_1 \ra$
and
$ \la X_2, \tau_2, \leq_2, C_2, O_2 \ra$
is defined as an Esakia function $f \colon X_1 \to X_2 $
that satisfies 
$f[C_1 ] \subseteq C_2$
and 
$f^{-1}[C_2] \subseteq O_1$.
The dual $L(f)$ of a morphism of NE-spaces
is defined as in Priestley duality, and is readily verified
to be a morphism in the category $\Tw$.
Conversely, the dual $X(h)$ of a $\Tw$-morphism $h$, once again defined as in Priestley duality,
is a morphism of NE-spaces. We thus reach the following equivalence results.

\begin{theorem}\cite[Thm.~3.11]{JaRi14} The category $\NE$ of NE-spaces 
is dually equivalent to the category $\Tw$ 
via the above-defined functors $L$ and $X$.
\end{theorem}

\begin{theorem}\cite[Cor.~3.12]{JaRi14}
The category $\NF$
of bounded $\PN$-lattices with homomorphisms is dually
  equivalent to the category 
  $\NE$ of NE-spaces 
  via the   functors  $X \circ T$ and $N \circ L $.
\end{theorem}

Having laid down the details 
of both alternative dualities
for bounded
$\PN$-lattices, we can close the section with a  comparative remark. 
As mentioned earlier, an advantage of the two-sorted duality
is that it allows one to work within the more well-known setting of Esakia spaces,
without any need to refer to Cornish and Fowler's results on De Morgan algebras;
while a drawback, one might argue, is that the duality is indirect in that it relies essentially
on the covariant equivalence between $\NF$
and  $\Tw$, and thus on the twist representation for $\PN$-lattices. However,
an inspection of the proofs contained in~\cite{Od10} reveals
that Odintsov's duality also relies heavily (if more covertly) on the twist representation.

Another feature that makes 
Odintsov's proposal  somewhat unusual among
Priestley-type dualities is the following. 
As is well known, it is possible to give a logical reading of Priestley-type dualities:
since the points of the space dual to an algebra (say, a Heyting or a Boolean algebra) are logical filters, they 
correspond to (prime) theories 
of (intutionistic, or classical) logic; the same continues to hold in the setting of the Cornish-Fowler duality for De Morgan lattices if we view the latter as an algebraic
semantics for the Belnap-Dunn logic~\cite{Font97}. This correspondence is not preserved in Odintsov's duality, because
the space dual to a  $\PN$-lattice consists of  filters of the first kind (which correspond to theories of Nelson's paraconsistent logic) as well as filters of the second kind (which do not). 
In contrast, the correspondence between points and logical theories is  restored by the duality of~\cite{JaRi14}, because
the poset of (prime) filters on the Heyting algebra $H(\A)$ associated to a given $\PN$-lattice $\A$ is isomorphic
to the poset of the (prime) filters of the first kind on $\A$. 

%
%
%
%
%
%
%

\section{Nelson and residuated lattices} \label{sec:n3}

\subsection{$\N$- and $\PN$-lattices as residuated algebras}
 \label{ss:nasres}

As mentioned  in Subsection~\ref{subsec:nel},
an $\N$-lattice may be  alternatively defined as: 
\begin{enumerate}[label=(\roman*)]
\item 
an involutive CIBRL $\A = \zseq{A; \land, \lor, *, \imp, 1, 0 }$
that further satisfies~\eqref{Eqn:Nel};
\item  an $\PN$-lattice  $\alg{A} = \zseq{A; \land, \lor, \to, \nnot}$, given as per Definition~\ref{Def:N4},
that further satisfies any of the properties in Proposition~\ref{Prop:NelWeakening} (this is Rasiowa's original definition as rephrased by Odintsov).
\end{enumerate}

The above correspondence amounts to a term equivalence between the two  classes of algebras. 
 To go from (i) to (ii), it suffices to let:
\begin{align*}
x \to y & := x \imp (x \imp y) \\ 
\nnot x & : = x \imp 0.
\end{align*}
For the other direction, one defines (letting $\abs{x} := x \imp x$):
\begin{align*}
x \imp y & := \lmb(x \to y \rmb) \land \lmb(\nnot y \to  \nnot x) \\
x * y & : = \nnot (x \imp \nnot y ) \\
1 & := \abs{x}  \\
0 & := \nnot 1. 
\end{align*}

As mentioned earlier, Spinks and Veroff~\cite{Spin18} have established a similar result for general $\PN$-lattices.
In this endeavour, the first challenge must have been to come up with a suitable alternative presentation for  $\PN$-lattices
as residuated structures. Note, for instance, that one cannot hope to be dealing with  \emph{residuated lattices} in the sense of~\cite{GaJiKoOn07}, for we know that an $\PN$-lattice may not have any term definable algebraic constant,
hence there may be no neutral element for  the $*$ operation.
Furthermore, it is not hard to check that 
the definition
$x \to y := x \imp (x \imp y)$ cannot possibly work in the setting of  general $\PN$-lattices; a more
involved term needs to be used,  significantly complicating the proof of term equivalence. Simpler terms and a simpler proof may be obtained if we allow ourselves to expand the language of $\PN$-lattices with a constant 
(to be interpreted as the neutral element of the monoid operation). Such a strategy, pursued in~\cite{Busa09},  is  unsatisfactory in that we are restricting our attention to a proper subclass -- which is not even a sub(quasi)variety -- of  $\PN$-lattices.

Spink and Veroff's solution to the first issue was to introduce the class of 
\emph{dimorphic paraconsistent Nelson RW-algebras}. Here,  ``RW'' is a reference
to Brady's contraction-free relevance logic RW~\cite{Brad91}, and in fact 
Spinks and Veroff's term equivalence result entails that paraconsistent Nelson logic may be viewed as  an axiomatic extension of RW
(for further background on the involved terminology, we refer the interested reader to~\cite{Spin18}).
We may  gradually approach their definition through the following observations.
Starting with an $\PN$-lattice  $\alg{A} = \zseq{A; \land, \lor, \to, \nnot}$  
given as per Definition~\ref{Def:N4}, let
$$
x \imp y := \lmb(x \to y \rmb) \land \lmb(\nnot y \to  \nnot x)
$$ 
and 
$$x * y  : = \nnot (x \imp \nnot y ).$$ 
In this way one  may be reassured  to verify  that 
the pair $(*, \imp)$ satisfies the residuation property introduced in Subsection~\ref{subsec:nel}.
One may furthermore show that:
\begin{enumerate}[label=(\roman*)]
\item The algebra  $ \zseq{A; \land, \lor, *, \imp }$ is the $1$-free subreduct of a commutative residuated lattice.
\item The identity $x \imp y \approx \nnot y \imp \nnot x$ (well known to hold on all involutive residuated lattices) is satisfied.
\end{enumerate}
The class of algebras satisfying the first of the above items has been characterised as the variety of
\emph{adjunctive residuated lattice-ordered semigroups}. These are algebras $ \zseq{A; \land, \lor, *, \imp }$
of type
$\zseq{2,2,2,2}$
such that $ \zseq{A; \land, \lor }$ is a lattice (with an order $\leq$), 
the operation $*$ is associative and commutative,
and,
for all $a, b, c \in A$:
\begin{enumerate}[label=(\roman*)]
\item if $a \leq b$, then $a * c \leq b * c$ 
\hfill (compatibility)
\item 
$a * b \leq c$ iff $a \leq b \imp c$ \hfill (residuation)
\item $(\abs{a} \land \abs{b}) \imp c \leq c$ \hfill (adjunction).
\end{enumerate}
It is known that the above conditions may be replaced by identities; hence adjunctive residuated lattice-ordered semigroups form an equational class.

Summing up the previous observations, we have that every $\PN$-lattice 
$\alg{A} = \zseq{A; \land, \lor, \to, \nnot}$ 
has a term definable reduct which is the subreduct of a (distributive) residuated lattice;
moreover, the operation $\nnot$ is a \emph{compatible} involution, i.e.~we have
 $a \imp b = \nnot b \imp \nnot a$ and $\nnot \nnot a = a $ for all $a,b \in A$.
Even all these properties taken together, however, are still not enough to provide an alternative 
 characterisation of $\PN$-lattices. To achieve this, we need one more identity, and some ingenuity. 
Defining:
\begin{equation}
\label{eq:pcn}
x \to' y := (x \land \abs{y}) \imp ((x \land \abs{y}) \imp y)
\end{equation}
%
we   observe that every $\PN$-lattice satisfies 
the identity $  x \to' y \approx x \to y $,
suggesting one   may define in this  way the weak implication
in terms of the strong one. We further note that the following
identity, called
\emph{Internal Weakening} in~\cite[p.~332]{Spin18}, is also satisfied:
$$
(x *y ) \to' x \approx \abs{(x *y ) \to' x}.
$$



Let us define $ \A = \zseq{A; \land, \lor, *, \imp, \nnot }$ to be an algebra of type
$\zseq{2, 2, 2, 2, 1}$ such that:
\begin{enumerate}[label=(\roman*)]
\item the reduct $ \zseq{A; \land, \lor, *, \imp }$ is a distributive and  adjunctive residuated lattice-ordered semigroup,
\item the operation  $\nnot$ is a {compatible} involution,
\item letting 
$$x \to y := (x \land \abs{y}) \imp ((x \land \abs{y}) \imp y),
$$
we have:
$$
x \imp y \approx \lmb(x \to y \rmb) \land \lmb(\nnot y \to  \nnot x) 
\quad  \text{and} \quad (x *y ) \to x \approx \abs{(x *y ) \to x}.
$$
\end{enumerate}
Spinks and Veroff~\cite{Spin18} dub such a structure a
\emph{dimorphic paraconsistent Nelson RW-algebra,}
and proceed to show that, from every such algebra $ \A = \zseq{A; \land, \lor, *, \imp, \nnot }$,
we may obtain an $\PN$-lattice $ \zseq{A; \land, \lor, \to, \nnot }$ by the above prescriptions. 
Since every $\PN$-lattice also gives rise to a dimorphic paraconsistent Nelson RW-algebra as described above,
we have mutually inverse translations that establish a term equivalence, as anticipated.

Spinks and Veroff's result has several consequences, both logical and algebraic. 
On a logical level, it allows us to
view the paraconsistent system of Nelson as a contraction-free relevance logic,
and thus to place it within the taxonomy of well-known relevance systems. 
We have observed earlier that Brady's logic RW may be viewed as a weakening of paraconsistent Nelson logic;   we may now
further verify that several other systems can be obtained as axiomatic extensions of Nelson's -- for instance
the system BN of Slaney, 
the three-valued relevance logic with mingle RM3, 
the  nilpotent minimum logic NM of Esteva and Godo 
(see~\cite[p.~328]{Spin18} for further examples and background). Obviously the original logic of Nelson, as well as all its strengthenings (e.g.~the three-valued logic of {\L}ukasiewicz), are also  extensions of paraconsistent Nelson. 

On an algebraic level, the term equivalence  allows one to apply or adapt to the setting of ($\N$- and) $\PN$-lattices results that are known to hold for residuated structures. We have, for instance,
that the variety of $\PN$-lattices is congruence-permutable~\cite[Thm.~4.24]{Spin18}, therefore also arithmetical,
and enjoys equationally definable principal congruences. 
Spinks and Veroff also obtain further information on definable term functions: they show, for instance, that $\PN$-lattices
possess  a ternary and a quaternary deductive term in the sense of Blok and Pigozzi~\cite{Blok94a},
and identify the discriminator varieties of $\PN$-lattices. All these results obviously specialise to $\N$-lattices as well.

\subsection{Further characterisations of $\N$-lattices}
 \label{ss:fuca}


We have mentioned in Subsection~\ref{ss:pcl} that, in contrast to the case of $\PN$-lattices, for $\N$-lattices
a concise and elegant equational axiomatisation is available. This is due to 
Diana Brignole~\cite{Brig69}, who built on earlier work by Antonio Monteiro. In Brignole's presentation,
an  $\N$-lattice is defined as an algebra $\A = \zseq{A; \land, \lor, \to, \nnot, 1 }$ 
of type $\zseq{2,2,2,1, 0}$ that satisfies the following identities:

\begin{enumerate}[label=(N3.\roman*), leftmargin=*, itemsep=2pt]
\item \label{Itm:N1eq} $x \lor 1 \approx 1 $
 \item \label{Itm:N2eq} $x \land (x \lor y) \approx x $
\item \label{Itm:N3eq} $x \land (y \lor z) \approx (x \land y) \lor (x \land z)$

\item \label{Itm:N4eq}  $\nnot \nnot x \approx x $ 
\item \label{Itm:N5eq}  $\nnot ( x \land y) \approx \nnot x  \lor \nnot y $
\item \label{Itm:N6eq} $x \land \nnot x  \leq y \lor \nnot y $
\item \label{Itm:N7eq} $x \to x \approx 1 $
\item \label{Itm:N8eq} $\nnot x \lor y \leq x \to y $ 
\item \label{Itm:N9eq} $x \land (\nnot x \lor y) \approx x \land (x \to y)$
 \item \label{Itm:N10eq}  $x \to (y \land z) \approx (x \to y) \land (x \to z)$
\item \label{Itm:N11eq} $ x \to (y \to z) \approx (x \land y) \to z $.
\end{enumerate}

The preceding identities entail that, upon letting $0 := \nnot 1 $, 
the structure
$\zseq{A; \land, \lor, \nnot, 0,  1 }$ is a Kleene algebra; thus, conversely, a Nelson algebra
may be defined as a Kleene algebra structurally expanded with a binary connective $\to$
(of weak implication)
satisfying  identities~\ref{Itm:N7eq} to~\ref{Itm:N11eq}.

A number of  new  characterisations of $\N$-lattices in terms of the strong implication have been presented in the paper~\cite{Nasc1z},
which is probably  the latest
study of $\N$-lattices  from the standpoint of residuated lattices. We conclude the section by
recalling a few of them. 

Let $\A = \zseq{A; \land, \lor, *, \imp, 1, 0 }$ be a 
 \emph{compatibly involutive}
CIBRL, i.e.~one that satisfies the identities $x \imp y \approx \nnot y \imp \nnot x$
and $\nnot \nnot x \approx x $ (where, as before
$\nnot x : = x \imp 0$). Abbreviating $x^2 := x * x$ and $x^3 := x * x * x$,
we can write the important  \emph{3-potency} identity: 
$$
x^3 \approx x^2. 
$$
Let us also abbreviate
$x \to y := x^2 \imp y$.
Consider, as in Subsection~\ref{ss:pcl}, a relation $\preceq$ defined, for all $a,b \in A$, by
      $a \preceq b$ iff $a \to b = 1$. Note that, on a CIBRL $\A$, we could equivalently define       $a \preceq b$ iff 
      $a^2 \leq b$. If $\A$ is moreover {3-potent}, 
      then  $\preceq $ is a quasiorder which may be used to obtain a Heyting algebra quotient (cf.~\cite[p.~2309]{Nasc1z} and Definition~\ref{Def:N4} above).
      
The first characterisation we present allows us to view the property of ``being Nelson'' as an order-theoretic one.      

\begin{proposition}\cite[Thm.~4.11]{Nasc1z} \label{Prop:eqord}
Let $\A$ be a 3-potent compatibly involutive CIBRL (with a lattice order $\leq$), and let  $\preceq$ be given as above.
Further define $a \preceq' b $ iff $\nnot b \preceq \nnot a$ for all $a,b \in A$. The following are equivalent:
\begin{enumerate}[label={\rm (\alph*)}]
\item $\A$ is an $\N$-lattice.
\item The relation $\preceq \cap \preceq'$ is a partial order  coinciding with $\leq$.
\item The relation $\preceq \cap \preceq'$ is a partial order.
\item The relation $\preceq \cap \preceq'$ is antisymmetric. 
\end{enumerate}

\end{proposition}

The following result  is instead purely algebraic, and gives us two equivalent alternatives to~\eqref{Eqn:Nel}.

\begin{proposition}\label{Prop:synt}
Let $\A$ be a compatibly involutive CIBRL. The following are equivalent:
\begin{enumerate} [label={\rm (\alph*)}]
\item $\A$ is an $\N$-lattice.
\item $\A$ satisfies  $ x* y \approx (x^2 * y) \lor (x * y^2) $ \hfill \cite[Prop.~4.1]{Nasc1z}.
\item $\A$ satisfies   $  x \approx x^2 \lor (x \land \nnot x) $\hfill \cite[Thm.~6.1]{Nasc1z}. 
\end{enumerate}
\end{proposition}


We stress that the equivalences stated in Proposition~\ref{Prop:synt}
rely essentially on the involutivity of the negation; it is for instance
easy to observe that, in a non-necessarily involutive setting, the identity $x \approx x^2 \lor (x \land \nnot x) $
follows from, but does not entail,~\eqref{Eqn:Nel}~\cite[Lemma~5.1]{WNM}.

The next two characterisations we present have once more an order-theoretic flavour.  

\begin{proposition}\cite[Prop.~4.4]{Nasc1z}
 \label{Prop:eq3pot}
Let $\A$ be a compatibly involutive CIBRL. The following are equivalent:
\begin{enumerate} [label={\rm (\alph*)}]
\item $\A$ is an $\N$-lattice.
\item $\A$ is 3-potent and satisfies the quasi-identity: if $x^2 \approx y^2$ and $(\nnot x)^2 \approx (\nnot y)^2$, then
$x \approx y$. 
\end{enumerate}
\end{proposition}

%

The preceding characterisation suggests 
that~\eqref{Eqn:Nel}
not only entails (as is easily verified) the above quasi-identity, but also 3-potency.
On the other hand,
 in the absence of 3-potency the  quasi-identity alone does not seem to be sufficient to ensure that a compatibly involutive CIBRL be an $\N$-lattice.

We know that we may drop the assumption of 3-potency if we slightly modify the above quasi-identity, as shown
by the following result. 

\begin{proposition}\cite[Prop.~4.10]{Nasc1z}
 \label{Prop:eqnopot}
Let $\A$ be a compatibly involutive CIBRL. The following are equivalent:
\begin{enumerate} [label={\rm (\alph*)}]
\item $\A$ is an $\N$-lattice.
\item   $\A$  satisfies the quasi-identity: if $x^2 \leq y$ and $(\nnot y)^2 \leq \nnot x$, then
$x \leq y$. 
\end{enumerate}
\end{proposition}



In the light of the Priestley duality for $\N$-lattices expounded in Subsection~\ref{ss:pri},
the quasi-identities considered above may be read as follows: whenever two elements $a,b $
of an $\N$-lattice are distinct, then either $a^2 \neq b^2$, in which case there is a special (prime) filter of the first kind 
containing (say) $a$ but not containing $b$, 
or $(\nnot a)^2 \neq (\nnot b)^2$, in which case there is a special (prime) filter of the second kind containing 
(say) $a$ but not  $b$. Given the well-known  correspondence between filters and congruences on residuated lattices, it is natural to ask whether the previous properties might also be rephrased in congruence-theoretic
terms. This leads us to consider the notion of \emph{congruence orderability} introduced in~\cite{Nasc1z}.

Given any algebra $\A$, let us denote by
$\theta(a,b)$ the least congruence on $\A$ that relates $a$ and $b$. 
Assuming $\A$ has some constant $c$ in its language,
 we shall say that $\A$ is 
\emph{$c$-congruence orderable} if, for all $a,b \in A$,
$$
\theta(a,c) = \theta (b,c) \qquad \text{ implies } \qquad a = b.
$$
The term ``orderability'' may be justified by noting that, on a $c$-congruence orderable algebra $\A$, the relation $\leq$ defined by
$$
a \leq b  \qquad
 \text{ if and only if  } \qquad \theta(a,c) \subseteq \theta (b,c)
$$
is indeed a partial order. The notion of congruence orderability, which can be traced back to the work of B\"uchi and Owens~\cite{BuO}, has been formally introduced and extensively explored in~\cite{Idzi09}. 

In the setting of  lattices that are either bounded or residuated, it is natural  to investigate 
the $c$-congruence orderable algebras 
for $c \in \{ 0, 1 \}$. 
In many cases, 
$c$-congruence orderability turns out to be captured by identities:
we have, for instance, that the $1$-congruence orderable CIBRLs (no involutivity being assumed) are precisely the Heyting
algebras~\cite[Prop.~5.8]{Nasc1z}, while the $0$-congruence orderable ones are precisely
the Boolean algebras~\cite[Prop.~5.9]{Nasc1z}. If we restrict our attention
to involutive CIBRLs, then the  $1$-congruence orderable members coincide with the
$0$-congruence orderable ones, which are  the Boolean algebras~\cite[Prop.~5.10]{Nasc1z}.

The previous considerations entail that $\N$-lattices are, in general, neither $1$- nor $0$-congruence orderable; but one may ask whether  a similar property may be employed to characterise $\N$-lattices, for instance
among involutive CIBRLs. This question led to the introduction of the notion of $(c,d)$-congruence orderability in~\cite{Nasc1z}, which we now proceed to formally define. 

Let  $\A$ be an algebra having constants $c$ and $d$ in its language\footnote{To capture the intuition that $c$ and $d$ should normally \emph{not} be interpreted as the same element of $\A$, we further require them to be \emph{residually distinct}, which means that $\theta(c,d) = A \times A$.}. We say that $\A$ is
\emph{$(c,d)$-congruence orderable} when, for all $a,b \in A$, we have that
$$
\theta(a,c) = \theta (b,c) \quad 
\text{ and } \quad \theta(a,d) = \theta (b,d) \qquad
\text{ imply } \qquad a = b.
$$
It is clear that the above notion may be applied to various classes of bounded lattices, residuated lattices (having a designated constant $0$) etc. With regards to $\N$-lattices, we may observe the following:
\begin{enumerate}[label=(\roman*)]
\item On any $\N$-lattice $\A$ and for all $a,b \in A$, we have $\theta (a, 1) =\theta (b,1)$ 
if and only if the elements $a$ and $b$ generate the same special filter of the first kind (that is, $a$ and $b$ are indistinguishable by means of special filters of the first kind).
\item Symmetrically, we have that $\theta (a, 0) =\theta (b,0)$ 
if and only if $a$ and $b$ generate the same special filter of the second kind.
\item By the previous considerations, every $\N$-lattice is $(0,1)$-congruence orderable~\cite[Cor.~5.4]{Nasc1z}.
\item In fact, \emph{$\N$-lattices are precisely the class of $(0,1)$-congruence orderable involutive CIBRLs}~\cite[Thm.~5.11]{Nasc1z}.
\end{enumerate}

 $\N$-lattices further satisfy an even stronger congruence-theoretic property, namely, they are \emph{$(0,1)$-Fregean}. 
According to Idziak et al.~\cite{Idzi09},  an algebra $\A$ having a constant $c$ is \emph{$c$-Fregean} if 
$\A$ is both $c$-congruence orderable and \emph{$c$-regular}, the latter meaning that  congruences on $\A$
are determined by their $c$-coset (for all congruences $\theta_1, \theta_2$, one has  $\theta_1 = \theta_2$ whenever $c / \theta_1 = c / \theta_2$). $\N$-lattices are $1$-regular as well as $0$-regular, but they are obviously 
neither $1$- nor $0$-Fregean. Following~\cite[p.~2313]{Nasc1z}, we say that an algebra $\A$
having residually distinct constants $c$ and $d$ is  \emph{$(0,1)$-Fregean} when
$\A$ is $(0,1)$-congruence orderable and both $c$- and $d$-regular.  
According to this definition,
$\N$-lattices
are $(0,1)$-Fregean~\cite[Thm.~5.11]{Nasc1z}; in fact, we have the following more informative result. 

\begin{proposition}\cite[Cor.~7.2]{Nasc1z}
 \label{Prop:frege}
Let $\A$ be a compatibly involutive CIBRL. The following are equivalent:
\begin{enumerate} [label={\rm (\alph*)}]
\item $\A$ is an $\N$-lattice.
\item   $\A$  is $(0,1)$-congruence orderable.
\item   $\A$  is $(0,1)$-Fregean.
\end{enumerate}
\end{proposition}

As the reader will have guessed, the properties of being  $(c,d)$-congruence orderable (or Fregean)
may be studied in more general settings
(see e.g.~\cite[Sec.~6]{Ri20}). 
For instance, as observed in~\cite[Example~5.6]{Nasc1z}, the 
$(0,1)$-congruence orderable  De Morgan algebras
are precisely the Kleene algebras. Indeed, the study of the notion of $(0,1)$-congruence orderability in the setting of 
non-necessarily involutive CIBRLs led to the introduction of \emph{quasi-Nelson algebras}, 
which can also be  defined equationally as the non-necessarily involutive of CIBRLs which satisfy~\eqref{Eqn:Nel}. 

\section{Rough sets and their algebras}
\label{sec:rough}

\subsection{Algebras with pseudocomplementation}
\label{ss:pseudo}

We begin this section by recalling some facts about pseudocomplemented lattices and 
their relationship with
three-valued {\L}u\-ka\-sie\-wicz algebras and semi-simple Nelson algebras. Also some essential results related to algebraic and
completely distributive lattices are recalled. These considerations are found, for instance, in the books \cite{Balb74, Boic91,Dave02, Grat78}.
For more specific results, a reference will be given.

In a lattice $\alg{L} = \zseq{L; \land, \lor, 0 }$ with  least element $0$, an element $a \in L$
is said to have a \emph{pseudocomplement} if there exists an element $a^* \in L$
having the property that for any $b \in L$, $a \wedge b = 0$ iff $  b \leq a^*$.
Note that if a pseudocomplement exists, it is unique. The lattice $\alg{L}$
itself is called \emph{pseudocomplemented}, if every element of $L$ 
has a pseudocomplement. Every pseudocomplemented lattice has a greatest element $1 = 0^*$, which can then safely be added to the algebraic
signature.

Let $\alg{L} = \zseq{L;\vee,\wedge,{^*},0,1}$ be a pseudocomplemented lattice. The following hold for every $a,b \in L$:
\begin{enumerate}[label=(\roman*)]
 \item $a \leq b$ implies $b^* \leq a^*$;
 \item the map $a \mapsto a^{**}$ is a closure operator;
 \item $a^* = a^{***}$;
 \item $(a \vee b)^* = a^* \wedge b^*$;
 \item $(a \wedge b)^* \geq a^* \vee b^*$.
\end{enumerate}

A distributive  pseudocomplemented  lattice $\alg{L}$ is called a \emph{Stone algebra} if $\alg{L}$ satisfies the \emph{Stone identity}:
\[ x^* \vee x^{**} \approx 1. \] 
In a Stone algebra $\alg{L}$, the following identities also hold:
\[ (x \wedge y)^* \approx x^* \vee y^* \quad \text{and} \quad (x \vee y)^{**}\approx x^{**} \vee y^{**}. \]

By dualising, we get the concepts of 
\emph{dual pseudocomplement}, 
\emph{dual pseudocomplemented lattice}, 
and \emph{dual Stone algebra}.
A \emph{double pseudocomplemented lattice} is a pseudocomplemented lattice which is also a dual
pseudocomplemented lattice. Similarly, a \emph{double Stone algebra} is a Stone algebra which is also a dual Stone algebra. 
Every double Stone algebra satisfies $x^* \leq x^+$, where $^+$ denotes the dual pseudocomplement operation.
We say that a double Stone algebra is \emph{regular} if it satisfies the quasi-identity
\begin{equation} \label{eq:condM} \tag{M}
x^* \approx y^* \text{ and } x^+ \approx y^+ \text{ imply } x\approx y. 
\end{equation}
Here `regularity' refers to `congruence-regularity'. 
An algebra is called \emph{congruence-regular} if every congruence is determined by any class of it: two congruences are necessarily equal when they have a class in common.
J.~Varlet has proved in \cite{Varl72} that double 
pseudocomplemented lattices satisfying \eqref{eq:condM} are
exactly the congruence-regular ones.

It is known (see e.g. \cite{Katrinak73}) that in any double Stone algebra, the `regularity condition' \eqref{eq:condM} is equivalent 
to $x \wedge x^+ \leq y \vee y^*$.

\begin{example} \label{Exa:DoubleStone}
Let $\alg{B} = \zseq{B; \vee, \wedge, {'}, 0, 1 }$  be a Boolean algebra. Denote
\[ B^{[2]} := \{ (a,b) \in B^2 \mid a \leq b \} .\]
Now $B^{[2]}$ is a regular double Stone algebra with the operations:
\[
\begin{gathered}
  (a,b) \vee (c,d) := (a \vee c, b \vee d), \qquad  (a,b) \wedge (c,d) := (a \wedge c, b \wedge d), \\
  (a,b)^* := (b',b'), \qquad (a,b)^+ := (a',a').
\end{gathered}
\]
More generally, if $F$ is a lattice filter of $\alg{B}$, then by \cite{Katrinak74},
\[ (B,F) := \{ (a,b) \in B^2 \mid a \leq b \text{ and } a \vee b' \in F\} \]
forms a regular double Stone algebra in which the operations are defined as in $B^{[2]}$.
\end{example}

The reader may have noticed a similarity between the algebra defined in the preceding example
and the twist-structure construction for Nelson algebras considered earlier (see e.g.~Definition~\ref{Def:twi} in Subsection~\ref{ss:tw}).

Following A.~Monteiro~\cite{Monteiro63c}, we can define a \emph{three-valued {\L}ukasiewicz algebra} as an algebra 
$\alg{L} = \zseq{L; \land, \lor, {\nnot}, {\triangledown}, 0, 1 }$ 
such that $\zseq{L; \land, \lor, {\nnot}, 0, 1 }$ is a De Morgan algebra and ${\triangledown}$ is an unary operation, called the \emph{possibility operator},
satisfying: 
\begin{enumerate}[label=(\roman*)]
\item  ${\sim} x \vee {\triangledown} x \approx 1$,
\item  ${\sim x}  \wedge x \approx {\sim} x   \wedge {\triangledown} x$,
\item  ${\triangledown}(x \wedge y) \approx {\triangledown}x \wedge {\triangledown} y$.
\end{enumerate}
Let us recall from \cite{Monteiro63c} that the map $x \mapsto {\triangledown} x$ is a closure operator. 
In addition,
\[ {\triangledown} 0 \approx 0 \quad \text{and} \quad {\triangledown}(x \vee y) \approx  {\triangledown}x \vee  {\triangledown} y.\]
The \emph{necessity operator} is defined by 
${\vartriangle} x := {\sim} {\triangledown} {\sim}  x$.
The operation $\vartriangle$ is the dual operator of $\triangledown$. Also $\vartriangle$ and $\triangledown$ have some mutual connections, for instance:
\[ {\vartriangle}\!{\triangledown} x \approx  {\triangledown} x \qquad \mbox{and} \qquad {\triangledown}\!{\vartriangle} x \approx  {\vartriangle} x .\]
Three-valued {\L}ukasiewicz algebras satisfy the following \emph{determination principle} introduced by Gr.~C.~Moisil:
\[  {\vartriangle}x \approx  {\vartriangle}y \quad \mbox{and} \quad {\triangledown}x \approx {\triangledown}y \quad \mbox{imply} \quad x \approx y. \]

Any regular double Stone algebra $\alg{L} = \zseq{L; \land, \lor, {^*},{^+},0,1 }$ defines a three-valued {\L}ukasiewicz algebra 
$\zseq{L; \land, \lor, {\nnot}, {\triangledown} ,0,1 }$ by setting
\[
  {\triangledown} x := x^{**} \qquad \mbox{and} \qquad {\nnot} x := x^* \vee (x \wedge x^+) .
\]
Similarly, each  three-valued {\L}ukasiewicz algebra $\alg{L} = \zseq{L; \land, \lor, {\nnot}, {\triangledown}, 0, 1}$ defines a double
Stone algebra $\zseq{L; \land, \lor, {^*},{^+},0,1 }$ by
\[
 x^* := {\nnot}{\triangledown}x \qquad \mbox{and} \qquad x^+ := {\triangledown}{\nnot}x.
\]
This double Stone algebra is regular, that is, $^*$ and $^+$ satisfy \eqref{eq:condM}. In addition, these pseudocomplements
determine each other by 
\[
{\sim} x^{\ast} \approx ({\sim} x)^{+} \ \mbox{ and } \ {\sim} x^{+} \approx ({\sim} x)^{\ast}.
\]
This correspondence between  regular double Stone algebras and  three-valued {\L}ukasiewicz algebras is one-to-one.

In universal algebra, an algebra $\alg{A}$ is called \emph{simple} if it has only the identity and the universal relations as its congruences. 
$\alg{A}$ is \emph{semisimple} whenever  $\alg{A}$ is isomorphic to a subdirect product of simple algebras. It is known that a Nelson algebra 
 $\alg{A} = \zseq{A; \land, \lor, \to, {\nnot}, 0, 1}$ 
is semisimple iff 
$x \vee  (x \to 0) \approx 1$. Every three-valued {\L}ukasiewicz algebra determines a semisimple Nelson algebra upon
setting $x \to y := {\triangledown}{\nnot}x \vee y$. Similarly, each semisimple Nelson algebra induces a three-valued {\L}ukasiewicz algebra upon setting
${\triangledown} x := {\nnot} x \to 0$; see e.g.~\cite{Monteiro1980}.

Jan~{\L}ukasiewicz's three-valued logic {\L}$_3$ allows propositions to have the values $0$, $\frac{1}{2}$, and $1$ where the third logical value
$\frac{1}{2}$ may be interpreted as ``possibility'' (see \cite{Luka70d}). {\L}ukasiewicz defined in {\L}$_3$ the operations $\sim$ and $\to$ by
\[ {\sim} a := 1 - a \qquad \text{and} \qquad a \to b := \min \{ 1, 1 - a + b \} . \]
In 1930's, Mordchaj Wajsberg presented an axiomatisation for {\L}$_3$. An algebraic counterpart of this axiomatisation is called a
\emph{Wajsberg algebra} $\alg{L} = \zseq{L; \to, {\sim}, 1 }$ defined by the identities~\cite{Font84}: 
\begin{enumerate}[label = {\rm (W\arabic*)}] 
\item $1 \to x \approx x$
\item $(x \to y) \to ((y \to z) \to (x \to y)) \approx 1$
\item $((x \to y) \to z) \approx ((y \to x) \to x)$
\item $({\sim} x \to {\sim} y) \to (y \to x) \approx 1$.
\end{enumerate}

A three-valued {\L}ukasiewicz algebra  $\alg{L} = \zseq{L; \land, \lor, {\sim}, {\triangledown}, 0, 1 }$ defines a Wajsberg algebra by setting
\[ x \to y := (\nabla {\nnot} x \vee y) \wedge (\nabla y \vee {\nnot} x).\]
Similarly, a Wajsberg algebra $\alg{L} = \zseq{L; \to, {\sim}, 1 }$ defines a three-valued {\L}ukasiewicz algebra in which
\begin{gather*}
 x \vee y := (x \to y) \to y, \quad x \wedge y := {\sim}({\sim} x \vee {\sim} y), \quad \nabla x := {\sim} x \to x, \quad 0 := {\sim }1.
\end{gather*}

C.C.~Chang introduced \emph{MV-algebras} in \cite{Chang58}. It is proved in \cite{Font84} that any MV-algebra defines a Wajsberg algebra, and
every Wajsberg algebra determies an MV-algebra. In addition, D.~Mundici
proved in \cite{Mundici86} that MV-algebras are categorically equivalent to bounded commutative
BCK-algebras introduced by S.~Tanaka in \cite{Tanaka75}. Thus, we can consider the following algebras as equivalent:
\begin{itemize}
\item three-valued {\L}ukasiewicz algebras,
\item regular double Stone algebras,
\item semisimple Nelson algebras,
\item Wajsberg algebras,
\item MV-algebras,
\item bounded commutative BCK-algebras.
\end{itemize}
Note that since the pseudocomplement and the dual pseudocomplement are uniquely determined by the order $\leq$ of a lattice $(L,\leq)$ forming a
regular Stone algebra, this means that the order-structure of $L$ fully determines the unique operations of the above-listed algebras.

As is well known, a lattice can be defined either as an algebra $\zseq{L; \vee, \wedge}$ or as an ordered set $\zseq{L; \leq}$. 
Focusing on the latter view, we now recall some facts on a complete lattice $\alg{L} = \zseq{L; \leq}$. We need these properties
especially when considering the structure of the complete lattice of rough sets. A complete lattice $\alg{L}$ is
\emph{completely distributive} if for any doubly indexed subset $\{a_{i,\,j}\}_{i \in I, \, j \in J}$ of $L$, we have:
\[
\bigwedge_{i \in I} \Big ( \bigvee_{j \in J} a_{i,\,j} \Big ) = 
\bigvee_{ f \colon I \to J} \Big ( \bigwedge_{i \in I} a_{i, \, f(i) } \Big ), \]
that is, any meet of joins may be converted into the join of all possible elements obtained by taking the meet over $i \in I$ of
elements $a_{i,\,k}$\/, where $k$ depends on $i$. 
The power set lattice $\zseq{\wp(U);\subseteq}$ of  a set $U$ is a well-known completely distributive lattice. The set 
$\wp(U) \times \wp(U)$ can be ordered coordinatewise by $\subseteq$, and the joins and meets are the coordinatewise unions and intersections, respectively.
Therefore, $\zseq{\wp(U) \times \wp(U);\leq}$ is a completely distributive lattice. Also a complete sublattice of a completely distributive lattice is
clearly completely distributive. 

Let $\alg{L}$ be a complete lattice and let $k \in L$. The element $k$ is said to be $\emph{compact}$ if for every subset $S$ of $L$,
\[ 
k \leq \bigvee S \mbox{ \ implies \ } k \leq \bigvee F \mbox{ \ for some finite subset $F$ of $S$}.
\]
The set of compact elements of $L$ is denoted $\mathcal{K}(\alg{L})$.
A complete lattice $\alg{L}$ is said to be \emph{algebraic} if, for each $a \in L$,
\[
  a = \bigvee \{ k \in \mathcal{K}(\alg{L}) \mid k \leq a \} .
\]
The powerset $\wp(U)$ forms an algebraic lattice in which finite subsets of $U$ are the compact elements. A product of algebraic lattices is algebraic
(see \cite[Proposition~\mbox{I-4.12}]{gierz2003}), which implies that $\wp(U) \times \wp(U)$ is algebraic. In addition, each complete sublattice of an
algebraic lattice is algebraic \cite[Exercise~{7.7}]{Dave02}. 

An element $j \in L$ of a complete lattice $\alg{L}$ is called \emph{completely join-irreducible} if $j = \bigvee S$ implies $j \in S$ for every subset $S$ of $L$.
Note that the least element $0 \in L$ is not completely join-irreducible. The set of completely join-irreducible elements of $L$ is denoted by 
$\mathcal{J}(\alg{L})$, or simply by $\mathcal{J}$ if there is no danger of confusion.
A complete lattice $\alg{L}$ is \emph{spatial}%
if for each $a \in L$,
\[ a = \bigvee \{ j \in \mathcal{J} \mid j \leq a \}. \]

If $\alg{L}$ is an algebraic lattice, then its completely join-irreducible elements are compact. Let the lattice $\alg{L}$
be both algebraic and spatial. Since any compact element can be written as a finite join and any finite join of compact elements is
compact, the compact elements of $\alg{L}$ are exactly those that can be written as a finite join of completely join-irreducible elements.

An \emph{Alexandrov topology} \cite{Alex37,Birk37} $\mathcal{T}$ on a set $U$ is a topology in which also intersections of
open sets are open, or equivalently, every point $x \in U$ has the \emph{least neighbourhood} $N(x) \in \mathcal{T}$.
For an Alexandrov topology $\mathcal{T}$, the least neighbourhood  of $x$ is $N(x)= \bigcap \{ B \in \mathcal{T} \mid x \in B \}$. 

A complete lattice $\alg{L}$ satisfies the \emph{join-infinite distributive law} if for any $S \subseteq L$ and $a \in L$,
\begin{equation*}\label{Eq:JID} \tag{JID}
a \wedge \big ( \bigvee S \big ) = \bigvee \{ a \wedge b \mid b \in S \}.
\end{equation*}
  
The following conditions for a complete lattice $\alg{L}$ are equivalent; see \cite{Dave02}, for instance:
\begin{enumerate}[label = (Alex\arabic*),leftmargin=*, itemsep=2pt] 
\label{alg_dist_spatial}
\item $\alg{L}$ is isomorphic to the lattice of all open sets in
an Alexandrov topology;
\item $\alg{L}$ is algebraic and completely distributive;
\item $\alg{L}$ is distributive and doubly algebraic (i.e.\@ both algebraic and dually algebraic);
\item $\alg{L}$ is spatial and satisfies (JID).
\end{enumerate}

Suppose $\alg{L} = \zseq{L;\vee,\wedge}$ is a lattice and $a,b \in L$. If there is a greatest element $c \in L$
such that $a \wedge c \leq b$, then this element $c$ is called the \emph{relative pseudocomplement of $a$ with respect to $b$} and is denoted
by $a \Rightarrow b$. If $a \Rightarrow b$ exists, then it is unique. A \emph{Brouwerian lattice} $\alg{L}$ is a lattice in which $a \Rightarrow b$
exists for all $a, b$ in $L$. Every Brouwerian lattice $\alg{L}$ has a greatest element $1$.

Indeed, for every $a \in L$, $a \Rightarrow a$ is the greatest element of $L$. As noted in \cite{Birk67},
any Brouwerian lattice is distributive. A complete lattice $\alg{L} = \zseq{L;\leq}$ is Brouwerian if and only if it satisfies (JID). In such a case,
for all $a,b \in L$,
\[ a \Rightarrow b = \bigvee \{ c \in L \mid a \wedge c  \leq b\} .\]
A \emph{Heyting algebra} is a Brouwerian lattice with least element $0$. Therefore, a complete lattice is a
Heyting algebra if and only if it is a Brouwerian lattice. In particular, any finite distributive lattice is a Brouwerian lattice
and a Heyting algebra. As in Section~2, we may regard Brouwerian algebras as universal algebras $\alg{L} = \zseq{L; \land, \lor, \imp, 1 }$ 
of type $\zseq{2, 2, 2, 0}$ and Heyting algebras as  algebras $\alg{L} = \zseq{L; \land, \lor, \imp, 0, 1 }$ of type
$\zseq{2, 2, 2, 0, 0}$.

It is known \cite{Moisil63, Monteiro1980} that every three-valued {\L}ukasiewicz algebra forms a Heyting algebra where
\begin{equation} \label{eq:heyting_luka}
x \Rightarrow y := {\vartriangle}{\sim} x \vee y \vee ({\triangledown}{\sim} x \wedge {\triangledown} y).
\end{equation}

Let $\alg{P} = \zseq{P; \leq}$ be an ordered set with a least element $0$. An element $a \in P$  is an \emph{atom} if $a$ covers $0$, that is, $0 \prec a$.
The ordered set $\alg{P}$ is \emph{atomic} if every element $b > 0$ has an atom below it, and $\alg{P}$ is \emph{atomistic}, if every element of $P$ is the join of atoms below it.
For a Boolean algebra $\alg{B}$, the following are equivalent:
\begin{enumerate}[label = {(\roman*)}] 
\item $\alg{B}$ is atomic,
\item $\alg{B}$ is atomistic,
\item $(B, \leq)$ is a completely distributive lattice,
\item $\alg{B}$ is isomorphic to $\zseq{\wp(U); \cup, \cap, {^c}, \emptyset, U}$ for some set $U$.
\end{enumerate}

\subsection{Rough sets defined by equivalences} \label{ss:rough_equivalence}

The basic idea of rough set theory is that knowledge about objects is represented by indistinguishability relations. Indistinguishability
relations are originally \cite{Pawlak82} assumed to be equivalences interpreted so that two objects are equivalent if we cannot distinguish them by their
properties. We may observe objects only by the accuracy given by an indistinguishability relation. This means that  our ability to
distinguish objects is blurred -- we cannot distinguish individual objects, only their equivalence classes. In this section,
we recall from the literature the main facts about rough set algebras defined by equivalences.

Let $E$ be an equivalence on a set $U$. For each $x \in U$, we denote by $x / E$ the equivalence class of $x$.
For any subset $X$ of $U$, let
\[ X^\DOWN  :=  \{ x \in U \mid x/E \subseteq X\} 
\quad \text{and} \quad
 X^\UP  :=  \{ x \in U \mid X \cap x/E  \neq \emptyset\}. \] 
The sets $X^\DOWN$ and $X^\UP$ are called the \emph{lower} and the \emph{upper approximation} of $X$, respectively.
The set $B(X) := X^\UP \setminus X^\DOWN$ is the \emph{boundary} of $X$.

The above definitions mean that $x \in X^\UP$ if there is an element in $X$ to which $x$ is $E$-related. Similarly,
$x\in X^\DOWN$ if all the elements to which $x$ is $E$-related are in $X$. Furthermore, $x \in B(X)$ if both in $X$ and outside
$X$ there are elements which cannot be distinguished from $x$. If $B(X) = \emptyset$ for some $X \subseteq U$, 
this means that for any object $x \in U$, we can with certainty decide whether $x \in X$ just by knowing $x$ `modulo $E$'.

A set $X$ is called \emph{definable} if $X^\UP = X^\DOWN$. We denote by $\mathit{Def}(E)$ the set of
all definable sets. When there is no danger of confusion, we denote $\mathit{Def}(E)$ simply by $\mathit{Def}$.
It is clear that $X \in \mathit{Def} $ iff $ B(X) = \emptyset$.

A \emph{complete Boolean algebra} is a Boolean algebra whose underlying lattice is complete. 
Given a complete Boolean algebra $\alg{B}$ and a subalgebra $\alg{A} \leq \alg{B}$, 
we say that  $\alg{A}$ is a \emph{complete Boolean subalgebra} of $\alg{B}$ if $\zseq{A; \leq}$ is a complete sublattice
of $\zseq{B; \leq}$.
The family $\mathit{Def}$ forms a complete Boolean  subalgebra of $\zseq{\wp(U),\cup,\cap,^c,\emptyset,U}$.
Thus, the complete lattice $\alg{Def} = \zseq{\mathit{Def};\subseteq}$ is completely distributive. Moreover, $\alg{Def}$ is atomic and atomistic,
and its atoms are the $E$-classes. In particular, approximations are definable and for any $X \subseteq U$,
\[ X^\DOWN = \bigcup \{ A \in \mathit{Def} \mid A \subseteq X\} \quad \text{and} \quad
X^\UP = \bigcap \{ A \in \mathit{Def} \mid X \subseteq A\}.\]

Define a binary relation $\equiv$ on $\wp(U)$ by $X \equiv Y$ if $X^\DOWN = Y^\DOWN$ and $X^\UP = Y^\UP$. The equivalence classes
of $\equiv$ are called \emph{rough sets} \cite{Pawlak82}. By definition, each rough set is uniquely determined by the pair
$(X^\DOWN, X^\UP)$, where $X$ is a member of that rough set. Let us denote by $\mathit{RS}(E)$ the collection $\{ (X^\DOWN,X^\UP) \mid X \subseteq U\}$ of
all rough sets determined by $E$.  We will denote $\mathit{RS}(E)$ simply by $\mathit{RS}$ if there is no danger of confusion.
The set $\mathit{RS}$ is ordered naturally by $(X^\DOWN,X^\UP) \leq (Y^\DOWN,Y^\UP)$ if $X^\DOWN \subseteq Y^\DOWN$ and $X^\UP \subseteq Y^\UP$.
The pair $(\emptyset,\emptyset)$ is the least and $(U,U)$ is the greatest rough set.

An algebraic study of rough sets was started by T.~B.~Iwi{\'n}ski in \cite{Iwinski87}, where he considered
the system $\mathit{Def}^{[2]} = \{ (A,B) \in \mathit{Def}^2 \mid A \subseteq B\}$. As noted in Example~\ref{Exa:DoubleStone},
$\mathit{Def}^{[2]}$ forms a regular double Stone lattice such that
\[
\begin{gathered}
  (A,B) \vee (C,D) = (A \cup C, B \cup D), \qquad (A,B) \wedge (C,D) = (A \cap C, B \cap D), \\
  (A,B)^* = (B^c,B^c), \qquad  (A,B)^+ = (A^c,A^c),
\end{gathered}
  \]
where $X^c$ denotes the \emph{complement} $U \setminus X$ of the set $X \subseteq U$. In addition, Iwi{\'n}ski associated with each pair
$(A,B) \in \mathit{Def}^{[2]}$ its \emph{rough complement} ${\sim}(A,B) = (B^c,A^c)$ and noted that
with respect to $\sim$, $\mathit{Def}^{[2]}$ forms a De~Morgan algebra.

Not all pairs $(A,B) \in \mathit{Def}^{[2]}$ form a rough set. The set $\mathit{RS}$ was characterised as a subset of $\mathit{Def}^{[2]}$ by P.~Pagliani
in \cite{Pagl96}. Denote by $\mathcal{S}$ the set of elements $x \in U$ such that $x/E = \{x\}$.
These elements are called \emph{singletons}. By the definition of rough set approximations, $x \in X^\DOWN \iff x \in X^\UP$ for
all for $X \subseteq U$ and $x \in \mathcal{S}$. Because $X^\DOWN \subseteq X^\UP$, this means that $\mathcal{S} \cap (X^\UP \setminus X^\DOWN) = \emptyset$ and
$\mathcal{S} \subseteq  (X^\UP \setminus X^\DOWN)^c = X^\DOWN \cup X^{\UP c}$.

\begin{theorem} \label{Thm:Pagliani1}
For any equivalence,
\begin{equation} \label{eq:RS_characterization}
\mathit{RS} = \{ (A,B) \in \mathit{Def}^2 \mid A \subseteq B \ \text{ and } \mathcal{S} \subseteq A \cup B^c\} 
\end{equation}
\end{theorem}

The $\subseteq$-part of \eqref{eq:RS_characterization} follows from the above observations. Let us consider the $\supseteq$-part in detail. 
Suppose  $(A,B) \in \mathit{Def}^2$ is such that $A \subseteq B$ and $\mathcal{S} \subseteq A \cup B^c$. Now both $A$ and $B \setminus A$ belong to
$\mathit{Def}$ and thus they are unions of some $E$-classes. By our assumption, $B \setminus A$ is a union of such $E$-classes that contain at least
two elements each. Using the \emph{Axiom of Choice}, we can pick one element from each of the  $E$-classes forming $B \setminus A$.
Let us denote by $C$ the set of these elements. 
Now $(A,B)$ is the rough set of $A \cup C$, because $(A \cup C)^\DOWN = A$ and $(A \cup C)^\UP = B$.

That $\mathit{RS}$ is a complete sublattice of $\wp(U) \times \wp(U)$ is 
not obvious, because it is not clear that for any 
$\{({X_i}^\DOWN,{X_i}^\UP) \mid i \in I\} \subseteq \mathit{RS}$,
the pairs
$({\bigcap}_i {X_i}^\DOWN, {\bigcap}_i {X_i}^\UP)$
and 
$({\bigcup}_i {X_i}^\DOWN, {\bigcup}_i {X_i}^\UP)$
really belong to $\mathit{RS}$. But they do, because
for each $x \in U \setminus \mathcal{S}$,
we pick one element from $x/E$. Denote by $C$ the set of all these elements. For $X \subseteq U$, we denote $X^\alpha := X^\DOWN \cup (X^\UP \cap C)$.
Clearly, $X \equiv X^\alpha$. The family $\{ X^\alpha \mid X \subseteq U\}$ is a complete sublattice of $\wp(U)$ isomorphic to $\mathit{RS}$.
By applying this observation, J.~Pomyka{\l}a and J.~A.~Pomyka{\l}a \cite{PomPom88} proved the following result.

\begin{theorem} \label{Thm:CompleteStoneLattice}
For any equivalence, $\mathit{RS}$ forms a complete Stone lattice such that for all $(A,B) \in \mathit{RS}$ and
$\{(A_i,B_i)\}_{i \in I} \subseteq \mathit{RS}$,
\begin{align*}
\bigvee_{i \in I} ( A_i, B_i ) &= \big ( \bigcup_{i \in I} A_i, \bigcup_{i \in I} B_i \big ) \\
\bigwedge_{i \in I} ( A_i, B_i ) & = \big ( \bigcap_{i \in I} A_i, \bigcap_{i \in I} B_i \big ) , \\
(A,B)^* &= (B^c,B^c).
\end{align*}
\end{theorem}

It should be noted that Z.~Bonikowski presented a similar observations
in \cite{Bonikowski92}.  He used so-called minimal lower samples to prove that  $\mathit{RS}$ forms a complete atomic Stone algebra. 
Also M.~Gehrke and E.~Walker considered
the above representative $X^\alpha$ in \cite{GeWa92}. These results and
more are considered in \cite{Banerjee2004}, which gives a summary of 
the work done related to algebras and rough sets defined by equivalences.

Theorem~\ref{Thm:CompleteStoneLattice} was improved by S.~D.~Comer \cite{Come93} by stating 
that $\mathit{RS}$ forms a regular double Stone algebra and
the dual pseudocomplement $(A,B)^+$ of $(A,B)$ is $(A^c,A^c)$. 
A semantical study of these pseudocomplementation operations is given
in \cite{Kumar2017}.
Note also that because $\mathit{RS}$ is a complete sublattice of
$\wp(U) \times \wp(U)$, $\mathit{RS}$ is algebraic and completely distributive.

\begin{example} \label{Exa:Iwo}
The family ${\uparrow}\mathcal{S} = \{ X \in \mathit{Def} \mid \mathcal{S} \subseteq X\}$ is a lattice filter of $\mathit{Def}$.
Using \eqref{eq:RS_characterization}, I.~D{\"u}ntsch \cite{Dunt97} noted that $\mathit{RS}$ coincides with
$\{ (A,B) \in \mathit{Def}^2 \mid A \subseteq B \ \text{ and } \ A \cup B^c \in {\uparrow}\mathcal{S} \}$.
From this it also follows that $\mathit{RS}$ is a regular double Stone algebra as shown in Example~\ref{Exa:DoubleStone}.
\end{example}

In  \cite{GeWa92}, M.~Gehrke and E.~Walker proved that $RS$ is order-isomorphic to $\mathbf{2}^I \times \mathbf{3}^K$, where $I$ is the set of singleton 
$E$-classes and $K$ is the set of nonsingleton equivalence classes of $E$. 
Comer \cite{Comer95} gave the following representation theorem of complete atomic regular double Stone algebras in terms of rough sets.

\begin{theorem} \label{thm:Comer}
Each  complete atomic regular double Stone algebra is isomorphic to some rough set double Stone algebra determined by an equivalence.
\end{theorem}  

As we already noted, there is a one-to-one correspondence between regular double Stone algebras and three-valued
{\L}ukasiewicz algebras. This implies the following corollary; cf.~\cite{Banerjee97, BanChak96, Dai08, Ittu98, Pagl96}.

\begin{corollary} For any equivalence, $\mathit{RS}$ forms a three-valued {\L}ukasiewicz algebra such that
\[ {\sim} (A,B) = (B^c, A^c),  \quad {\vartriangle} (A,B) = (A, A), \quad {\triangledown} (A, B) = (B, B) .\]
\end{corollary} 
L.~Iturrioz \cite{Ittu98} also proved that each three-valued {\L}ukasiewicz algebra can be embedded into some rough set {\L}ukasiewicz algebra defined by
an equivalence. 

Pagliani considered in \cite{Pagl96} the so-called \emph{disjoint-representation} of rough sets
\[  \mathit{dRS} := \{ (X^\DOWN,X^{\UP c}) \mid X \subseteq U\}. \]
Obviously, there is a one-to-one correspondence between $\mathit{RS}$ and $\mathit{dRS}$. Pagliani defined a congruence 
$\theta$ on $(\mathit{Def}, \cup, \cap)$ by
\[ \theta := \{ (A,B) \mid (\exists Z \in {\uparrow}\mathcal{S}) \, A \cap Z = B \cap Z \}. \]
For $X \in \mathit{Def}$, $X \theta U$ iff there is a $Z \in {\uparrow}\mathcal{S}$ such that $X \cap Z = U \cap Z = Z$, that is, $Z \subseteq X$. This means that 
\[ X \theta U \iff \mathcal{S} \subseteq X. \]
This implies by \eqref{eq:RS_characterization} the following proposition:

\begin{proposition} \label{prop:pagliani2}
Let $E$ be an equivalence on $U$. Then
\[
\mathit{dRS} = \{ (A,B) \in \mathit{Def} \mid A \cap B = \emptyset \ \text{ and } \ (A \cup B) \theta U \}.
\]
\end{proposition}
Because $\mathit{Def}$ is a complete Boolean algebra, the congruence $\theta$ is trivially a Boolean congruence. Therefore,
$\mathit{dRS}$ forms a Nelson algebra by  Sendlewski's construction~\cite{Send90}:
\[ N_{\theta}(\alg{L}) = \{ (a, b) \in L^2 \mid a \wedge b = 0 \text{ and } (a \vee b) \theta 1 \}, \]
where $\alg{L}$ is a Heyting algebra and $\theta$ is a Boolean congruence on $\alg{L}$.
Pagliani noted in \cite{Pagl96} that the Nelson algebras on $\mathit{dRS}$ and $\mathit{RS}$ are semisimple.

Let $(A,B)$ and $(C,D)$ be elements of $\mathit{RS}$. The Nelson implication is defined in $\mathit{RS}$ by
\[ (A,B) \to (C,D) = (A^c \cup C, A^c \cup D).\] 
Since every three-valued {\L}ukasiewicz algebra forms a Heyting algebra,
$\mathit{RS}$ is a Heyting algebra in which
\begin{equation} \label{eq:Heyting_equivalences}
(A,B) \Rightarrow (C,D) = ( (A^c \cup C) \cap (B^c \cup D), B^c \cup D).
\end{equation}
In \cite{Panicker19}, the authors consider so-called $C$-algebras. They 
show that $\mathit{dRS}$, with suitable operations, forms a $C$-algebra and
that each $C$-algebra can be embedded in $\mathit{dRS}$.

These results fit into a larger picture. Recall from universal algebra that the (\textit{ternary}) \textit{discriminator}
\cite[Definition~IV{\S}9.1]{Burr81} on a set~$A$ is the function $t \colon A^3 \to A$ defined for all $a, b, c \in A$ by
\begin{gather*}
t(a, b, c) := \begin{cases}
              c & \text{if $a = b$;} \\
              a & \text{otherwise}.
              \end{cases}
\end{gather*}
A (\textit{ternary}) \textit{discriminator variety} is a variety~$\class{V}$ for which there exists a ternary term $t(x, y, z)$ of~$\class{V}$
that realises the discriminator on each subdirectly irreducible member of~$\class{V}$; by an instructive characterisation due to
Blok and Pigozzi~\cite[Corollary~3.4]{Blok84} (see also \cite{Frie83}), an equational class forms a discriminator variety iff it is congruence
permutable, semisimple, and has Equationally Definable Principal Congruences (EDPC) in the sense of \cite{Kohl80}.

According to Burris and Sankappanavar \cite[Chapter IV{\S}9, 10]{Burr81}, discriminator varieties constitute
``\dots the most successful generalisation of Boolean algebras to date, successful because we obtain Boolean product representations.''
As such, discriminator varieties have been considered extensively in the literature---standard references include
Werner~\cite{Wern78} and J\'{o}nsson~\cite{Jons95}---and it is known, in particular, that 
the regular double Stone algebras,
the three-valued {\L}ukasiewicz algebras, and the semisimple Nelson algebras all form discriminator varieties.

\subsection{Complete lattices of rough sets defined by quasiorders} 
\label{ss:CLRSQ}

In the literature,  numerous studies exist on rough sets that are determined by different types of relations
reflecting distinguishability or indistinguishability of the elements of the universe of discourse $U$; see e.g.~\cite{DemOrl02}.
Rough sets induced by quasiorders have been in the focus of recent interest; see 
\cite{Ganter07,  Jarv13, Jarv11a,  Jarv14, JarRad21, JarRadVer09,
Kortelainen1994, KumBan2012, KumBan2015,  Umadevi13a, Qiao2012},
for example. In this section, we consider the order-theoretical properties of rough sets defined by a quasiorder.

Let $\lesssim$ be a quasiorder on $U$, that is, $\lesssim$ is reflexive and transitive binary relation on the set $U$.
The inverse of $\lesssim$ is $\gtrsim$. Denote $[ x ) = \{ y \in U \mid x \lesssim y \}$ and $( x ] = \{ y \in U \mid x \gtrsim y \}$.
We define the following rough approximation operators for any $X \subseteq U$: 
\[
\begin{gathered}
    X^\UP   := \{ x \in U \mid [x) \cap X \neq \emptyset\},  \qquad \qquad 
    X^\DOWN := \{ x \in U \mid [x) \subseteq X \},         \\
    X^\Up  := \{ x \in U \mid (x] \cap X \neq \emptyset\},  \qquad \qquad
    X^\Down := \{ x \in U \mid (x] \subseteq X \}.          
\end{gathered}
\]
By definition,
\[ X^{\DOWN c} = X^{c \UP}, \ X^{\Down c} = X^{c \Up}, \ X^{\UP c} = X^{c \DOWN}, \ X^{\Up c} = X^{c \Down}.\]  
In \cite{Kortelainen1994} it was noted that 
\begin{equation*}
X^{\UP \Down} = X^\UP, \ X^{\Up \DOWN} = X^\Up, \ X^{\DOWN \Up} = X^\DOWN, \ X^{\Down \UP} = X^\Down. 
\end{equation*}

A quasiorder $\lesssim$ can be interpreted as a \emph{specialisation order}, where $x \lesssim y$ may be read as ``$y$ is a specialisation of $x$''.
In \cite{Ganter07}, a specialisation order is viewed as ``non-symmetric indiscernibility'' such that each element is indistinguishable with all its
specialisations, but not necessarily the other way round. Then, in our interpretation, $x \in X^\UP$ means that there is at least one specialisation $y$ in $X$,
which  cannot be distinguished from $x$. Similarly, $x$ belongs to $X^\DOWN$ if all its specialisations are in $X$; this is then interpreted so that
$x$ needs to be in $X$ in the view of the knowledge $\lesssim$.

Each Alexandrov topology $\mathcal{T}$ on $U$ defines a quasiorder $\lesssim_\mathcal{T}$ on $U$ by $x \, \lesssim_\mathcal{T} \, y$ if and only if
$y \in N(x)$  for all $x, y \in U$. On the other hand, for a quasiorder $\lesssim$ on $U$, the set of all $\lesssim$-closed subsets of $U$, called
also the \emph{up-sets} of $(U, \lesssim)$, forms an Alexandrov topology $\mathcal{T}_\lesssim$.
Thus, $B \in \mathcal{T}_\lesssim$ if and only if $x \in B$ and $x \lesssim y$ imply $y \in B$. 
The correspondences $\mathcal{T} \mapsto {\lesssim_\mathcal{T}}$ and ${\lesssim} \mapsto \mathcal{T}_\lesssim$ are
mutually invertible bijections between the classes of all Alexandrov topologies and of all quasiorders on the set $U$.

We also have that
\begin{equation*}
\mathcal{T}_\lesssim = \{ X^\DOWN \mid X \subseteq U\} = \{ X^\Up \mid X \subseteq U\} 
\end{equation*}
and
\begin{equation*}
\mathcal{T}_\gtrsim = \{ X^\UP \mid X \subseteq U\} = \{ X^\Down \mid X \subseteq U\}. 
\end{equation*}
For the Alexandrov topology $\mathcal{T}_\gtrsim$,  $(x] = \{x\}^\UP$ is the smallest neighbourhood of the point $x \in U$ and
$^\UP \colon \wp(U) \to \wp(U)$ is the smallest neighbourhood operator. The map $^\Up \colon \wp(U) \to \wp(U)$ is the closure operator
of  $\mathcal{T}_\gtrsim$. Note that the family of closed sets for the topology $\mathcal{T}_\gtrsim$ is $\mathcal{T}_\lesssim$.
The map $^\Down \colon \wp(U) \to \wp(U)$ is the interior operator.
 
Similarly, $[x) = \{x\}^\Up$ is the smallest neighbourhood of a point $x \in U$ in $\mathcal{T}_\lesssim$. The map
$^\Up \colon \wp(U) \to \wp(U)$ is the smallest neighbourhood operator, $^\UP \colon \wp(U) \to \wp(U)$ is the closure operator and
$^\DOWN \colon \wp(U) \to \wp(U)$ is the interior operator
of $\mathcal{T}_\lesssim$.  

For a quasiorder $\lesssim$, we denote by $\mathit{RS}(\lesssim)$ the set of rough sets  $\{ (X^\DOWN, X^\UP) \mid X \subseteq U \}$
induced by $\lesssim$. As in the case of equivalences, we denote $\mathit{RS}(\lesssim)$ simply by $\mathit{RS}$ when there is no chance for
confusion.  The set $\mathit{RS}$ can be ordered coordinatewise and in \cite{JarRadVer09} J{\"a}rvinen, Radeleczki, and Veres  proved the following
theorem stating that, as in the case of equivalences, $\mathit{RS}$ forms a complete sublattice of $\wp(U) \times \wp(U)$.

\begin{theorem} \label{Thm:quasi_complete_lattice}
For any quasiorder, $\zseq{\mathit{RS}; \leq}$ forms a complete lattice such that for all $\{(A_i,B_i)\}_{i \in I} \subseteq \mathit{RS}$,
\[
\bigvee_{i \in I} ( A_, B_i ) = 
\big ( \bigcup_{i \in I} A_i, \bigcup_{i \in I} B_i \big )
\ \mbox{ and } \
\bigwedge_{i \in I} ( A_i, B_i ) = 
\big ( \bigcap_{i \in I} A_i, \bigcap_{i \in I} B_i \big ) .
\]
\end{theorem}

In \cite{Jarv11a} it is noted that
\[ \alg{RS} = \zseq{\mathit{RS}; \land, \lor, {\sim}, (\emptyset,\emptyset), (U,U)}\]
is a Kleene algebra such that
${\sim}(X^\DOWN,X^\UP) = (X^{c \DOWN}, X^{c \UP} )$. A De~Morgan algebra $\alg{A}$ is \emph{centered} if there exists an element such
that $c = {\sim}c$; this element $c$ is called the \emph{center} of $\alg{A}$. It is well known and obvious that a Kleene algebra can have 
at most one center. We also write $\mathcal{S} := \{ x \in U \mid [x) = \{x\} \}$ for the set of the \emph{singletons}.

Let $R$ be a binary relation on $U$ that is at least transitive. A \emph{successor} of $x \in U$ is an element $y \in U$
such that $x \, R \, y$. Let $X \subseteq Y \subseteq U$. Then, $X$ is \emph{cofinal in} $Y$ if each $x \in Y$ has a successor in $X$.
We also say that a set is \emph{cofinal}, if it is cofinal in $U$. M.~H.~Stone has proved \cite[Theorem~1]{Stone68} that
a necessary and sufficient condition that the set $U$ has a partition into $k$ cofinal subsets is that each element of $U$
has at least $k$ successors.

\begin{proposition} \label{prop:Centered}
The Kleene algebra $\mathbf{RS}$ is centered if and only if the set of $\lesssim$-singletons $\mathcal{S}$ is empty.
\end{proposition}

\begin{proof}
Indeed, if $\mathcal{S} = \emptyset$, then each element $x \in U$ has at least two successors. This implies that
$U$ can be divided into two cofinal subsets $X$ and $Y$. Obviously, $X^\UP = U$. Suppose that $X^\DOWN \neq \emptyset$,
Then there is $x \in U$ such that $x \in [x) \subseteq X$. But this is not possible because $Y$ is cofinal in $U$ and
$x$ has a successor in $Y$. Since $X$ and $Y$ are disjoint, we have a contradiction. Thus, $X^\DOWN = \emptyset$ and
$(\emptyset,U)$ is the unique center of $\alg{RS}$.

On the other hand, suppose that there is a center $(C^\DOWN,C^\UP)$ in $\mathit{RS}$.
Then $(C^\DOWN,C^\UP) = (C^\DOWN,C^\UP) \wedge {\sim} (C^\DOWN,C^\UP) = (C^\DOWN \cap C^{\UP c}, C^\UP \cap C^{\DOWN c})$
giving $C^\DOWN = C^\DOWN \cap C^{\UP c} \subseteq C^\DOWN \cap C^{\DOWN c} = \emptyset$. Analogously,
$(C^\DOWN,C^\UP) = (C^\DOWN,C^\UP) \vee {\sim} (C^\DOWN,C^\UP) = (C^\DOWN \cup C^{\UP c}, C^\UP \cup C^{\DOWN c})$ yields
$C^\UP = C^\UP \cup C^{\DOWN c} \supseteq C^\DOWN \cup C^{\DOWN c} = U$. Thus, $(C^\DOWN,C^\UP) = (\emptyset, U)$.
Assume $\mathcal{S} \neq \emptyset$. Then, there is $x \in U$ such that $[x) = \{x\}$. Because $x \in C^\UP = U$, we have
$x \in C$ and $x \in C^\DOWN$, which is impossible because $C^\DOWN = \emptyset$. Thus, $\mathcal{S} = \emptyset$.
\end{proof}

The characterisation of $\mathit{RS}$ was given by J{\"a}rvinen, Pagliani and Radeleczki \cite{JarvPaglRade11, Jarv13} and independently by
E.~K.~R. Nagarajan and D.~Umadevi \cite{Umadevi13a} and Quanxi Qiao \cite{Qiao2012}:

\begin{theorem} \label{thm:quasiorder_characterization}
For any quasiorder $\lesssim$,
\[
\mathit{RS} = \{ (A,B) \in \mathcal{T}_\lesssim \times \mathcal{T}_\gtrsim \mid A \subseteq B \ \text{and} \ A \cap \mathcal{S} = B \cap \mathcal{S} \}. 
\]
\end{theorem}
Equivalently we can write 
\[ 
\mathit{RS} = \{ (A,B) \in \mathcal{T}_\lesssim \times \mathcal{T}_\gtrsim \mid  A \subseteq B \ \text{and} \ \mathcal{S} \subseteq A \cup B^c  \}. 
\]

\begin{remark}
The proofs of Theorems \ref{Thm:quasi_complete_lattice} and \ref{thm:quasiorder_characterization}
require the condition that $U$ has a partition into $k$ cofinal subsets if and only if every element of $U$ has at least $k$
successors. As we already mentioned, J.~Pomyka{\l }a and J.~A.~Pomyka{\l }a showed in \cite{PomPom88} that 
for equivalence relations, $\mathit{RS}$ is a Stone lattice. In their proof they 
used Zermelo's Axiom of Choice. Interestingly, the proof of the condition for cofinality by Stone also requires
Axiom of Choice.
\end{remark}

Because $\alg{RS} = \zseq{\mathit{RS}; \leq}$ is a complete sublattice of $\zseq{\wp(U) \times \wp(U); \leq}$, $\alg{RS}$ is completely distributive.
This means that $\alg{RS}$ forms a Heyting algebra. Although this fact has been well known, in the literature one cannot find
the description of the relative pseudocomplement operation. Our next proposition removes this disadvantage.
For the proof, note that Theorem~\ref{thm:quasiorder_characterization} means that $(A,B)$ is a rough set
iff $A \subseteq B$ and for $x \in \mathcal{S}$, $x \in B$ implies $x \in A$.

\begin{proposition} \label{prop:HeytingImplication}
For rough sets $(A,B)$ and $(C,D)$ determined by a quasiorder $\lesssim$,
\[ (A,B) \Rightarrow (C,D) = ((A^c \cup C)^\DOWN \cap (B^c \cup D)^{\Down \DOWN}, (B^c \cup D)^\Down). \]  
\end{proposition}

\begin{proof} Clearly, $(A^c \cup C)^\DOWN \cap (B^c \cup D)^{\Down \DOWN}$ belongs to $\mathcal{T}_\lesssim$ and
$(B^c \cup D)^\Down$ is in $\mathcal{T}_\gtrsim$. It is also obvious that 
$(A^c \cup C)^\DOWN \cap (B^c \cup D)^{\Down \DOWN} \subseteq (B^c \cup D)^\Down$.

Let $x \in \mathcal{S}$ be a singleton such that $x \in (B^c \cup D)^\Down$. Because $[x) = \{x\}$, we have
$x \in (B^c \cup D)^{\Down \DOWN}$. In addition, $x \in B^c \cup D$ implies $x \in A^c \cup D$ because $A \subseteq B$.
Again, $[x) = \{x\}$ implies $x \in (A^c \cup D)^{\DOWN}$. If $x \in D$, then $x \in C$ because $(C,D)$ is a rough set
and $x$ is a singleton. We have $x \in (A^c \cup C)$ and $x \in  (A^c \cup C)^\DOWN$. Thus, we have proved that 
$((A^c \cup C)^\DOWN \cap (B^c \cup D)^{\Down \DOWN}, (B^c \cup D)^\Down)$ is a rough set.

It is also clear that
\[
\begin{gathered}
  (A,B) \wedge ((A^c \cup C)^\DOWN \cap (B^c \cup D)^{\Down \DOWN}, (B^c \cup D)^\Down)  \subseteq 
  (A \cap (A^c \cup C), B \cap (B^c \cup D)) \\
   = (A \cap C, B \cap D) \subseteq (C,D).
\end{gathered}
\]

Now assume that $(A,B) \wedge (X,Y) \leq (C,D)$ for some rough set $(X,Y)$. Then $A \cap X \subseteq C$ and
$B \cap Y \subseteq D$ imply $X \subseteq A^c \cup C$ and $Y \subseteq B^c \cup D$. Because $X \in \mathcal{T}_\lesssim$
and $Y \in \mathcal{T}_\gtrsim$, we have $X = X^\DOWN \subseteq (A^c \cup C)^\DOWN$ and $Y = Y^\Down \subseteq (B^c \cup D)^\Down$.
Now $X \subseteq Y$ implies $X = X^\DOWN \subseteq Y^\DOWN \subseteq (B^c \cup D)^{\Down \DOWN}$. Combining these, we get
$X \subseteq (A^c \cup C)^\DOWN \cap  (B^c \cup D)^{\Down \DOWN}$. This means that have now shown that
\[ (X,Y) \leq ((A^c \cup C)^\DOWN \cap (B^c \cup D)^{\Down \DOWN}, (B^c \cup D)^\Down),\]
which completeness the proof.
\end{proof}

Note that if $\lesssim$ is an equivalence $E$ on $U$, then $\lesssim$ is symmetric and we have $X^\DOWN = X^\Down$ and
$X^\UP = X^\Up$ for all $X \subseteq U$. Additionally, $\mathcal{T}_\lesssim = \mathcal{T}_\gtrsim$
equals the family of $E$-definable sets. This means that the operation of Proposition~\ref{prop:HeytingImplication}
becomes
\[ (A,B) \Rightarrow (C,D) = ((A^c \cup C) \cap (B^c \cup D), B^c \cup D), \]
which coincides with \eqref{eq:Heyting_equivalences}.

In any Heyting algebra $\alg{L}$, the pseudocomplement $x^*$ equals $x \Rightarrow 0$. Therefore, for any rough set
$(A,B)$, we obtain
\[
  (A,B)^* = \big( (A,B) \Rightarrow (\emptyset,\emptyset)\big) = (A^{c \DOWN} \cap B^{c \Down \DOWN}, B^{c \Down}) = (B^{\Up \UP c}, B^{\Up c}).
\]
Note that $A \subseteq B$ gives $B^c \subseteq A^c$ and $B^{c \Down \DOWN} \subseteq A^{c \Down \DOWN} = A^{\Up c \DOWN} = A^{c \DOWN}$. Because
$A \in \mathcal{T}_\lesssim$, $A^\Up = A$. Now $(A,B)^*$ coincides with the pseudocomplement given in
\cite[Proposition~4.2]{JarRadVer09}. In addition,
\[ \alg{RS} = (\mathit{RS}, \land, \vee, {\sim}, ^*, (\emptyset,\emptyset), (U,U) ) \]
forms a pseudocomplemented Kleene algebra in which the dual pseudocomplement is defined by $(A,B)^+ = (A^{\Down c}, A^{\Down \DOWN c})$,
as stated in \cite[Proposition~4.3]{JarRadVer09}.

A pseudocomplemented De Morgan algebra $\alg{A}$ is called \emph{normal} if
\[ x^* \leq {\sim} x \leq x^+. \]
Let $(A,B) \in \mathit{RS}$. Then $A \subseteq B$ and $B^c \subseteq A^c$ give
\[ (B^{c \Down \DOWN}, B^{c \DOWN}) \leq (B^c, A^c) \leq (A^{c \Up}, A^{c \Up \UP}) , \]
that is $(A,B)^* \ \leq \ {\sim}(A,B) \ \leq \ (A,B)^+$ and $\alg{RS}$ is normal.
  
\medskip

Completely join irreducible elements of $\mathit{RS}$ were found in \cite{JarRadVer09}:
\begin{proposition} \label{Prop:join_irreducibles}
Let $\lesssim$ be a quasiorder on $U$. Then,
\[
\mathcal{J} = \{ (\emptyset,\{x\}^\UP) \mid \text{card} \big ( [x) \big ) \geq 2\} \cup \{ ( \{x\}^\Up,\{x\}^{\Up \UP}) \mid x \in U\}.
\]
\end{proposition}  
Because the complete lattice $\alg{RS} = \zseq{\mathit{RS};\leq}$ is algebraic and completely distributive, $\alg{RS}$ is spatial.
Note that Proposition~\ref{Prop:join_irreducibles} implies that if $\mathcal{S} = \emptyset$, then $(\emptyset, U)$ is a rough set; cf.~Proposition~\ref{prop:Centered}.

The dually pseudocomplemented distributive lattice $\alg{RS}$ is not generally regular or a Stone lattice.
For a quasiorder $\lesssim$ on $U$, the smallest equivalence containing $\lesssim$ is the transitive closure of the relation
${\lesssim} \cup {\gtrsim}$.  The following characterisation is presented in \cite{JarRadVer09}.

\begin{proposition} \label{prop:QuasiStone}
Let $\lesssim$ be a quasiorder. The following are equivalent:
\begin{enumerate}[label = {\rm (\alph*)}]  
\item $\zseq{\mathit{RS}; \vee,\wedge,{^*},{^+},(\emptyset,\emptyset), (U,U)}$ is a (double) Stone lattice;
\item ${\gtrsim} \circ {\lesssim}$ is equal to  the smallest equivalence containing $\lesssim$.
\end{enumerate}
\end{proposition}
  
For any binary relation $R$ on $U$, a set $C$ is called a \emph{connected component} of $R$, 
if $C$ is an equivalence class of the smallest equivalence relation containing $R$.
Denote by $\mathit{Co}(R)$ the set of all $R$-connected components. We also denote $\mathit{Co}(R)$ simply by $\mathit{Co}$ in the case there is
no danger of confusion. Let $\lesssim$ be a quasiorder. For any connected component $C \in \mathit{Co}$ and $x \in U$,
$[x) \cap C \neq \emptyset $ implies $x \in C$, and $x \in C$ implies $[x) \subseteq C$.
Hence, $C^\UP \subseteq C \subseteq C^\DOWN$. This means that   $C^\DOWN = C = C^\UP$ for any
connected component $C$.
 
We denote for each $C \in \mathit{Co}$ by $\mathit{RS}(C)$ the set of rough sets on the component $C$ determined by the restriction of $\lesssim$ to $C$. 
A binary relation $R$ on $U$ is \emph{left-total} if for any $x \in U$, there exists $y \in U$ such that $x \, R \, y$.
Note that every reflexive  relation is left-total. In the literature left-total relations are also called \emph{total} or \emph{serial} relations.

In \cite{JarRadVer09}, the following decomposition theorem was proved even in the general setting of a left-total relation.

\begin{theorem} \label{thm:decomposition}
For any quasiorder, 
\[  \zseq{\mathit{RS};\leq} \cong \zseq{\prod_{C \in \mathit{Co}} \mathit{RS}(C); \leq }.  \]
\end{theorem}

Note that from Theorem~\ref{thm:decomposition} we obtain the above-mentioned result by Gehrke and Walker stating that for an equivalence $E$, 
$\zseq{\mathit{RS};\leq}$ is order-isomorphic to a pointwise-ordered direct product of chains of two and three elements. This is because if $x$ is a singleton,
then its connected component is $C = \{x\}$ and $\mathit{RS}(C)$ forms a chain of two elements $(\emptyset,\emptyset)$ and $(C,C)$.
In the case $x$ is not singleton, then the connected component $C$ containing $x$ has at least two elements and $\mathit{RS}(C)$ is a chain of the three elements
$(\emptyset,\emptyset)$, $(\emptyset,C)$ and $(C,C)$.

As in the case of equivalences, we say that $X \subseteq U$ is \emph{definable} with respect to $\lesssim$ if $X^\DOWN = X^\UP$.
Clearly, $\emptyset$ and $U$ are definable. The following result is not appearing in the literature. Therefore, we are presenting
a proof for it.

\begin{proposition} \label{prop:definable}
For a quasiorder $\lesssim$, the definable sets are the unions of the connected components of $\lesssim$ and the empty set $\emptyset$.
\end{proposition}

\begin{proof}
Let $U$ be a set and let $X$ be a union of connected components of a quasiorder $\lesssim$ on $U$.
This means that there exists a subfamily $\mathcal{H} \subseteq \mathit{Co}$
such that $X = \bigcup \mathcal{H}$. Because $C^\DOWN = C^\UP$ for all $C \in \mathcal{H}$, we have
$X^\UP = \big ( \bigcup \mathcal{H} )^\UP = \bigcup \{ C^\UP \mid C \in \mathcal{H} \} =
\bigcup \{ C^\DOWN \mid C \in \mathcal{H} \} \subseteq (\bigcup \mathcal{H} \big) ^\DOWN = X^\DOWN \subseteq X^\UP$.
Thus, $X^\UP = X^\DOWN$.

Conversely, if $X \neq \emptyset$ is definable, then $X^\DOWN = X = X^\UP$ and $X$ belongs both to $\mathcal{T}_\lesssim$ and $\mathcal{T}_\gtrsim$.
Therefore, $X$ is closed with respect to $\lesssim $ and $\gtrsim $.
This means that for any $x\in X$ and $y,z\in U$, $x \lesssim y$ implies $y \in X$ and $x\gtrsim z$ implies $z\in X$.
Since the smallest equivalence $E$ containing $\lesssim $ and $\gtrsim $ is their lattice-theoretical join, we have $(u,v)\in E$ for some
$u,v\in U$ if and only if there exist $z_{0}, z_{1},...,z_{n} \in U$ such that $u=z_{0}$, $v=z_{n}$, and $z_{i-1}\lesssim z_{i}$ or $z_{i-1}\gtrsim z_{i}$
for all $1 \leq i \leq n$. Therefore, for any $x \in X$, $(x,w)\in E$ implies $w \in X$.
Hence $X$ is closed with respect to the equivalence $E$ and this implies that $X$ equals to a union of some classes of $E$.
However, the classes of $E$ are just the connected components of $\lesssim $. Thus, $X$ is a union of some connected components of $\lesssim $.
\end{proof}

We end this subsection by noting that there are also other choices for lower-upper approximation pairs defined in terms of a quasiorder $\lesssim$ on $U$.
In \cite{KumBan2012,KumBan2015}
Kumar and Banerjee define the operators $\mathsf{L}$ and $\mathsf{U}$ by 
\[ 
\mathsf{L}(X) = \bigcup \{ D \in \mathcal{T}_\lesssim \mid D \subseteq X \} \quad \mbox{and} \quad
\mathsf{U}(X) = \bigcap \{ D \in \mathcal{T}_\lesssim \mid X \subseteq D \} \]
for any $X \subseteq U$. The sets $\mathsf{L}(X)$ and $\mathsf{U}(X)$ belong to the same topology $\mathcal{T}_\lesssim$, whose
elements Kumar and Banerjee called ``definable''. These operators can be also be written in form
\[ \mathsf{L}(X) = \{ x \in U \mid [x) \subseteq X\} \quad \mbox{and} \quad \mathsf{U}(X) = \{ x \in U \mid (x] \cap X \neq \emptyset \}. \]
This approach differs significantly from the one considered here, because now the rough set system
\[ \{ ( \mathsf{L}(X), \mathsf{U}(X) ) \mid X \subseteq U \} \]
is not generally a lattice with respect to the coordinatewise order, as noted in \cite{JarRad21, KumBan2015}.
In \cite{KumBan2012} it is shown that 
$\{ ( \mathsf{L}(X), \mathsf{U}(X) ) \mid X \subseteq U \}$ 
is isomorphic to the complete lattice
$\{(D_1, D_2) \in  \mathcal{T}_\leq \times  \mathcal{T}_\leq \mid D_1 \subseteq D_2 \}$ if and only if all the pairs $(D_1, D_2)$ satisfy
$|D_2 \setminus D_1| \neq 1$.

\subsection{Nelson algebras of rough sets}
\label{ss:NelsonRS}

In this section, we consider Nelson algebras of quasiorder-based rough sets. We also recall the
representation theorem for Nelson algebras defined on algebraic lattices in terms of rough sets.

According to R.~Cignoli \cite{Cign86}, a \emph{quasi-Nelson algebra} is a Kleene algebra \[\alg{A} = \zseq{A; \land, \lor, {\nnot}, 0, 1 }\]
such that for  all $a,b \in A$, the \emph{weak relative pseudocomplement} of $a$ with respect to $b$
\begin{equation} \label{eq:NelsonImplication}
a \to b :=  a \Rightarrow ({\sim}a \vee b)
\end{equation}
exists. This means that every Kleene algebra whose underlying lattice is a Heyting algebra, and, in particular, any Kleene algebra defined on an
algebraic lattice, forms a quasi-Nelson algebra.

We say that a De~Morgan algebra $\alg{A}$ is \emph{completely distributive}, if its underlying lattice $A$ is completely distributive. In such
a case, we define for any $j \in \mathcal{J}$ the element
\begin{equation} \label{eq:brother}
  g(j) = \bigwedge \{  a \in A \mid a \nleq {\sim} j \} .
\end{equation}  
The properties of $g(j)$ were presented in \cite{Mont63b} for finite De~Morgan algebras. We recall them here in the case
$\alg{A}$ is a completely distributive De~Morgan algebra. First we note that
\begin{equation} \label{eq:gee_1}
(\forall j \in \mathcal{J}) \,  g(j) \nleq {\sim} j.
\end{equation}  
Indeed, if $g(j) \leq {\sim} j$, then
\[ {\sim} j = {\sim} j \vee \bigwedge \{ a \in A \mid a \nleq {\sim} j \} = 
\bigwedge \{ a \vee {\sim} j \mid a \in A \text{ and } a \nleq {\sim} j \}. \]
Since $j \in \mathcal{J}$, the element ${\sim} j$ is completely meet-irreducible.
We have that ${\sim} j =  a \vee {\sim} j$ and $a \leq {\sim}j$ for some $a \in A$ such that $a \nleq {\sim} j$, a contradiction.

For $a \in A$ and $j \in \mathcal{J}$,
\begin{equation} \label{eq:gee_2}
  j \nleq {\sim} a \iff g(j) \leq a . 
\end{equation}
To verify \eqref{eq:gee_2}, assume $j \nleq {\sim} a$. Because this is equivalent to $a \nleq {\sim}j$, we directly get $g(j) \leq a$.
Conversely, assume $g(j) \leq a$. If $j \leq {\sim} a$, then $j \leq {\sim} a \leq {\sim} g(j)$ and $g(j) \leq {\sim} j$ contradicting $g(j) \nleq {\sim} j$.
Thus, $j \nleq {\sim} a$.

Note also that by \eqref{eq:brother} and \eqref{eq:gee_1}, $g(j)$ is the least element of $A$ which is not below ${\sim}j$. 
Using this fact we can prove that $g(j) \in \mathcal{J}$.
Namely, if $g(j) = \bigvee S$ for some $S \subseteq A$, then $b \leq g(j)$ for all $b \in S$. Assume that $b < g(j)$ for all $b \in S$. Then
$b \notin \{ a \in A \mid a \nleq {\sim} j \}$ and $b \leq {\sim} j$ for all $b \in S$. Therefore, $g(j) = \bigvee S \leq {\sim} j$, a contradiction to
\eqref{eq:gee_1}.
We obtain that $b = g(j)$ for some $b \in S$ and $g(j) \in S$. This means that $g(j) \in \mathcal{J}$.

Thus, $g$ is a mapping $\mathcal{J} \to \mathcal{J}$. The idea is that for a completely distributive De~Morgan algebra $A$,
the map $g$ on $\mathcal{J}$ behaves similarly as the map $g$ of Section~\ref{ss:pri} on the set $X(A)$ of prime filters of $A$. But dealing with
completely join-irreducible elements is easier than dealing with prime filters. 
By overloading the notation $g$, we can write that for every $j \in \mathcal{J}$,
\begin{align*}
  g([j)) &= \{ a \in A \mid {\sim}a \in [j) \}^c = \{ a \mid j \leq {\sim} a\}^c 
   = \{ {\sim} a \mid j \leq a \}^c \\
   &= \{ a \mid a \leq {\sim} j \}^c 
        = \{ a \mid a \nleq {\sim} j \} = \big [g(j) \big ).
\end{align*}
It is well-known that in any distributive lattice, $[j)$ is a prime filter for each $j \in \mathcal{J}$.

We have that
\begin{equation} \label{eq:J1}
  (\forall j \in \mathcal{J}) \, g(g(j)) = j.
\end{equation}
This is because $g(g(j)) = \bigwedge \{ a \in A \mid a \nleq {\sim} g(j) \}$, $g(j) \nleq {\sim}j$ gives $j \nleq {\sim} g(j)$ and $g(g(j)) \leq j$.
On the other hand, $g(j) \in \mathcal{J}$ gives $g(g(j)) \nleq {\sim}g(j)$ by \eqref{eq:gee_1}. This implies $g(j) \nleq g(g(j))$.
By \eqref{eq:gee_2}, this is equivalent to $j \leq g(g(j))$. Thus, $j = g(g(j))$.

Let $j \leq k$ for some $j,k \in \mathcal{J}$. Then ${\sim}k \leq {\sim}j$. This means that for any $a \in A$, $a \leq {\sim} k$ implies
$a \leq {\sim} j$, or equivalently, $a \nleq {\sim} j$ implies $a \nleq {\sim} k$. Therefore,
$\{ a \in A \mid a \nleq {\sim} j \} \subseteq  \{ a \in A \mid a \nleq {\sim} k \}$ and 
$g(j) = \bigwedge \{ a \in A \mid a \nleq {\sim} j \} \geq \bigwedge  \{ a \in A \mid a \nleq {\sim} k \} = g(k)$.
We can write:
\begin{equation} \label{eq:J2}
  (\forall j,k \in \mathcal{J}) \, j \leq k \implies g(j) \geq g(k).
\end{equation}
Note that \eqref{eq:J1} and \eqref{eq:J2} mean that $\mathcal{J}$ is self-dual, that is, $(\mathcal{J},\leq) \cong (\mathcal{J},\geq)$.

Let $\alg{A}$ be a completely distributive Kleene algebra. Assume that $j \nleq g(j)$ and $g(j) \nleq j$ for some $j \in \mathcal{J}$.
By \eqref{eq:gee_2}, we get that $j \leq {\sim} j$ and $g(j) \leq {\sim} g(j)$. Because $\alg{A}$ is a Kleene algebra, we have
\[ g(j) = g(j) \wedge {\sim}g(j) \leq j \vee {\sim} j = {\sim} j \]
contradicting \eqref{eq:gee_1}. Thus, we have shown that 
\begin{equation} \label{eq:J3}
  (\forall j \in \mathcal{J}) \, j \leq g(j) \ \text{or} \ g(j) \leq j.
\end{equation} 

Define three sets:
\begin{align*}
  \mathcal{J}^- &: = \{j \in \mathcal{J} \mid j < g(j)\}, \\
  \mathcal{J}^\circ & := \{j \in \mathcal{J} \mid j = g(j)\}, \\
  \mathcal{J}^+ & : = \{j \in \mathcal{J} \mid j > g(j)\}.
\end{align*}
Because $j$ and $g(j)$ are always comparable, we have
\[ \mathcal{J} = \mathcal{J}^- \cup \mathcal{J}^\circ \cup \mathcal{J}^+.\]
Clearly, $j \in \mathcal{J}^- \iff g(j) \in \mathcal{J}^+$.
We also have that for all $j \in \mathcal{J}$,
\[ [j) \in X(A)^+ \iff j \in \mathcal{J}^+ \cup \mathcal{J}^\circ \text{ \ and \ }
  [j) \in X(A)^- \iff j \in \mathcal{J}^- \cup \mathcal{J}^\circ,\]
where $X(A)^+$ and $X(A)^-$ are defined as in Section~\ref{ss:pri}, that is,
$X(A)^+ = \{P \in X(A) \mid P \subseteq g(P)\}$ and $X(A)^- = \{P \in X(A) \mid g(P) \subseteq P \}$,

A Kleene algebra is said to have the \emph{interpolation property} \cite{Cign86} if for any prime filters $P$ and $Q$ such that $P,Q \subseteq g(P), g(Q)$, there is
prime filter $F$ such that
\[ P,Q \subseteq F \subseteq g(P), g(Q).\]

\begin{theorem}\cite[Theorem~3.5]{Cign86} \label{thm:quasi_nelson_filter}
A quasi-Nelson algebra is a Nelson algebra if and only if it has the interpolation property.
\end{theorem}

Let $\alg{A}$ be a completely distributive Kleene algebra. We say that $\mathcal{J}$ satisfies the \emph{interpolation property}
if for any $p,q \in \mathcal{J}$ such that $p,q \leq g(p),g(q)$, there is $f \in \mathcal{J}$ such that
\[p,q \leq f \leq g(p),g(q)  .\]

In \cite[Proposition~3.5]{Jarv11a} is presented the following theorem for Kleene algebras defined on algebraic lattices.
This result can be viewed as a counterpart of Theorem~\ref{thm:quasi_nelson_filter}.

\begin{theorem} \label{thm:quasi_nelson_irreducibles}
If $\alg{A} = \zseq{A; \land, \lor, {\nnot}, 0, 1}$ is a Kleene algebra defined on an algebraic lattice, then 
$\zseq{A; \land, \lor,  \to, {\nnot}, 0, 1}$ is a Nelson algebra,
where the operation $\to$ is defined by \eqref{eq:NelsonImplication}, if and only if $\mathcal{J}$ satisfies the
interpolations property.
\end{theorem}

\begin{lemma}\cite{Jarv11a}
Let $\lesssim$ be a quasiorder on $U$. We have that in $\zseq{\mathit{RS};\leq}$:
\begin{align*}
  \mathcal{J}^- &= \{ (\emptyset, \{x\}^\UP) \mid | R(x) | \geq 2 \}, \\
  \mathcal{J}^\circ &= \{ (\{x\}, \{x\}^\UP) \mid R(x) = \{x\}  \}, \\
  \mathcal{J}^+ &= \{ (R(x), R(x)^\UP) \mid | R(x) | \geq 2 \}.
\end{align*}
\end{lemma}

For any $(\emptyset, \{x\}^\UP) \in \mathcal{J}^-$, $g\big( (\emptyset, \{x\}^\UP) \big ) ) =  (R(x), R(x)^\UP)$.
It is proved in \cite{Jarv11a} that the set $\mathcal{J}(\mathit{RS})$ of completely join-irreducible elements of $\mathit{RS}$
has the interpolation property. This means that

\begin{theorem} \label{thm:quasiorder_nelson}
Let $\lesssim$ be a quasiorder on $U$. Then,
\[ \alg{RS} = \zseq{\mathit{RS}; \wedge, \vee, \to,{\sim},(\emptyset,\emptyset),(U,U)} \]
is a Nelson algebra.
\end{theorem}
For the rough sets $(A,B)$ and $(C,D)$, the operation $\to$ is defined by
\begin{align*}
  (A,B) \to (C,D) &= (A,B) \Rightarrow ({\sim}(A,B) \vee (C, D)) \\
  &= (A,B) \Rightarrow (B^c \cup C, A^c \cup D) \\
  & = ( (A^c \cup B^c \cup C)^\DOWN \cap (B^c \cup A^c \cup D)^{\Down \DOWN}, (B^c \cup A^c \cup D)^\Down ) \\
  & = ( (A^c \cup C)^\DOWN \cap (A^c \cup D)^{\Down \DOWN}, (A^c \cup D)^\Down). 
\end{align*}
Note that $A^c$ and $D$ belong to $\mathcal{T}_\lesssim$, meaning that also $A^c \cup D$ is in  $\mathcal{T}_\lesssim$.
Therefore, $(A^c \cup D)^{\Down} = A^{c} \cup D$. This also implies that  $(A^c \cup C)^\DOWN \cap (A^c \cup D)^{\Down \DOWN} = (A^c \cup C)^{\DOWN}$.
We have that
\[ (A,B) \to (C,D) = ((A^c \cup C)^{\DOWN}, A^{c} \cup D).  \]

One can also show that $\alg{RS}$ is a Nelson algebra by using Sendlewski's construction.
Let $\alg{L}$ be a distributive lattice and let $F$ be a lattice filter of $\alg{L}$. The equivalence
\[ \theta_F := \{ (a,b) \in L \times L \mid (\exists c \in F) \, a \wedge c = b \wedge c \} \]
is a congruence on $\alg{L}$. In a pseudocomplemented lattice, an element $a$ is called \emph{dense} if $a^* = 0$.
A filter $F$ is \emph{dense} if it contains all dense elements of $\alg{L}$. In particular,
a filter $F$ is dense  if and only if $\theta_F$ is a Boolean congruence; see e.g.~\cite[Exercise 8.28]{Blyt05}.
In case $\alg{L}$ forms a Heyting algebra, these definitions agree with the definitions of dense elements and filters
given in Section~2 for Brouwerian lattices.

If $X \in \mathcal{T}_\lesssim$ is dense, then $\mathcal{S} \subseteq X$. Indeed, if $x \in \mathcal{S}$, then
$x \in X^\UP = U$ means $[x) \cap X  \neq \emptyset$. Because $x$ is a singleton, we have $x \in X$. 
This gives that if $F = {\uparrow} \mathcal{S}$ is the principal filter generated by $\mathcal{S}$,
then $F$ is contains the filter of all dense elements of $\mathcal{T}_\lesssim$. Therefore,
\[  \theta_\mathcal{S} = \{ (A,B) \in \mathcal{T}_{\lesssim} \times  \mathcal{T}_{\lesssim}   \mid (\exists Z \in {\uparrow} \mathcal{S}) \, A \cap Z = B \cap Z \} \]
is a Boolean congruence on $\zseq{\mathcal{T}_\lesssim;\cup,\cap}$.  It is easy to see that
\[
\begin{gathered}
  (A \cup B) \,  \theta_\mathcal{S} \, U \iff (\exists Z \in {\uparrow} \mathcal{S}) \, Z \cap (A \cup B) = Z \cap U = Z \\
  \iff (\exists Z \in {\uparrow} \mathcal{S}) \, Z \subseteq (A \cup B) \iff \mathcal{S} \subseteq A \cup B.
\end{gathered}
\]
Now
\[ N_{\theta_\mathcal{S}}(\mathcal{T}_\lesssim) = \{ (A,B) \in \mathcal{T}_\lesssim \times \mathcal{T}_\lesssim \mid  A \cap B
= \emptyset \ \text{ and } \ (A \cup B) \,  \theta_\mathcal{S} \, U \} \]
is a Nelson algebra. By replacing $B$ with its complement $B^c$, it follows that
\[ \{ (A,B) \in \mathcal{T}_\lesssim \times \mathcal{T}_\gtrsim \mid  A \subseteq B \ \text{ and } \ \mathcal{S} \subseteq A \cup B^c \} \]
forms a Nelson algebra. The inclusion $\mathcal{S} \subseteq A \cup B^c$ is obviously equivalent to $A \cap \mathcal{S} = B \cap \mathcal{S}$ for all
$(A,B) \in \mathcal{T}_\lesssim \times \mathcal{T}_\gtrsim$ such that $A \subseteq B$.
This means by Theorem~\ref{thm:quasiorder_characterization} that we can use also Sendlewki's construction to show that
$\alg{RS}$ forms a Nelson algebra with the operations already presented in this subsection.

It should be noted that in \cite{Wolski2006} M.~Wolski considered Vakarelov's twist structures $\{ (a,b) \in L \times L \mid a \wedge b = 0\}$ in the
case $L$ is the Alexandrov topology $\mathcal{T}_\lesssim$ and called them as \emph{special $N$-lattices of approximations}. Using this construction,
we obtain a rough set Nelson algebra determined by a quasiorder $\lesssim$ only if $\mathcal{S} = \emptyset$.

The next representation theorem is given in \cite{Jarv11a}.

\begin{theorem} \label{thm:nelson_representation}
Let  $\alg{A} = \zseq{A; \land, \lor,  \to, {\nnot}, 0, 1}$ be a Nelson algebra defined on an algebraic lattice. Then, there exists a set $U$
and a quasiorder $\lesssim$ on $U$ such that $\alg{A} \cong \alg{RS}$.
\end{theorem}

Let us recall the main ideas of the proof of Theorem~\ref{thm:nelson_representation} from \cite{Jarv11a}.
It is noted in that paper that if  $\alg{L} = \zseq{L; \land, \lor, {\nnot}, 0, 1}$ and $\alg{K} = \zseq{K; \land, \lor, {\nnot}, 0, 1}$ 
are two De~Morgan algebras defined on algebraic lattices and $\varphi \colon \mathcal{J} \to \mathcal{K}$ is an order-isomorphism such that
\begin{equation}\label{eq:iso_gee}
  \varphi(g(j)) = g(\varphi(j))
\end{equation}
for all $j \in \mathcal{J}$, then $\alg{L} \cong \alg{K}$. This means that proving $\alg{A} \cong \alg{RS}$ can be done by proving that
the ordered sets of completely irreducible elements of $\alg{A}$ and $\alg{RS}$ are isomorphic and $\varphi$ is compatible with $g$.

We set $U = \mathcal{J}$ and define a mapping $\ell \colon \mathcal{J} \to \mathcal{J}$ by
\[ \ell(j) = \left \{
\begin{array}{ll}
g(j) & \text{if $j \in \mathcal{J}^+$} \\
j    & \text{otherwise}
\end{array}
\right .
\]
Now $\ell(\ell(j)) = \ell(j)$ and $\ell(j) = \ell(g(j))$ for all $j \in \mathcal{J}$. A quasiorder $\lesssim$ on $\mathcal{J}$ is defined by setting
\begin{equation} \label{eq:induced_quasiorder}
j \lesssim k \iff \ell(j) \leq \ell(k) .
\end{equation}

Let $\alg{A}$ be a Nelson algebra defined on an algebraic lattice and let the relation $\lesssim$ on $U = \mathcal{J}$ be defined as in \eqref{eq:induced_quasiorder}.
The isomorphism $\varphi$ satisfying \eqref{eq:iso_gee} between the set $\mathcal{J}$ of the completely join-irreducibles of $A$ and $\mathcal{J}(\mathit{RS})$
is defined by \[ \varphi(j) = \left \{
\begin{array}{ll}
  (\emptyset, \{x\}^\UP) & \text{if $j \in \mathcal{J}^-$},\\
  (\{x\}, \{x\}^\UP)     & \text{if $j \in \mathcal{J}^\circ$},\\
  (R(x), R(x)^\UP)       & \text{if $j \in \mathcal{J}^+$}.
\end{array}
\right .
\]

\begin{example}\label{exa:RS_construction}
Consider the Nelson algebra $\alg{A}$ of Figure~\ref{Fig:figure1}(a). Since $\alg{A}$ is finite, it
is trivially defined on an algebraic lattice. Suppose that the operation $\sim$ is defined by
${\sim}0 = 1$, ${\sim} a = f$, ${\sim} b = e$, and ${\sim} c = d$. The completely join-irreducible
elements $\mathcal{J}$ are marked by filled circles, and we have $g(a) = e$, $g(b) = f$, and $g(d) = d$.
The induced quasiorder on $U = \mathcal{J} = \{a,b,d,e,f\}$ is given in Figure~\ref{Fig:figure1}(b) and
the corresponding rough set structure $\mathit{RS}$ is depicted in Figure~\ref{Fig:figure1}(c).
Recall that the operation $\sim$ is defined in $\alg{RS}$ by ${\sim}(X^\DOWN,X^\UP) = (X^{\UP c}, X^{\DOWN c})$.

\begin{figure}[ht]
\centering
\includegraphics[width=120mm]{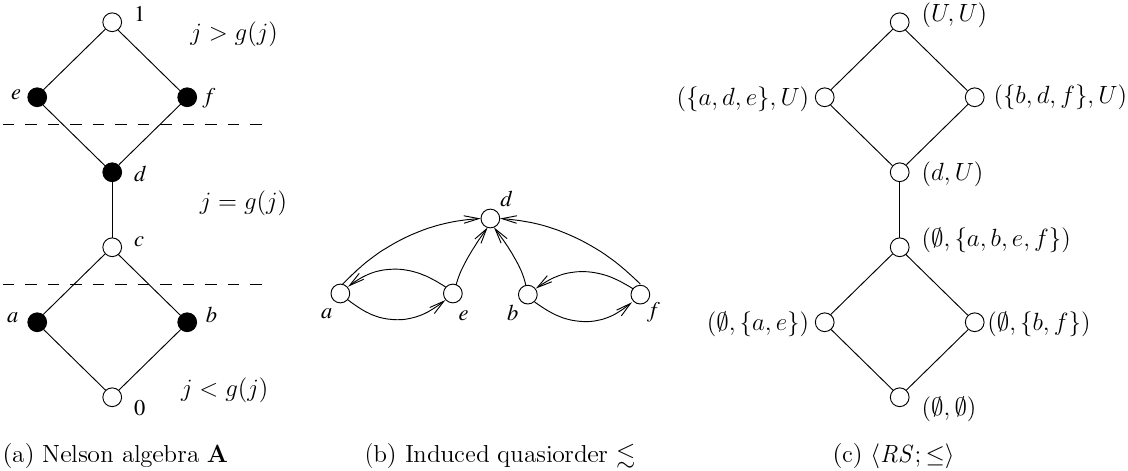}
\caption{\label{Fig:figure1}Constructing $\alg{RS}$ from $\alg{A}$.}
\end{figure}
\end{example}


\begin{proposition}[\cite{Jarv11a}]
If $\lesssim$ is a quasiorder, then the rough set lattice $\mathit{RS}$ is a
three-valued  {\L}ukasiewicz algebra if and only if $\lesssim$ is an equivalence.
\end{proposition}

Recall from Section~\ref{ss:rough_equivalence} that rough sets determined by equivalences
correspond to atomic regular double Stone algebras. Because regular double Stone algebras, semisimple Nelson algebras, 
and three-valued {\L}ukasiewicz algebras are equivalent notions, we obtain that a regular double Stone algebra is atomic if and only if
it is algebraic. In fact, we can write the following proposition.

\begin{proposition} \label{prop:toComer}
Let $\alg{L}$ be a regular double Stone algebra. The following are equivalent:
\begin{enumerate}[label={\rm (\alph*)}]
\item $\alg{L}$ is a direct product of the chains $\mathbf{2}$ and $\mathbf{3}$;
\item $\alg{L}$ is algebraic;
\item $\alg{L}$ is complete and spatial;
\item $\alg{L}$ is complete and atomic.
\end{enumerate}
\end{proposition}

\begin{proof}
Since $\mathbf{2}$ and $\mathbf{3}$ are finite lattices, they are algebraic. As we have already noted, a direct product of algebraic lattices is algebraic.
Thus, (i) implies (ii). An algebraic lattice is complete by definition. Because $\alg{L}$ is a regular double Stone algebra,
it is distributive and self-dual (via the induced $\sim$).
If $\alg{L}$ is algebraic, its dual is also algebraic. Thus, $\alg{L}$ is spatial by the equivalence of conditions (A1)--(A4), and (ii) implies (iii).

Because $\alg{L}$ is a regular double Stone algebra, it is a Heyting algebra as shown by T.~Kat\-ri\-\v{n}\'{a}k \cite[Theorem~1]{Katrinak73}.
Since $\alg{L}$ is a complete lattice, this means that $\alg{L}$ satisfies (JID). Again, by the  equivalence of (A1)--(A4), this means that $\alg{L}$ is algebraic.
Since $\alg{L}$ is algebraic and regular, the set $\mathcal{J}$ of completely
join-irreducible elements has at most two levels; see \cite[Proposition 4.4]{jarvinen_radeleczki_2018}. Let $x$ be a nonzero element of $L$.
Because $\alg{L}$ is spatial, there is
$j \in \mathcal{J}$ such that $j \leq x$. If $j$ is not an atom, then $g(j)$ is an atom below $j$. Therefore, (iii) implies (iv).
The implication (iv)$\Rightarrow$(i) follows explicitly from \cite[Lemma~2.6]{Comer95}.
\end{proof}
Note also that the equivalent conditions of Proposition~\ref{prop:toComer} do not imply that $\alg{L}$ is atomistic, because atomistic pseudocomplemented lattices are
Boolean algebras; see \cite[Theorem 5.2.]{Erne2015}.

Concerning the discriminator varieties, the (subdirectly irreducible) rough set algebras defined by a quasiorder
admit a discriminator term if and only if the quasiorder is an equivalence. This is because Nelson algebras admit a discriminator term if
and only if they form a semisimple variety, and this holds if and only if the relation is an equivalence. In this case, all the other algebras
defined on them in Section~\ref{ss:rough_equivalence} are term equivalent to semisimple Nelson algebras.

\medskip

Let us briefly recall the topological representation of Nelson algebras by D.~Vakarelov \cite{Vaka77}. 
This is needed to have a general
representation of Nelson algebras in terms of rough sets. The construction here resembles the Priestley duality presented in Section~\ref{ss:pri}.
Let $(X,\leq,g)$ be a structure such that $(X,\leq)$ is an ordered set and $g$ is a map on $X$ satisfying the following conditions for all $x,y \in X$:
\begin{enumerate}[label = {\rm (J\arabic*)}] 
\label{Page:Js}
 \item if $x \leq y$, then $g(y) \leq g(x)$,
 \item $g(g(x)) = g(x)$,
 \item $x \leq g(x)$ or $g(x) \leq x$,
 \item if $x,y \leq g(x),g(y)$, then there is $z \in X$ such that $x,y \leq z \leq g(x),g(y)$.
\end{enumerate}
Following Vakarelov, we call such systems \emph{Monteiro spaces}.
Because $\leq$ is an order, the Alexandrov topology $\mathcal{T}_\leq$ forms a $\mathrm{T_0}$-space, that is, 
for any two different points  $x$ and $y$, there is an open set in $\mathcal{T}_\leq$ which contains one of these points and not the other. 
Using the notation of Section~\ref{ss:pri}, we denote the topology $\mathcal{T}_\leq$ defined by the Monteiro space $(X,\leq,g)$ here by $L(X)$.

Each Monteiro space $(X,\leq,g)$ defines a Nelson algebra
\[ \alg{L} (X) = \zseq{L(X); \cup, \cap, \to, {\sim}, \emptyset, X},\]
where the operations ${\sim}$ and $\to$ are defined by:
\[ {\sim} A := \{ x \in X \mid g(x) \notin A \} \quad \text{ and } \quad A \to B := A \Rightarrow ({\sim} A \cup B).\]
The operation $\Rightarrow$ is defined in the Heyting algebra $\mathcal{T}_\leq$ by
\[ B \Rightarrow C = \{ x \in X \mid x \leq y \text{ and } y \in B \text{ imply } y \in C \}.\]

Vakarelov showed that every Nelson algebra $\alg{A}$ can be embedded into the 
Nelson algebra $\alg{L} (X(\alg{A}))$ defined from
the Monteiro space $(X(A),\subseteq,g)$, where $X(\alg{A})$ is the set of prime filters of $\alg{A} = \zseq{A; \vee,\wedge}$ and
the map $g \colon X(\alg{A}) \to X(\alg{A})$ is defined as
\[ g(P) := \{ x \in A \mid {\sim}x \notin P\}.
\]
The embedding $\alg{A} \to \alg{L}(X(\alg{A}))$ is the same as in Section~\ref{ss:pri}, 
that is,
\[ x \mapsto \{ P \in X(A) \mid x \in P\}. \]
Because the induced Nelson algebra $\alg{L} (X(\alg{A})) $ is such that the underlying lattice is algebraic, there exists an isomorphic rough set
Nelson algebra $\alg{RS}$. This implies the following representation theorem given in \cite{Jarv14}.

\begin{theorem} \label{Thm:NelsonRepresentation}
Let $\alg{A}$ be a Nelson algebra. There is a set $U$ and a quasiorder $\lesssim$ on $U$ such that $\alg{A}$
is isomorphic to a subalgebra of $\alg{RS}$.
\end{theorem}

By applying Theorem~\ref{Thm:NelsonRepresentation}, the following completeness result (with the finite model property) for Nelson logic was proved in \cite{Jarv14}. 

\begin{theorem} \label{Thm:CompleteCLSN}
Let $\alpha$ be a formula of the Nelson logic. The following are equivalent:
\begin{enumerate}[label={\rm (\alph*)}]
 \item $\alpha$ is a theorem,
 \item $\alpha$ is valid in every finite rough set-based Nelson algebra determined by a quasiorder.
\end{enumerate}
\end{theorem}
In 1989, P.~A.~Miglioli with his co-authors \cite{Miglioli89a} introduced a constructive logic with strong negation, called \emph{effective logic zero}
and denoted by $E_{0}$, containing a modal operator $\mathbf{T}$ such that for any formula $\alpha$ of $E_{0}$, $\mathbf{T} (\alpha)$
means that $\alpha$ is classically valid. The motivation of the logical system $E_{0}$ was to grasp two distinct aspects of computation in program
synthesis and specification: the algorithmic aspect and data. P.~Pagliani \cite{Pagliani90} gave an algebraic model
for $E_0$, called \emph{effective lattices}. They are special type of Nelson algebras
determined by Glivenko congruences on Heyting algebras. More precisely, for any Heyting algebra $\alg{L}$ and its Glivenko congruence
$\Gamma := \{ (a,b) \in L \times L \mid a^* = b^*\}$, the corresponding \emph{effective lattice} is the Nelson algebra 
$\alg{N}_{\Gamma}(\alg{L})$. Note that for all
$a \in H$, $a \, {\Gamma} \, 1$ if and only if $a$ is dense. This means that 
the universe of $\alg{N}_{\Gamma}(\alg{L})$ 
 consists of the pairs $(a,b)$ such
that $a \wedge b = 0$ and $a \vee b$ is dense (see Remark~2 in \cite{Send90}). It is proved in \cite{Jarv14} that the rough set Nelson algebra
$\alg{RS}$ is an effective lattice iff the set $\mathcal{S}$ of the singletons is cofinal---this means that for all $x \in U$, the principal quasiorder
filter $[x)$ intersects with $\mathcal{S}$.

\medskip

We end this section by considering residuated lattices determined by rough sets. Let
$\alg{A} = \zseq{A; \land, \lor, \to, {\nnot}, 0, 1}$ be a Nelson algebra. 
By recalling from Section~\ref{ss:nasres} the definitions
\[ x * y := {\sim} (x \to {\sim} y) \vee {\sim}(y \to {\sim} x) \quad \text{ and } \quad x \Rightarrow y := (x \to y) \wedge ({\sim} y \to {\sim} x), \]
we obtain a Nelson $\class{FL}_\mathit{ew}$-algebra $\zseq{A; \land, \lor, \imp, 0, 1}$, on which ${\nnot} x = x \imp 0$.

We can define the operations $*$ and $\Rightarrow$ for rough set algebras, obtaining a Nelson $\class{FL}_\mathit{ew}$-algebra
$\zseq{RS; \wedge,\lor, *,\imp, (\emptyset,\emptyset), (U,U)}$, where $\Rightarrow$ is the adjoint operation to $*$, not the relative
pseudocomplement operation considered earlier. Let $(A,B)$ and $(C,D)$ be rough sets. Then, 
\begin{align*}
(A,B) * (C,D) &= {\sim} ((A,B) \to {\sim} (C,D)) \vee {\sim} ((C,D) \to {\sim} (A,B)) \\
  &= {\sim} ((A,B) \to (D^c,C^c)) \vee {\sim} ((C,D) \to (B^c,A^c)) \\
  &= {\sim} ((A^c \cup D^c)^\DOWN, A^c \cup C^c) \vee {\sim} ((C^c \cup B^c)^\DOWN, C^c \cup A^c) \\
  &= ((A^c \cup C^c)^c , (A^c \cup D^c)^{\DOWN c}) \vee ((A^c \cup C^c)^c , (C^c \cup B^c)^{\DOWN c}) \\
  &= (A \cap C, (A \cap D)^\UP) \vee (A \cap C, (B \cap C)^\UP) \\
  &= (A \cap C, (A \cap D)^\UP \cup (B \cap C)^\UP) \\
  &= (A \cap C, ((A \cap D) \cup (B \cap C))^\UP)
\end{align*}
and
\begin{align*}
  (A,B) \Rightarrow (C,D) &= ((A,B) \to (C,D)) \wedge ({\sim} (C;D) \to {\sim} (A,B)) \\
  &= ((Ac \cup C)^\DOWN, A^c \cup D)) \wedge ((D^c,C^c) \to (B^c, A^c)) \\
  &= ((A^c \cup C)^\DOWN, A^c \cup D)) \wedge (D \cup B^c)^\DOWN, D \cup A^c) \\
  &= ((A^c \cup C)^\DOWN \cap (B^c \cup D)^\DOWN, A^c \cup D) \\
  &=  \big ( ((A^c \cup C) \cap (B^c \cup D))^\DOWN, A^c \cup D \big ).
\end{align*}
It can be easily checked that
\[ \big ( (A,B) \Rightarrow (\emptyset,\emptyset) \big ) =  \big ( (A^c \cap B^c)^\DOWN, A^c \big ) =
\big ( B^{c \DOWN}, A^c  \big ) = (B^c, A^C) = {\sim}(A,B).\]
This is because $B \in \mathcal{T}_\gtrsim$ gives $B^c  \in \mathcal{T}_\lesssim$ and so $B^{c \DOWN} = B^c$.

Let us denote by $\Rightarrow^*$ the relative pseudocomplement operation in $\mathit{RS}$.

\begin{proposition} Let $\lesssim$ be a quasiorder on $U$. For all rough sets $(A,B)$ and $(C,D)$,
\[ (A,B) \Rightarrow^* (C,D) \quad \leq \quad (A,B) \Rightarrow (C,D) \quad \leq \quad (A,B) \to (C,D) .\]
\end{proposition}

\begin{proof} Since $B^c \subseteq A^c$, we have
\begin{align*}
  (A,B) \Rightarrow^* (C,D) &= ((A^c \cup C)^\DOWN \cap (B^c \cup D)^{\Down\DOWN}, (B^c \cup D)^\Down) \\
  & \leq ((A^c \cup C)^\DOWN \cap (B^c \cup D)^\DOWN, A^c \cup D) = (A,B) \Rightarrow (C,D) \\
  & \leq ((A^c \cup C)^\DOWN, A^c \cup D) = (A,B) \to (C,D). \qedhere
\end{align*}
\end{proof}

Because Nelson $\class{FL}_\mathit{ew}$-algebras are term equivalent with Nelson algebras, we have that the identities satisfied by
Nelson $\class{FL}_\mathit{ew}$-algebras are the identities satisfied by the finite rough set based residuated lattices determined by a quasiorder.
By Theorem~\ref{Thm:NelsonRepresentation}, a formula $\alpha$ is a theorem iff it is valid in any finite rough set Nelson algebra defined by a quasiorder R.
Therefore, any identity which holds in every finite rough set Nelson $\class{FL}_\mathit{ew}$-algebra holds also in every Nelson $\class{FL}_\mathit{ew}$-algebra.

An element $a$ of a lattice $\alg{L}$ is a \emph{Boolean} if there exists an element $a' \in L$ such that $a \wedge a' = 0$ and $a \vee a' = 1$.
We know that if $\alg{L}$ is distributive, then $a'$ (if it exists) is unique, and is called the \emph{Boolean complement} of $a$.
In view of R.~Cignoli and F.~Esteva \cite{Cignoli2009}, the set of Boolean elements of a bounded residuated lattice 
$\alg{L} =  \zseq{L; \land, \lor, *, \imp, 0, 1}$ are $B(\alg{L}) = \{ x \in L \mid x \vee ( x \imp 0) = 1\}$. 

We say that a rough set $(A,B)$ is \emph{exact} if $A = B$. Note that $X$ is $\lesssim$-definable if and only if $(X^\DOWN,X^\UP)$ is exact.

\begin{proposition} Let $\lesssim$ be a quasiorder. The Boolean elements of the  Nelson $\class{FL}_\mathit{ew}$-algebra
$\zseq{RS; \wedge,\lor, *,\imp, (\emptyset,\emptyset), (U,U)}$ are the exact rough sets.
\end{proposition}

\begin{proof} Suppose that $((A,B)\imp (\emptyset,\emptyset)) \vee (A,B) = (U,U)$ for some rough set $(A,B)$.
Then $(B^c,A^c) \vee (A,B) = (U,U)$ yields $A \cup B^c = U$, that is, $A^c \subseteq B^c$.
Since $A \subseteq B$ is equivalent to $B^c \subseteq A^c$, we obtain $A^c = B^c$ and $A = B$.
This means that $(A,B)$ is an exact rough set. Conversely, if $(A,A)$ is an exact set,
then  $((A,A) \imp (\emptyset,\emptyset) )\vee (A,A) =  {\sim} (A,A) \vee (A,A) = (A^c \cup A, A^c \cup A) = (U,U)$.
\end{proof}

Let us now consider the operations in the case $\mathit{RS}$ is determined by an equivalence $E$. Now
\[
\begin{gathered}
(A,B) * (C,D) = (A \cap C, ((A \cap D) \cup (B \cap C)) ); \\
 (A,B) \Rightarrow (C,D) =  ( (A^c \cup C) \cap (B^c \cup D), A^c \cup D ).
\end{gathered}
\]
In particular, for exact sets $(X,X)$ and $(Y,Y)$,
\[ (X,X) * (Y, Y) = (X \cap Y, X \cap Y) \quad \text{and} \quad (X,X) \rightarrow (Y, Y) = (X^c \cup Y, X^c \cup Y), \]
which are again exact sets. The operation $*$ is equal to $\wedge$ and
$X^c \cup Y$ is the relative pseudocomplement of $X$ with respect to $Y$ in $\zseq{\mathit{Def}(E);\cup,\cap}$.

Recall that an $\class{FL}_\mathit{ew}$-algebra is  \emph{3-potent} if 
it satisfies the identity $x^3 \approx x^2$
As mentioned in Subsection~\ref{ss:fuca},
Nelson $\class{FL}_\mathit{ew}$-algebras are 3-potent. In particular, in the case of equivalences
$(A,B) * (A,B) =  (A \cap A, ((A \cap B) \cup (B \cap A)) ) = (A,A)$. Thus, $(A,B)^k = (A,A)$ for all $k \geq 2$.

\section*{Concluding remarks}

As mentioned in the Introduction, limitations of space and scope have led us to exclude from the present survey 
a number of recent  and interesting developments in the area Nelson-related logics. Having almost concluded our journey, we would now like to spend a few more words on these directions, which may prove the most fruitful in future research. We begin with the topics cited earlier.

\subsection*{Abstract, universal algebraic characterisations 
of Nelson algebras.} As mentioned at the end of Subsection~\ref{ss:fuca},
the paper~\cite{Nasc1z} characterises the variety of $\N$-lattices as precisely the class of $(0,1)$-congruence orderable involutive CIBRLs. Due to its abstract universal algebraic nature, the notion
of $(c,d)$-congruence orderability may be studied in wider contexts, beginning with non-necessarily involutive algebras 
(see below)
but potentially leading to  generalisations that might be 
applicable to non-pointed classes of algebras, where no algebraic constant is term definable (e.g.~$\PN$-lattices).

\subsection*{Investigations of closely related logico--algebraic systems, such as the substructural logic introduced
by Nelson under the name of $\mathcal{S}$.} This logic has been recently shown to be precisely the logic of the class of 3-potent involutive CIBRLs~\cite{Nasc1y,Nasc1x}. It is worth mentioning that another prominent solution to the question of ``how to extend intuitionistic logic with an involutive negation'' is provided by Moisil's \emph{symmetric modal logic}, whose algebraic counterpart is the variety of symmetric modal algebras~\cite{Monteiro1980}. These algebras, which are comparatively little known, certainly deserve further study, both in isolation and in connection with the most recent developments in the theory of $\N$-lattices.

\subsection*{The extension of the theory of Nelson algebras beyond  the involutive setting.} This project has proved to be  quite fruitful, leading to a substantial research output~\cite{rivieccio2021quasi,riviecciotwonegs,rivieccio2021fragments,rivieccio2019quasi,riviecciosp2021quasi,thiago2021negation}. In our opinion,  two particularly promising lines of research are worth mentioning. Firstly, the abstract study of the twist construction and the limits of its applicability within the setting of substructural logics (see, in particular,~\cite{BuGa}).
Secondly, and most importantly for our present interest, the investigation of potential connections with rough set theory. Indeed,  $\PN$-lattices and the quasi-Nelson algebras introduced in~\cite{rivieccio2019quasi}
are both mild generalisations of $\N$-lattices, obtained by abandoning, respectively,
the explosive and the involutive law. It seems thus natural to speculate whether some suitable generalisation of rough sets induced by quasiorders might provide sufficiently general classes of structures, which one could match with the above-mentioned varieties. 

\subsection*{Structures arising from rough sets.}
Both the algebraic and the order-theoretic structure of rough sets 
systems defined by equivalences were thoroughly established during the
1990s. However, exploring structures induced by various types of relations has yielded fruitful results. Rough sets defined by quasiorders were in a central role in Section~\ref{sec:rough}, where we
discussed how rough sets defined by quasiorders correspond the Nelson
algebras defined on algebraic (e.g. finite) lattices.  Hence, in view of our formerly discussed results, they can be presented as algebras defined on residuated lattices as well. We also know
that rough sets defined by tolerances (reflexive and symmetric
binary relations) induced by irredundant coverings
are (up to isomorphism) the regular pseudocomplemented Kleene algebras defined on algebraic lattices
\cite{jarvinen_radeleczki_2018}.
It is worth noting that, for an arbitrary tolerance, the ordered set defined on the rough set algebra $\mathit{RS}$ is not necessarily a lattice, as highlighted in \cite{Jarvinen2007}. 
In \cite{Umadevi2015}, D.~Umadevi  presented the 
Dedekind-MacNeille completion of $\mathit{RS}$ for
arbitrary binary relations. Subsequently, it was shown in \cite{Jarvinen2023} that, for reflexive relations, this completion forms a pseudo-Kleene algebra. Moreover, a so-called PBZ-lattice can always be defined on it. It is important to notice that 
covering-based rough sets provide another way to generalize  approximations based on equivalences, and there exists an extensive body of literature related to them, as discussed in \cite{Yao2012}. This diversity of approaches means that there are numerous interesting algebraic structures emerging from rough sets defined 
in various settings.

\section*{Acknowledgements}

The idea of writing a survey on Nelson algebras and rough sets was originally proposed by Matthew Spinks, 
and we agreed he would be the fourth author of the present paper. 
Health issues unfortunately obliged him to abandon the project; he did, however,   contribute  
to earlier versions of the paper and   was able to provide us with useful feedback throughout the rest of the  writing process. 
It is therefore only fair to take this occasion for thanking Matthew and  acknowledging his contribution.

We also extend our gratitude for the valuable suggestions provided by the anonymous reviewers.

\bibliographystyle{abbrv}
\bibliography{survey.bib}

\begin{thebibliography}{100}

\bibitem{agliano2021varieties}
P.~Aglian{\`o} and M.~A. Marcos.
\newblock Varieties of {K}-lattices.
\newblock {\em Fuzzy Sets and Systems}, 2021.

\bibitem{Alex37}
P.~Alexandroff.
\newblock Diskrete {R}{\"a}ume.
\newblock {\em Matemati{\v{c}}eskij Sbornik}, 2:501--518, 1937.

\bibitem{Almu84}
A.~Almukdad and D.~Nelson.
\newblock Constructible falsity and inexact predicates.
\newblock {\em Journal of Symbolic Logic}, 49:231--233, 1984.

\bibitem{Balb74}
R.~Balbes and P.~Dwinger.
\newblock {\em Distributive {L}attices}.
\newblock University of Missouri Press, Columbia, 1974.

\bibitem{Banerjee97}
M.~Banerjee.
\newblock Rough sets and 3-valued {{\L}}ukasiewicz logic.
\newblock {\em Fundamenta Informaticae}, 31:213--220, 1997.

\bibitem{BanChak96}
M.~Banerjee and M.~K. Chakraborty.
\newblock Rough sets through algebraic logic.
\newblock {\em Fundamenta Informaticae}, 28:211--221, 1996.

\bibitem{Banerjee2004}
M.~Banerjee and M.~K. Chakraborty.
\newblock Algebras from rough sets.
\newblock In S.~K. Pal, L.~Polkowski, and A.~Skowron, editors, {\em
  Rough-Neural Computing: Techniques for Computing with Words}, pages 157--184.
  Springer, Berlin, Heidelberg, 2004.

\bibitem{Beln77}
N.~D. Belnap.
\newblock A useful four-valued logic.
\newblock In J.~M. Dunn and G.~Epstein, editors, {\em {Modern Uses of
  Multiple-Valued Logics. Invited papers from the fifth International Symposium
  on Multiple-Valued Logic held at Indiana University, Bloomington, Indiana,
  May 13--16, 1975, with a bibliography of many-valued logic by Robert. G.
  Wolf}}, volume~2 of {\em Episteme. A Series in the Foundational,
  Methodological, Philosophical, Psychological, Sociological and Political
  Aspects of the Sciences, Pure and Applied}, pages 8--37. D. Riedel Publishing
  Company, Dordrecht, 1977.

\bibitem{Bial57}
A.~Bialynicki-Birula and H.~Rasiowa.
\newblock On the representation of quasi-{B}oolean algebras.
\newblock {\em Journal of Symbolic Logic}, 22(4), 1957.

\bibitem{Birk37}
G.~Birkhoff.
\newblock Rings of sets.
\newblock {\em Duke Mathematical Journal}, {3}:443--454, 1937.

\bibitem{Birk67}
G.~Birkhoff.
\newblock {\em Lattice {T}heory}.
\newblock Number~25 in Colloquium Publ. American Mathematical Society,
  Providence, 3rd edition, 1967.

\bibitem{Blok99b}
W.~J. Blok and B.~J\'{o}nsson.
\newblock Algebraic structures for logic.
\newblock Twenty-third Holiday Mathematics Symposium, New Mexico State
  University, Lecture Notes, Jan. 1999.

\bibitem{Blok84}
W.~J. Blok, P.~K\"{o}hler, and D.~Pigozzi.
\newblock On the structure of varieties with equationally definable principal
  congruences {II}.
\newblock {\em Algebra Universalis}, 18:334--379, 1984.

\bibitem{Blok89}
W.~J. Blok and D.~Pigozzi.
\newblock Algebraizable logics.
\newblock {\em Memoirs of the American Mathematical Society}, 77(396), 1989.

\bibitem{Blok94a}
W.~J. Blok and D.~Pigozzi.
\newblock On the structure of varieties with equationally definable principal
  congruences {III}.
\newblock {\em Algebra Universalis}, 32:545--608, 1994.

\bibitem{Blyt05}
T.~Blyth.
\newblock {\em Lattices and {O}rdered {A}lgebraic {S}tructures}.
\newblock Springer-Verlag, London, 2005.

\bibitem{Boic91}
V.~Boicescu, A.~Filipoiu, G.~Georgescu, and S.~Rudeanu.
\newblock {\em {\L}ukasiewicz--{M}oisil algebras}.
\newblock North-Holland Publishing Company, Amsterdam, 1991.

\bibitem{Bonikowski92}
Z.~Bonikowski.
\newblock A certain conception of the calculus of rough sets.
\newblock {\em Notre Dame J. Formal Log.}, 33(3):412--121, 1992.

\bibitem{BoRi11}
F.~Bou and U.~Rivieccio.
\newblock The logic of distributive bilattices.
\newblock {\em Logic Journal of the IGPL. Interest Group in Pure and Applied
  Logic}, 19(1):183--216, 2011.

\bibitem{Bou13}
F.~Bou and U.~Rivieccio.
\newblock Bilattices with implications.
\newblock {\em Studia Logica}, 101:651--675, 2013.

\bibitem{Brad91}
R.~T. Brady.
\newblock Gentzenization and decidability of some contraction-less relevant
  logics.
\newblock {\em Journal of Philosophical Logic}, 20:97--117, 1991.

\bibitem{Brig69}
D.~Brignole.
\newblock Equational characterisation of {N}elson algebra.
\newblock {\em Notre Dame Journal of Formal Logic}, 10:285--297, 1969.

\bibitem{BuO}
J.~B\"{u}chi and T.~Owens.
\newblock Skolem rings and their varieties.
\newblock In S.~M. Lane and D.~Siefkes, editors, {\em The Collected Works of J.
  Richard B\"{u}chi}, pages 53--80. Springer, New York, 1975.

\bibitem{Burr81}
S.~Burris and H.~P. Sankappanavar.
\newblock {\em A {C}ourse in {U}niversal {A}lgebra}, volume~78 of {\em Graduate
  Texts in Mathematics}.
\newblock Springer-Verlag, New York, 1981.

\bibitem{Busa09}
M.~Busaniche and R.~Cignoli.
\newblock Residuated lattices as algebraic semantics for paraconsistent
  {N}elson's logic.
\newblock {\em Journal of Logic and Computation}, 19:1019--1029, 2009.

\bibitem{Busa14}
M.~Busaniche and R.~Cignoli.
\newblock The subvariety of commutative residuated lattices represented by
  twist-products.
\newblock {\em Algebra Universalis}, 71:5--22, 2014.

\bibitem{BuGa}
M.~Busaniche, N.~Galatos, and M.~A. Marcos.
\newblock Twist structures and {N}elson conuclei.
\newblock {\em Studia Logica}, 110(4):949--987, 2022.

\bibitem{CaCrPr15}
L.~M. Cabrer, A.~P. Craig, and H.~A. Priestley.
\newblock Product representation for default bilattices: an application of
  natural duality theory.
\newblock {\em Journal of pure and applied algebra}, 219(7):2962--2988, 2015.

\bibitem{CaPr15c}
L.~M. Cabrer and H.~A. Priestley.
\newblock A general framework for product representations: bilattices and
  beyond.
\newblock {\em Logic Journal of the IGPL. Interest Group in Pure and Applied
  Logic}, 23(5):816--841, 2015.

\bibitem{CaPr16}
L.~M. Cabrer and H.~A. Priestley.
\newblock Natural dualities through product representations: {B}ilattices and
  beyond.
\newblock {\em Studia Logica}, 104(3):567--592, 2016.

\bibitem{Chang58}
C.~C. Chang.
\newblock Algebraic analysis of many valued logics.
\newblock {\em Transactions of the American Mathematical Society}, 88:467--490,
  1958.

\bibitem{Cign86}
R.~Cignoli.
\newblock The class of {K}leene algebras satisfying an interpolation property
  and {N}elson algebras.
\newblock {\em Algebra Universalis}, 23:262--282, 1986.

\bibitem{Cignoli2009}
R.~Cignoli and F.~Esteva.
\newblock Commutative integral bounded residuated lattices with an added
  involution.
\newblock {\em Annals of Pure and Applied Logic}, 161:150--160, 2009.

\bibitem{Come93}
S.~D. Comer.
\newblock On connections between information systems, rough sets, and algebraic
  logic.
\newblock In C.~Rauszer, editor, {\em Algebraic {M}ethods in {L}ogic and
  {C}omputer {S}cience}, volume~28 of {\em Banach Centre Publications}, pages
  117--124. Institute of Mathematics, Polish Academy of Sciences, Warszawa,
  1993.

\bibitem{Comer95}
S.~D. Comer.
\newblock Perfect extensions of regular double {S}tone algebras.
\newblock {\em Algebra Universalis}, 34:96--109, 1995.

\bibitem{CoFo77}
W.~H. Cornish and P.~R. Fowler.
\newblock Coproducts of {D}e {M}organ algebras.
\newblock {\em Bull. Austral. Math. Soc.}, 16(1):1--13, 1977.

\bibitem{CoFo79}
W.~H. Cornish and P.~R. Fowler.
\newblock Coproducts of {K}leene algebras.
\newblock {\em J. Austral. Math. Soc. Ser. A}, 27(2):209--220, 1979.

\bibitem{DaCa74}
N.~da~Costa.
\newblock On the theory of inconsistent formal systems.
\newblock {\em Notre Dame Journal of Formal Logic}, 15:497--510, 1974.

\bibitem{Dai08}
J.-H. Dai.
\newblock Rough 3-valued algebras.
\newblock {\em Information Sciences}, 178:1986--1996, 2008.

\bibitem{Dave02}
B.~Davey and H.~Priestley.
\newblock {\em Introduction to {L}attices and {O}rder}.
\newblock Cambridge University Press, Cambridge, 2nd edition, 2002.

\bibitem{DemOrl02}
S.~P. Demri and E.~S. Or{\l}owska.
\newblock {\em Incomplete Information: Structure, Inference, Complexity}.
\newblock Springer, Berlin Heidelberg, 2002.

\bibitem{Dunt97}
I.~D\"{u}ntsch.
\newblock A logic for rough sets.
\newblock {\em Theoretical Computer Science}, 1--2:427--436, 1997.

\bibitem{Erne2015}
M.~Ern{\'e} and V.~Joshi.
\newblock Ideals in atomic posets.
\newblock {\em Discrete Mathematics}, 338:954--971, 2015.

\bibitem{Es74}
L.~Esakia.
\newblock Topological {K}ripke models.
\newblock {\em Soviet Math. Dokl.}, 15:147--151, 1974.

\bibitem{Fide77}
M.~Fidel.
\newblock The decidability of the calculi {$C_n$}.
\newblock {\em Reports on Mathematical Logic}, 8:31--40, 1977.

\bibitem{Fide78}
M.~M. Fidel.
\newblock An algebraic study of a propositional system of {N}elson.
\newblock In {\em Mathematical Logic. Proceedings of the First Brazilian
  Conference, Campinas 1977}, volume~39 of {\em Lecture Notes in Pure and
  Applied Mathematics}, pages 99--117, New York, 1978. Marcel Dekker.

\bibitem{Fide80}
M.~M. Fidel.
\newblock An algebraic study of constructive logic with strong negation.
\newblock In {\em Proceedings of the Third Brazilian Conference on Mathematical
  Logic, Recife 1979}, pages 119--129, Campinas, 1980. Sociedade Brasileira de
  Logica.

\bibitem{Font97}
J.~M. Font.
\newblock Belnap's four-valued logic and {D}e {M}organ lattices.
\newblock {\em Logic Journal of the IGPL. Interest Group in Pure and Applied
  Logic}, 5:413--440, 1997.

\bibitem{Font84}
J.~M. Font, A.~J. Rodr\'{\i}guez, and A.~Torrens.
\newblock Wajsberg algebras.
\newblock {\em Stochastica}, 8:5--31, 1984.

\bibitem{Frie83}
E.~Fried and E.~W. Kiss.
\newblock Connections between congruence-lattices and polynomial properties.
\newblock {\em Algebra Universalis}, 17:227--262, 1983.

\bibitem{GaJiKoOn07}
N.~Galatos, P.~Jipsen, T.~Kowalski, and H.~Ono.
\newblock {\em Residuated lattices: an algebraic glimpse at substructural
  logics}, volume 151 of {\em Studies in Logic and the Foundations of
  Mathematics}.
\newblock Elsevier, Amsterdam, 2007.

\bibitem{Ganter07}
B.~Ganter.
\newblock Non-symmetric indiscernibility.
\newblock In {\em Knowledge Processing and Data Analysis}, volume 6581 of {\em
  Lecture Notes in Computer Science}, pages 26--34. Springer, 2007.

\bibitem{GeWa92}
M.~Gehrke and E.~Walker.
\newblock On the structure of rough sets.
\newblock {\em Bulletin of Polish Academy of Sciences, Mathematics},
  40:235--245, 1992.

\bibitem{ghorbani2016hoop}
S.~Ghorbani.
\newblock Hoop twist-structures.
\newblock {\em Journal of Applied Logic}, 18:1--18, 2016.

\bibitem{gierz2003}
G.~Gierz, K.~H. Hofmann, K.~Keimel, J.~D. Lawson, M.~Mislove, and D.~S. Scott.
\newblock {\em Continuous Lattices and Domains}.
\newblock Encyclopedia of Mathematics and its Applications. Cambridge
  University Press, 2003.

\bibitem{Gi88}
M.~L. Ginsberg.
\newblock Multivalued logics: {A} uniform approach to inference in artificial
  intelligence.
\newblock {\em Computational Intelligence}, 4:265--316, 1988.

\bibitem{Grat78}
G.~Gr\"{a}tzer.
\newblock {\em General {L}attice {T}heory}.
\newblock Birkh\"{a}user Verlag, Basel und Stuttgart, 1978.

\bibitem{gratzer1958generalized}
G.~Gr{\"a}tzer and E.~Schmidt.
\newblock On the generalized {B}oolean algebra generated by a distributive
  lattice.
\newblock {\em Indag. Math}, 20:547--553, 1958.

\bibitem{Idzi09}
P.~M. Idziak, K.~S{\l}omczy\'{n}ska, and A.~Wro\'{n}ski.
\newblock Fregean varieties.
\newblock {\em International Journal of Algebra and Computation}, 19:595--645,
  2009.

\bibitem{Ittu98}
L.~Iturrioz.
\newblock Rough sets and three-valued structures.
\newblock In E.~Orlowska, editor, {\em {L}ogic at {W}ork. {E}ssays {D}edicated
  to the {M}emory of {H}elena {R}asiowa}, pages 596--603. Physica-Verlag,
  Berlin, 1998.

\bibitem{Iwinski87}
T.~B. Iwi{\'n}ski.
\newblock Algebraic approach to rough sets.
\newblock {\em Bulletin of the Polish Academy of Sciences, Mathematics},
  35:673--683, 1987.

\bibitem{jakl2016bitopology}
T.~Jakl, A.~Jung, and A.~Pultr.
\newblock Bitopology and four-valued logic.
\newblock {\em Electronic Notes in Theoretical Computer Science}, 325:201--219,
  2016.

\bibitem{JaRi12}
R.~Jansana and U.~Rivieccio.
\newblock Residuated bilattices.
\newblock {\em Soft Computing}, 16(3):493--504, 2012.

\bibitem{Jans13}
R.~Jansana and U.~Rivieccio.
\newblock Priestley duality for {N4}-lattices.
\newblock In {\em Eighth Conference of the European Society for Fuzzy Logic and
  Technology (EUSFLAT 2013). Proceedings}, volume~32 of {\em Advances in
  Intelligent Systems Research (AISR)}, pages 223--229, Amsterdam, 2013.
  Atlantis Press.

\bibitem{JaRi14}
R.~Jansana and U.~Rivieccio.
\newblock Dualities for modal {N4}-lattices.
\newblock {\em Logic Journal of the IGPL. Interest Group in Pure and Applied
  Logic}, 22(4):608--637, 2014.

\bibitem{Jarvinen2007}
J.~J\"arvinen.
\newblock Lattice theory for rough sets.
\newblock {\em Transaction on Rough Sets}, VI:400--498, 2007.

\bibitem{JarvPaglRade11}
J.~J\"{a}rvinen, P.~Pagliani, and S.~Radeleczki.
\newblock Atomic information completeness in generalised rough set systems.
\newblock In {\em Extended abstracs of the 3rd international workshop on rough
  set theory (RST11)}, pages 14--16, 2011.

\bibitem{Jarv13}
J.~J\"{a}rvinen, P.~Pagliani, and S.~Radeleczki.
\newblock Information completeness in {N}elson algebras of rough sets induced
  by quasiorders.
\newblock {\em Studia Logica}, 101:1073--1092, 2013.

\bibitem{Jarv11a}
J.~J\"{a}rvinen and S.~Radeleczki.
\newblock Representation of {N}elson algebras by rough sets determined by
  quasiorders.
\newblock {\em Algebra Universalis}, 66:163--179, 2011.

\bibitem{Jarv14}
J.~J\"{a}rvinen and S.~Radeleczki.
\newblock Monteiro spaces and rough sets determined by quasiorder relations:
  {M}odels for {N}elson algebras.
\newblock {\em Fundamenta Informaticae}, 131:205--215, 2014.

\bibitem{jarvinen_radeleczki_2018}
J.~J{\"a}rvinen and S.~Radeleczki.
\newblock Representing regular pseudocomplemented {K}leene algebras by
  tolerance-based rough sets.
\newblock {\em Journal of the Australian Mathematical Society}, 105(1):57--78,
  2018.

\bibitem{JarRad21}
J.~J{\"a}rvinen and S.~Radeleczki.
\newblock Defining rough sets as core--support pairs of three-valued functions.
\newblock {\em International Journal of Approximate Reasoning}, 135:71--90,
  2021.

\bibitem{Jarvinen2023}
J.~J{\"a}rvinen and S.~Radeleczki.
\newblock {Pseudo-Kleene} algebras determined by rough sets.
\newblock {\em International Journal of Approximate Reasoning}, 161:108991,
  2023.

\bibitem{JarRadVer09}
J.~J{\"{a}}rvinen, S.~Radeleczki, and L.~Veres.
\newblock Rough sets determined by quasiorders.
\newblock {\em Order}, 26:337--355, 2009.

\bibitem{Jons95}
B.~J\'{o}nsson.
\newblock Congruence distributive varieties.
\newblock {\em Mathematica Japonica}, 42:353--401, 1995.

\bibitem{JuRi12}
A.~Jung and U.~Rivieccio.
\newblock Priestley duality for bilattices.
\newblock {\em Studia Logica}, 100(1-2):223--252, 2012.

\bibitem{Kami12}
N.~Kamide and H.~Wansing.
\newblock Proof theory of {N}elson's paraconsistent logic: {A} uniform
  perspective.
\newblock {\em Theoretical Computer Science}, 415:1--38, 2012.

\bibitem{Katrinak73}
T.~Katri\v{n}\'{a}k.
\newblock The structure of distributive double $p$-algebras. {R}egularity and
  congruences.
\newblock {\em AlgeU}, 3:238--246, 1973.

\bibitem{Katrinak74}
T.~Katri\v{n}{\'a}k.
\newblock Construction of regular double $p$-algebras.
\newblock {\em Bulletin de la Soci{\'e}t{\'e} Royale des Sciences de
  Li{\`e}ge}, 43:238--246, 1974.

\bibitem{Kohl80}
P.~K\"{o}hler and D.~Pigozzi.
\newblock Varieties with equationally definable principal congruences.
\newblock {\em Algebra Universalis}, 11:213--219, 1980.

\bibitem{Kortelainen1994}
J.~Kortelainen.
\newblock On relationship between modified sets, topological spaces and rough
  sets.
\newblock {\em Fuzzy Sets and Systems}, 61:91--95, 1994.

\bibitem{KumBan2012}
A.~Kumar and M.~Banerjee.
\newblock Definable and rough sets in covering-based approximation spaces.
\newblock In {\em RSKT~2012}, volume 7414 of {\em Lecture Notes in Computer
  Science}, pages 488--495. Springer, 2012.

\bibitem{KumBan2015}
A.~Kumar and M.~Banerjee.
\newblock Algebras of definable and rough sets in quasi order-based
  approximation spaces.
\newblock {\em Fundamenta Informaticae}, 141:37--55, 2015.

\bibitem{Kumar2017}
A.~Kumar and M.~Banerjee.
\newblock A semantic analysis of {S}tone and dual {S}tone negations with
  regularity.
\newblock In S.~Ghosh and S.~Prasad, editors, {\em Logic and Its Applications},
  pages 139--153, Berlin, Heidelberg, 2017. Springer.

\bibitem{Lope72}
E.~G.~K. L\'{o}pez-Escobar.
\newblock Refutability and elementary number theory.
\newblock {\em Indagationes Mathematicae (Proceedings)}, 75(4):362--374, 1972.

\bibitem{Luka70d}
J.~{\L}ukasiewicz.
\newblock {Philosophical remarks on many-valued systems of propositional
  logic}.
\newblock In {L. Borkowski}, editor, {\em {Selected Works of Jan
  {\L}ukasiewicz}}, {Studies in Logic and the Foundation of Mathematics}, pages
  153--178. {North-Holland Publishing Company}, Warsaw, 1970.

\bibitem{Miglioli89a}
P.~Miglioli, U.~Moscato, M.~Ornaghi, and G.~Usberti.
\newblock A constructivism based on classical truth.
\newblock {\em Notre Dame Journal of Formal Logic}, 30:67--90, 1989.

\bibitem{MoPiSlVo00}
B.~Mobasher, D.~Pigozzi, G.~Slutzki, and G.~Voutsadakis.
\newblock A duality theory for bilattices.
\newblock {\em Algebra Universalis}, 43(2-3):109--125, 2000.

\bibitem{Moisil63}
G.~C. Moisil.
\newblock Les logiques non-chrysipiennes et leurs applications.
\newblock {\em Acta Philosophica Fennica}, 16:137--152, 1965.
\newblock Proceedings of a colloquium on Modal and many-valued logics,
  Helsinki, 23--26 August, 1962.

\bibitem{Mont63b}
A.~Monteiro.
\newblock Construction des alg\`{e}bres de {N}elson finies.
\newblock {\em Bull. Acad. Polon. Sci. S{\'e}r. Sci. Math. Astronom. Phys.},
  11:359--362, 1963.

\bibitem{Monteiro63c}
A.~Monteiro.
\newblock Sur la d\'{e}finition des alg\`{e}bres de {{\L}ukasiewicz}
  trivalentes.
\newblock {\em Bulletin Math\'{e}matique de la Soci\'{e}t\'{e} Scientifique
  Math\'{e}matique Physique R.~P.~Roumanie}, 7:3--12, 1963.

\bibitem{Monteiro1980}
A.~Monteiro.
\newblock Sur les algèbres de {H}eyting symétriques.
\newblock {\em Portugaliae Mathematica}, 39(1-4):1--237, 1980.

\bibitem{Mundici86}
D.~Mundici.
\newblock {MV}-algebras are categorically equivalent to bounded commutative
  {BCK}-algebras.
\newblock {\em Math. Japonica}, 31:889--894, 1986.

\bibitem{Umadevi13a}
E.~K.~R. Nagarajan and D.~Umadevi.
\newblock Algebra of rough sets based on quasi order.
\newblock {\em Fundam. Informaticae}, 126:83--101, 2013.

\bibitem{Nasc1z}
T.~{Nascimento}, J.~{Rivieccio}, and M.~{Spinks}.
\newblock Compatibly involutive residuated lattices and the {N}elson identity.
\newblock {\em Soft Computing}, 23:2297--2320, 2019.

\bibitem{thiago2021negation}
T.~Nascimento and U.~Rivieccio.
\newblock Negation and implication in quasi-{N}elson logic.
\newblock {\em Logical Investigations}, 27(1):107--123, 2021.

\bibitem{Nasc1y}
T.~Nascimento, U.~Rivieccio, J.~Marcos, and M.~Spinks.
\newblock Algebraic semantics for {N}elson's logic $\mathcal{S}$.
\newblock In L.~Moss, R.~de~Queiroz, and M.~Martinez, editors, {\em Logic,
  {L}anguage, {I}nformation, and {C}omputation. WoLLIC 2018}, volume 10944 of
  {\em Lecture Notes in Computer Science}, pages 271--288, Berlin, Heidelberg,
  2018. Springer.

\bibitem{Nasc1x}
T.~{Nascimento}, U.~{Rivieccio}, J.~{Marcos}, and M.~{Spinks}.
\newblock Nelson's logic $\mathcal{S}$.
\newblock {\em Logic Journal of the IGPL. Interest Group in Pure and Applied
  Logic}, 28(6):1182--1206, 2020.

\bibitem{Nels49}
D.~Nelson.
\newblock Constructible falsity.
\newblock {\em Journal of Symbolic Logic}, 14:16--26, 1949.

\bibitem{Nels59}
D.~Nelson.
\newblock Negation and separation of concepts in constructive systems.
\newblock In A.~Heyting, editor, {\em Constructivity in Mathematics}, pages
  208--225. North-Holland Publishing Company, Amsterdam, 1959.

\bibitem{Odin03}
S.~P. Odintsov.
\newblock Algebraic semantics for paraconsistent {N}elson's logic.
\newblock {\em Journal of Logic and Computation}, pages 453--468, 2003.

\bibitem{Odin04}
S.~P. Odintsov.
\newblock On the representation of {N4}-lattices.
\newblock {\em Studia Logica}, 76:385--405, 2004.

\bibitem{Odin05}
S.~P. Odintsov.
\newblock The class of extensions of {N}elson's paraconsistent logic.
\newblock {\em Studia Logica}, 80:291--320, 2005.

\bibitem{Odin06}
S.~P. Odintsov.
\newblock Transfer theorems for extensions of the paraconsistent {N}elson
  logic.
\newblock {\em Algebra and Logic}, 45:232--247, 2006.

\bibitem{Odin07}
S.~P. Odintsov.
\newblock {\em Constructive Negations and Paraconsistency}.
\newblock PhD thesis, Sobolev Institute of Mathematics, Russian Federation,
  2007.

\bibitem{Odin08}
S.~P. Odintsov.
\newblock {\em Constructive Negations and Paraconsistency}, volume~26 of {\em
  Trends in Logic. Studia Logica Library}.
\newblock Springer, Berlin, 2008.

\bibitem{Od10}
S.~P. Odintsov.
\newblock Priestley duality for paraconsistent {N}elson's logic.
\newblock {\em Studia Logica}, 96(1):65--93, 2010.

\bibitem{OnRi14}
H.~Ono and U.~Rivieccio.
\newblock Modal twist-structures over residuated lattices.
\newblock {\em Logic Journal of the IGPL. Interest Group in Pure and Applied
  Logic}, 22(3):440--457, 2014.

\bibitem{Pagliani90}
P.~Pagliani.
\newblock Remarks on special lattices and related constructive logics with
  strong negation.
\newblock {\em Notre Dame Journal of Formal Logic}, 31:515--528, 1990.

\bibitem{Pagl96}
P.~Pagliani.
\newblock Rough sets and {N}elson algebras.
\newblock {\em Fundamenta Informaticae}, 27:205--219, 1996.

\bibitem{Panicker19}
G.~Panicker and M.~Banerjee.
\newblock Rough sets and the algebra of conditional logic.
\newblock In T.~Mih{\'a}lyde{\'a}k, F.~Min, G.~Wang, M.~Banerjee,
  I.~D{\"u}ntsch, and Z.~Suraj, editors, {\em Rough Sets}, pages 28--39, Cham,
  2019. Springer.

\bibitem{Pawlak82}
Z.~Pawlak.
\newblock Rough sets.
\newblock {\em International Journal of Computer and Information Sciences},
  11:341--356, 1982.

\bibitem{PomPom88}
J.~Pomyka{\l}a and J.~A. Pomyka{\l}a.
\newblock The {S}tone algebra of rough sets.
\newblock {\em Bulletin of Polish Academy of Sciences. Mathematics},
  36:495--512, 1988.

\bibitem{priestley1970representation}
H.~A. Priestley.
\newblock Representation of distributive lattices by means of ordered {S}tone
  spaces.
\newblock {\em Bulletin of the London Mathematical Society}, 2(2):186--190,
  1970.

\bibitem{Pr84}
H.~A. Priestley.
\newblock Ordered sets and duality for distributive lattices.
\newblock In M.~Pouzet and D.~Richard, editors, {\em Orders: {D}escriptions and
  {R}oles}, volume~23 of {\em Annals of Discrete Mathematics}, pages 39--60.
  North-Holland, 1984.

\bibitem{Pynk99b}
A.~P. Pynko.
\newblock Definitional equivalence and algebraizability of generalised logical
  systems.
\newblock {\em Annals of Pure and Applied Logic}, 98:1--68, 1999.

\bibitem{Pynk99}
A.~P. Pynko.
\newblock Functional completeness and axiomatizability within {B}elnap's
  four-valued logic and its expansions.
\newblock {\em Journal of Applied Non-classical Logics}, 9:61--105, 1999.

\bibitem{Qiao2012}
Q.~Qiao.
\newblock Topological structure of rough sets in reflexive and transitive
  relations.
\newblock In {\em 2012 5th International Conference on BioMedical Engineering
  and Informatics}, pages 1585--1589, 2012.

\bibitem{Rasi74}
H.~Rasiowa.
\newblock {\em An {A}lgebraic {A}pproach to {N}on-{C}lassical {L}ogics},
  volume~78 of {\em Studies in Logic and the Foundations of Mathematics}.
\newblock North-Holland Publishing Company, Amsterdam, 1974.

\bibitem{Rivi11}
U.~Rivieccio.
\newblock Paraconsistent modal logics.
\newblock {\em Electronic Notes in Theoretical Computer Science}, 278:173--186,
  2011.

\bibitem{Ri14}
U.~Rivieccio.
\newblock Implicative twist-structures.
\newblock {\em Algebra Universalis}, 71(2):155--186, 2014.

\bibitem{riviecciotwonegs}
U.~Rivieccio.
\newblock Fragments of quasi-{N}elson: two negations.
\newblock {\em Journal of Applied Logic}, 7:499--559, 2020.

\bibitem{Ri20}
U.~Rivieccio.
\newblock Representation of {D}e {M}organ and (semi-){K}leene lattices.
\newblock {\em Soft Computing}, 24(12):8685--8716, 2020.

\bibitem{rivieccio2021fragments}
U.~Rivieccio.
\newblock Fragments of quasi-{N}elson: The algebraizable core.
\newblock {\em Logic Journal of the IGPL. Interest Group in Pure and Applied
  Logic}, 30(5):807--839, 2021.

\bibitem{Rivi0xa}
U.~Rivieccio, F.~Bou, and R.~Jansana.
\newblock Varieties of interlaced bilattices.
\newblock {\em Algebra Universalis}, 66:115--141, 2011.

\bibitem{WNM}
U.~Rivieccio, T.~Flaminio, and T.~Nascimento.
\newblock On the representation of (weak) nilpotent minimum algebras.
\newblock In {\em 2020 IEEE International Conference on Fuzzy Systems
  (FUZZ-IEEE)}, pages 1--8. IEEE, 2020.

\bibitem{rivieccio2021quasi}
U.~Rivieccio and R.~Jansana.
\newblock Quasi-{N}elson algebras and fragments.
\newblock {\em Mathematical Structures in Computer Science}, 31(3):257–285,
  2021.

\bibitem{rivieccio2021duality}
U.~Rivieccio and A.~Jung.
\newblock A duality for two-sorted lattices.
\newblock {\em Soft Computing}, 25(2):851--868, 2021.

\bibitem{rivieccio2017four}
U.~Rivieccio, A.~Jung, and R.~Jansana.
\newblock Four-valued modal logic: Kripke semantics and duality.
\newblock {\em Journal of Logic and Computation}, 27(1):155--199, 2017.

\bibitem{rivieccio2019quasi}
U.~Rivieccio and M.~Spinks.
\newblock Quasi-{N}elson algebras.
\newblock {\em Electronic Notes in Theoretical Computer Science}, 344:169--188,
  2019.

\bibitem{riviecciosp2021quasi}
U.~Rivieccio and M.~Spinks.
\newblock Quasi-{N}elson; or, non-involutive {N}elson algebras.
\newblock In {\em Algebraic Perspectives on Substructural Logics}, pages
  133--168. Springer, 2021.

\bibitem{Rout74}
R.~Routley.
\newblock Semantical analyses of propositional systems of {F}itch and {N}elson.
\newblock {\em Studia Logica}, 33:283--298, 1974.

\bibitem{sendlewski1984topological}
A.~Sendlewski.
\newblock Topological duality for {N}elson algebras and its applications.
\newblock {\em Bulletin of the Section of Logic, Polish Academy of Sciences},
  13:215--221, 1984.

\bibitem{Send90}
A.~Sendlewski.
\newblock Nelson algebras through {H}eyting ones:~{I}.
\newblock {\em Studia Logica}, 49:105--126, 1990.

\bibitem{Spin08a}
M.~Spinks and R.~Veroff.
\newblock Constructive logic with strong negation is a substructural logic.
  {I}.
\newblock {\em Studia Logica}, 88:325--348, 2008.

\bibitem{Spin08b}
M.~Spinks and R.~Veroff.
\newblock Constructive logic with strong negation is a substructural logic.
  {II}.
\newblock {\em Studia Logica}, 89:401--425, 2008.

\bibitem{Spin18}
M.~Spinks and R.~Veroff.
\newblock Paraconsistent constructive logic with strong negation as a
  contraction-free relevant logic.
\newblock In J.~Czelakowski, editor, {\em Don Pigozzi on Abstract Algebraic
  Logic, Universal Algebra, and Computer Science}, volume~16 of {\em
  Outstanding Contributions to Logic}, pages 323--379. Springer International
  Publishing AG, Cham, 2018.

\bibitem{Stone68}
A.~H. Stone.
\newblock On partitioning ordered sets into cofinal subsets.
\newblock {\em Mathematika}, 15:217--222, 1968.

\bibitem{Tanaka75}
S.~Tanaka.
\newblock On $\wedge$-commutative algebras.
\newblock {\em Math. Seminar Notes, Kobe Univ.}, 3:59--64, 1975.

\bibitem{Umadevi2015}
D.~Umadevi.
\newblock On the completion of rough sets system determined by arbitrary binary
  relations.
\newblock {\em Fundamenta Informaticae}, 137:413--424, 2015.

\bibitem{Vaka77}
D.~Vakarelov.
\newblock Notes on $\mathcal{N}$-lattices and constructive logic with strong
  negation.
\newblock {\em Studia Logica}, 36:109--125, 1977.

\bibitem{Vaka06}
D.~Vakarelov.
\newblock Non-classical negation in the works of {H}elena {R}asiowa and their
  impact on the theory of negation.
\newblock {\em Studia Logica}, 84:105--127, 2006.

\bibitem{Varl72}
J.~Varlet.
\newblock A regular variety of type $\zseq{2, 2, 1, 1, 0, 0}$.
\newblock {\em Algebra Universalis}, 2:218--223, 1972.

\bibitem{Wans95}
H.~Wansing.
\newblock Semantics-based nonmonotonic inference.
\newblock {\em Notre Dame Journal of Formal Logic}, 36(1):44--54, 1995.

\bibitem{Wern78}
H.~Werner.
\newblock {\em Discriminator-{A}lgebras}.
\newblock Number~6 in Studien zur Algebra und ihre Anwendungen.
  Akademie-Verlag, Berlin, 1978.

\bibitem{Wolski2006}
M.~Wolski.
\newblock Complete orders, categories and lattices of approximations.
\newblock {\em Fundamenta Informaticae}, 72:421--435, 2006.

\bibitem{Yao2012}
Y.~Yao and B.~Yao.
\newblock Covering based rough set approximations.
\newblock {\em Information Sciences}, 200:91--107, 2012.

\end{thebibliography}

\end{document}